\documentclass[12pt,a4paper,reqno]{amsart}
\usepackage[utf8]{inputenc}
\usepackage{upref}
\usepackage{amsmath}
\usepackage{amsthm}
\usepackage{comment}
\usepackage{amsfonts}
\usepackage{dsfont}
\usepackage{amssymb}
\usepackage{geometry}
\usepackage{esint}
\usepackage{bbm}
\usepackage[table]{xcolor}
\usepackage{graphicx}
\usepackage{tikz}
\usepackage{stmaryrd}
\usepackage{lscape}
\usepackage{mathrsfs}
\usepackage{amsrefs}

\usepackage{tikz}
\usepackage{fancyhdr}
\usepackage{epsfig}

\usepackage{breakcites}
\usepackage{mathtools}
\usepackage{mathrsfs}

\usepackage{tcolorbox}
\usepackage{color}
\usepackage[colorlinks=true, linkcolor=blue, citecolor=red]{hyperref}
\usepackage[]{appendix}
\numberwithin{equation}{section}

\newtheorem{theorem}{Theorem}[section]
\newtheorem{lemma}{Lemma}[section]
\newtheorem{prop}[lemma]{Proposition}

\newtheorem{corollary}[lemma]{Corollary}

\theoremstyle{definition}
\newtheorem{definition}[lemma]{Definition}

\theoremstyle{remark}
\newtheorem{remark}[lemma]{Remark}
\makeatletter \@addtoreset{figure}{section} \makeatother

\let\div\relax
\DeclareMathOperator{\div}{div}
\DeclareMathOperator{\curl}{curl}

\DeclareMathOperator{\supp}{supp}

\let\emptyset\varnothing

\newcommand{\s}{\mathbf{s}}
\renewcommand{\r}{\mathbf{r}}

\newcommand{\e}{\mathbf{e}}
\newcommand{\young}{\mathscr{Y}}

\newcommand{\cri}{\mathrm{cr}}
\newcommand{\cav}{\mathrm{cav}}

\def \loc {\mathrm{loc}}

\def \Z {\mathbf{Z}}
\def \bF {\mathbf{F}}

\newcommand{\Q}{\mathbf{Q}}

\def \i {\mathrm{i}}

\def \x {\mathbf{x}}

\def \epsi {\varepsilon}

\def \obstacle {O}

\def \bigo {\mathcal{O}}

\newcommand{\laplace}{\Delta}

\let\div\relax
\DeclareMathOperator{\div}{div}

\def \cf {\emph{cf.} }

\renewcommand{\d}{\, \mathrm{d}}
\newcommand{\de}{\mathrm{d}}

\newcommand{\bu}{\mathbf{u}}
\begin{document}

\title[Supersonic flow with cavitation]{The Morawetz problem for supersonic flow \\ with cavitation}

\author[G.-Q.~Chen]{Gui-Qiang G.~Chen}
\address[G.-Q.~Chen]{Mathematical Institute, University of Oxford, Radcliffe Observatory Quarter, Woodstock Road, Oxford, OX2 6GG, UK}
\email{Gui-Qiang.Chen@maths.ox.ac.uk}

\author[T.~P.~Giron]{Tristan P.~Giron}
\address[T.~P.~Giron]{Mathematical Institute, University of Oxford, Radcliffe Observatory Quarter, Woodstock Road, Oxford, OX2 6GG, UK}
\email{tristan.giron@gmail.com}

\author[S.~M.~Schulz]{Simon M.~Schulz}
\address[S.~M.~Schulz]
{Scuola Normale Superiore, Centro di Ricerca Matematica Ennio De Giorgi, P.zza dei Cavalieri, 3,  56126 Pisa, Italy}\email{simon.schulz@sns.it}



\begin{abstract}
We are concerned with the existence and compactness of entropy solutions
of the compressible Euler system for two-dimensional steady potential flow
around an obstacle for a polytropic gas with supersonic far-field velocity.
The existence problem, initially posed by Morawetz \cite{morawetz85} in 1985, 
has remained open since then.
In this paper, we establish the first complete existence theorem for the Morawetz problem
by developing a new entropy analysis, coupled with a vanishing viscosity method
and compensated compactness ideas.
The main challenge arises when the flow approaches cavitation, 
leading to a loss of strict hyperbolicity of the system
and a singularity of the entropy equation, particularly for the case of adiabatic exponent 
$\gamma=3$.
Our analysis provides a complete description of the entropy and entropy-flux pairs via
the Loewner--Morawetz relations, which, in turn, leads to the establishment of a
compensated compactness framework.
As direct applications of our entropy analysis and
the compensated compactness framework, we obtain the compactness of entropy solutions and the weak
continuity of the compressible Euler system in the supersonic regime.
\end{abstract}
\maketitle

\setcounter{tocdepth}{1}
\tableofcontents

\newpage

\section{Introduction}\label{sec:intro}
The Morawetz problem is concerned with the existence theory of global solutions of
the two-dimensional steady compressible potential flow equations:
\begin{equation}\label{eq:system in intro}
        \left\lbrace \begin{aligned}
       & \div(\varrho(\bu) \bu) = 0, \\
        & \curl \bu = 0,
    \end{aligned} \right.
\end{equation}
where $\bu:\mathbb{R}^2 \to \mathbb{R}^2$ is the velocity and $\varrho$
is the density, prescribed by the \emph{Bernoulli law}:
\begin{equation}\label{eq:bernoulli general gamma intro}
    \varrho(\mathbf{u}) = \big(1-\frac{\gamma-1}{2}|\bu|^2\big)^{\frac{1}{\gamma-1}}
\end{equation}
for $\gamma > 1$ (the adiabatic exponent).
In this paper, such a result is achieved by means of a vanishing viscosity method;
that is, such a solution $\bu$ of \eqref{eq:system in intro}--\eqref{eq:bernoulli general gamma intro}
is constructed as a limit of $\{\bu^\varepsilon\}_{\varepsilon>0}$ solving the approximate systems:
\begin{equation}\label{eq:potential approx}
    \left\lbrace \begin{aligned}
        &  \div(\varrho(\bu^\varepsilon) \bu^\varepsilon) = V_1^\varepsilon, \\
        &\curl \bu^\varepsilon = V_2^\varepsilon,
    \end{aligned} \right.
\end{equation}
where the viscous terms $V_1^\varepsilon$ and $V_2^\varepsilon$
are determined so that
they vanish in the sense of distributions in the limit $\varepsilon \to 0$
and some uniform estimates of the corresponding solutions $\bu^\varepsilon$ can be obtained.

System \eqref{eq:system in intro}--\eqref{eq:bernoulli general gamma intro} is
an archetypal system of conservation laws,
which is of fundamental importance to aerodynamics and fluid mechanics
({\it cf}. \cite{BersSonic,ChenFeldman,courantfried, dafermos-book, landau, morawetz58,morawetz85,morawetz95}).
Together with \eqref{eq:bernoulli general gamma intro}, $\,$\eqref{eq:system in intro}
forms a degenerate mixed-type system (\textit{cf.}~\cite{ChenFeldman}).
Define the \emph{local sonic speed} $c$ and the \emph{Mach number} $M$ to be
\begin{equation}\label{eq:local sound speed}
    c(\varrho) := \varrho^{\frac{\gamma-1}{2}}, \qquad M := \frac{|\bu|}{c},
\end{equation}
and introduce the velocity potential $\varphi$ so that $\bu = \nabla \varphi$.
Then the potential flow
system \eqref{eq:system in intro}--\eqref{eq:bernoulli general gamma intro} takes the form:
\begin{equation}\label{eq:pseudoanalytic}
    \div(\varrho(|\nabla \varphi|)\nabla\varphi)=0,
\end{equation}
which is of mixed elliptic-hyperbolic type:
elliptic when $|\nabla\varphi|<c$ (referred to as the \emph{subsonic region})
and hyperbolic when $|\nabla\varphi|>c$ (termed as \emph{supersonic region});
the curve $|\nabla\varphi|=c$ is identified as the \emph{sonic curve}.
This paper is devoted to the existence theory of entropy solutions
and the weak continuity of system \eqref{eq:system in intro}--\eqref{eq:bernoulli general gamma intro};
the relevance of this work stretches beyond the Morawetz problem,
as the inviscid Euler system serves as a paradigmatic
mixed-type system, sharing many common features with other mixed-type problems in fluid dynamics,
differential geometry, and related fields (see, for example,~\cites{ChenDafermosSlemrodWang, gq-isometric} and the references therein).

The early developments in the analysis of the potential flow system relied
on the methods from the theory of pseudo-analytic functions to solve the quasilinear second-order
equation \eqref{eq:pseudoanalytic} (\textit{cf.}~Berg \cite{berg} and Bers \cite{BersSonic}).
These methods mirrored the developments in the theory of minimal surfaces, where system \eqref{eq:system in intro}--\eqref{eq:bernoulli general gamma intro}
coincides with the minimal surface equations precisely in the case of  Chaplygin gas, $\gamma=-1$.
While this approach successfully addresses subsonic flows, corresponding to the elliptic equation \eqref{eq:pseudoanalytic}, 
it encounters limitations in describing the supersonic regime.
In fact, it was proved in \cite{morawetz58} that smooth transonic flows
along an obstacle are unstable, and shock waves are expected to arise generically.
To address this, an alternative approach was proposed Morawetz \cites{morawetz85,morawetz95},
employing a vanishing viscosity method as in \eqref{eq:potential approx}.
The key challenge lies in establishing the rigorous limit of problems \eqref{eq:potential approx} as $\epsilon\rightarrow 0$. 
This necessitates identifying uniform estimates of the approximate solutions and leveraging fine compactness properties of the system.
To circumvent the nonlinear nature of the potential flow system \eqref{eq:system in intro}--\eqref{eq:bernoulli general gamma intro},
Morawetz \cite{morawetz85} proposed utilizing compensated compactness, which had previously been applied successfully 
in addressing other nonlinear conservation laws (\cf. \cites{chen-lax-fried,dingchenluo1,diperna2,diperna,muratsmf,tartarnotes,tartar-nouvelle, tartar-consv-laws}).
The framework developed by Morawetz in \cites{morawetz85,morawetz95} 
requires the viscosity terms $V_1^\varepsilon$ and $V_2^\varepsilon$ to satisfy the two uniform bounds:
\begin{enumerate}
    \item[(i)] A uniform bound of the form: $|\bu^\varepsilon| \geq \underline{q}$ for some constant $\underline{q}>0$,
          \textit{i.e.}, \emph{the solutions stay uniformly away from stagnation};
    \item[(ii)] A uniform bound of the form: $\varrho(\bu^\varepsilon) \geq \underline{\varrho}$ for some constant $\underline{\varrho}>0$,
          \textit{i.e.}, \emph{the solutions stay uniformly away from cavitation}.
\end{enumerate}
However, these assumptions may not fully align with physical expectations, as it is generally anticipated 
that either cavitation or stagnation should occur. 
Additionally, apart from the condition away from stagnation, Morawetz's strategy proposed 
in \cites{morawetz85,morawetz95} relies on the utilization of a hodograph transformation. 
Yet, it remains unclear whether a rigorous justification for this transformation 
could be established in the presence of a vacuum.

Despite being formulated nearly forty years ago,
the Morawetz problem has remained open and largely unexplored since Morawetz's own work.
Notable exceptions include the contributions of Chen-Dafermos-Slemrod-Wang \cite{ChenDafermosSlemrodWang}
and Huang-Wang-Wang \cite{MorawetzFeimin} addressing the sonic-subsonic regime,
and Chen-Slemrod-Wang \cite{gq-transonic} for the transonic regime.
The former employs a compensated compactness approach in the sonic-subsonic regime, 
reflecting the earlier work of Bers \cite{BersSonic}.
The latter focuses on generating approximate problems admissible in Morawetz's framework, 
successfully excluding the possibility of vacuum occurrence when $\gamma\in (1, 3)$, 
without relying on Morawetz's problematic hodograph transformation. 
However, a complete existence proof remains elusive in the presence of stagnation, 
which cannot be \emph{a priori} excluded for this case.

The present work is concerned with system \eqref{eq:system in intro}--\eqref{eq:bernoulli general gamma intro}
for the case of a gas with $\gamma$-law
for the physically relevant adiabatic exponent $\gamma=3$,
which in fact corresponds to the case $\tilde{\gamma}=\frac{5}{3}$
in the isentropic Euler system \eqref{eq:euler intro} below; see \S \ref{subsec:connection euler}.
We establish a compactness framework that accommodates cavitation and
demonstrate that, for any incoming uniform supersonic flow,
stagnation does not occur in the fluid.
Consequently, we succeed to achieve the first complete existence
theorem for the Morawetz problem. The case $\gamma>3$ will be the subject of future investigations.

Compensated compactness was first introduced by Tartar and
Murat \cites{tartarnotes,tartar-nouvelle};
its first successful application to a system of hyperbolic conservation laws is
due to DiPerna  \cites{diperna2,diperna},
who established the global existence of an entropy solution of
the Cauchy problem for the one-dimensional isentropic Euler equations:
\begin{equation}\label{eq:euler intro}
    \left\lbrace \begin{aligned}
        & \varrho_t + (\varrho u)_x = 0, \\
        & (\varrho u)_t + \big(\varrho u^2 + \tilde{p}(\varrho) \big)_x = 0,
    \end{aligned}  \right.
\end{equation}
where $\tilde{p}(\varrho)= \kappa\varrho^{\tilde{\gamma}}$ with $\kappa>0$
for the cases: $\tilde{\gamma} = \frac{N+2}{N}$ for $N \geq 5$ integer.
For the interval $1<\gamma \leq \frac{5}{3}$,
the global existence problem was solved in
Ding-Chen-Luo  \cite{dingchenluo1} and Chen \cite{chen-lax-fried}.
The adiabatic exponent range $\gamma\ge 3$ was solved in \cite{lions-perthame-tadmor},
and the remaining case $\frac{5}{3}<\gamma<3$ was
closed in Lions-Perthame-Souganidis \cite{lions-perthame-souganidis}.

All these results rely on the cancellation properties of certain nonlinear functionals to obtain compactness
properties (see also \cite{chencompcomp}). More precisely, we consider a sequence of functions $\{\bu^\epsi\}_{\epsi>0}$,
bounded in $L^\infty$. Up to a subsequence, it converges weak-* to a limit $\bu$;
however, the weak-* convergence is not sufficient to pass to the limit
in nonlinear quantities, say $f(\bu^\epsi)$.
Under certain conditions on $f$ and sequence $\{\bu^\epsi\}_{\epsi>0}$,
compensated compactness ideas may be used to show that $\text{w*-}\lim f(\bu^\epsi)=f(\bu)$.
This method is typically applied to deal with the weak limits of approximate solutions
of systems of nonlinear conversation laws. In this context, sequence $\{\bu^\epsi\}_{\epsi>0}$
may be obtained for instance from a viscous system with vanishing viscosity.
It is not usually possible to show directly that the nonlinearities in the system satisfy
the hypotheses of compensated compactness,
even though this is possible sometimes (see \textit{e.g.} \cite{ChenDafermosSlemrodWang}).
Instead, with the help of Young measures, for any continuous function $f$,
the weak-* limit of $f(\bu^\epsi)$ may be characterized as
\begin{equation*}
    f(\bu^\epsi)\,\overset{*}{\rightharpoonup}\,
    \int f\d\young_\x,
\end{equation*}
where $\{\young_\x\}_{\x}$ is a family of probability measures on the phase-space.
If $f$ is shown to be affine on the support of the Young measure,
it then follows that $f(\bu^\varepsilon) \overset{*}{\rightharpoonup} f(\bu)$.
Moreover, if $f$ is assumed to be genuinely nonlinear, so that it is nowhere affine,
then it follows that the Young measure concentrates as a Dirac mass $\young_\x=\delta_{\bu(\x)}$
whence $\bu^\varepsilon \to \bu$ almost everywhere.
Then the convergence is sufficiently strong to pass to the limit in the equations.

Hyperbolic systems of nonlinear conservation laws are endowed with
families of entropy pairs ({\it i.e.}, the pairs of entropy and entropy-flux).
While the original nonlinearity does not typically satisfy the assumptions of compensated compactness,
the key strategy is that compensated compactness arguments can be used with the entropy pairs.
Realizing the above strategy therefore requires a fine understanding of the entropy structure
of the underlying system.

In this paper, we provide a detailed analysis of the entropy structure
for the potential flow system by generating two linearly independent families of entropy pairs.
These families of entropy pairs are subsequently utilized to establish our compensated compactness framework.
The description of the entropy pair families involves a delicate analysis of a Tricomi--Keldysh
equation of the form:
\begin{equation*}
    H_{\nu\nu} - \frac{M^2-1}{\varrho^2}H_{\vartheta\vartheta} = 0,
\end{equation*}
which is of mixed-type: hyperbolic for $M>1$ and elliptic when $M<1$ (exhibiting
Tricomi degeneracy at $M=1$).
Additionally, the equation is singular at cavitation $\varrho=0$, displaying Keldysh degeneracy. 
This analysis constitutes the primary focus of the present work.

\smallskip
The rest of the paper is organized as follows.
In \S \ref{sec:main results}, we introduce the key notions required for our analysis,
discuss formal links between the problem at hand and the isentropic Euler equations,
 provide the precise statements of our main results, and outline our strategy of proof.
 In \S \ref{sec:Section hyperbolicity and genuine nl}, we rewrite the equations
 as a genuinely nonlinear hyperbolic system and outline our method of entropy production.
 In \S \ref{sec:reg kernel} and \S \ref{sec:sing kernel}, we prove the existence and cancellation properties
 for the regular and singular entropy generator kernels, respectively.
Based on the entropy analysis in \S \ref{sec:Section hyperbolicity and genuine nl}--\S \ref{sec:sing kernel},
we first establish a compensated compactness framework in \S \ref{sec:comp comp framework and red}
by an argument of the Young measure reduction for our problem,
and we then proceed in \S\ref{sec:weak continuity}
 to prove the compactness of entropy solutions and the weak continuity of
 the system.
 In \S \ref{sec:viscous}, we present an example to show how a sequence of approximate solutions
 can be constructed to satisfy the conditions of the
compensated compactness framework established
in \S \ref{sec:comp comp framework and red} in any fixed bounded domain $\mathscr{D}$.
Finally, in \S \ref{sec:entropy sol}, we
first combine the compensated compactness framework in  \S \ref{sec:comp comp framework and red}
with the approximate solutions constructed in \S \ref{sec:viscous}
to show the existence of entropy solutions on the bounded domains, and then employ the compactness
theorem in \S\ref{sec:weak continuity} to establish
the existence of entropy solutions on the whole domain $\Omega$
by considering the exhaustion limit
in the half-plane.
Appendix \ref{appendix:Glambda} contains technical results
pertaining to the asymptotic descriptions of the generator kernels that
are the fundamental solutions of the Tricomi--Keldysh equation.

\smallskip
\noindent
We end this introductory section by highlighting some basic notation used throughout the rest of the paper.
For a vector field $\mathbf{w} = (w_1,w_2)$, we write $$\mathbf{w}^\perp = (-w_2,w_1),
\quad \curl \mathbf{w} = \partial_x w_2 -\partial_y w_1 = -\div \mathbf{w}^\perp,$$
and similarly $\nabla^\perp \varphi = (\nabla \varphi)^\perp$ for any scalar function $\varphi$.
Denote $\langle \cdot, \cdot \rangle$ as the duality bracket between test functions and distributions;
$\mathcal{S}'(\mathbb{R})$ and $\mathcal{E}'(\mathbb{R})$ are  the spaces of tempered distributions and
compactly supported distributions, respectively. The Fourier transform of a tempered distribution is denoted by $\widehat{\cdot}$,
or occasionally $\mathscr{F}$ when confusion might arise,
and is always taken with respect to variable $s$.
The test functions are denoted by the letters $\varphi$ or $\psi$.
The upper half-plane is denoted by $\mathbb{H}$.
For any ${D} \subset \mathbb{R}^2$, the one-dimensional Hausdorff measure on
boundary $\partial{D}$ is denoted by $\mathcal{H}$.
In \S \ref{sec:viscous}, the unit vector for the velocity field
is $\e(\vartheta)=(\cos\vartheta,\sin\vartheta)$,
and $\mathbf{n}$ denotes the inward unit normal to the boundary,
pointing into the region filled with fluid.
A property holds in the sense of distributions in $D$ if it holds in duality
with $C^\infty_{\rm c}({D})$.

\section{The Morawetz Problem and Main Theorems}\label{sec:main results}
In this section, we first set up the problem, introduce the key notions required for our analysis,
and then provide the precise statements of our main results and some outline
of our strategies of proof. Finally,
we discuss formal links between the Morawetz problem and the isentropic Euler equations.

\subsection{Problem setup}\label{sec:setup}
We are concerned with the existence of an entropy solution $\bu=(u,v)$ of the nonlinear
system \eqref{eq:system in intro}--\eqref{eq:bernoulli general gamma intro} for the case $\gamma=3$, \textit{i.e.},
\begin{equation*}
    \left\lbrace \begin{aligned}
        &\div(\varrho(\bu) \bu) = 0, \\
        & \curl \bu = 0,
    \end{aligned} \right.
\end{equation*}
where, denoting $q:= |\bu|$ to be the \emph{scalar speed}, the density $\varrho$ is prescribed by the Bernoulli law:
$$
\varrho(\bu) = \sqrt{1- q^2}.
$$

The problem is posed in the upper half-plane $\mathbb{H}$ in the exterior of an obstacle $O$,
a closed connected bounded subset $\Omega:= \mathbb{H} \setminus O$ as shown in Figure 2.1.
The obstacle boundary $\partial O$ is endowed with a unit normal
vector $\mathbf{n}$ pointing into $\Omega$ which is filled with the fluid.
A uniform supersonic flow $\bu_\infty=(q_\infty, 0)$ with constant speed $q_\infty>0$
is given at $x=-\infty$, and a Neumann-type condition is imposed on the obstacle boundary
$\partial O$; this is discussed in detail later (see also \S \ref{sec:viscous}).

\vspace{5pt}
\begin{figure}[ht]
\begin{center}
\begin{tikzpicture}
    \draw (0, 0) node[inner sep=0] {\includegraphics[width=6.5cm]{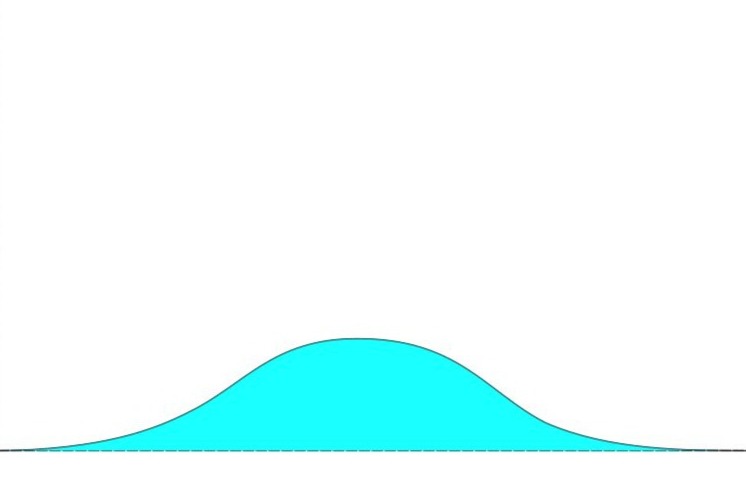}};
    \draw [-stealth](-4.8,1) -- (-3.8,1);
    \draw [-stealth](-4.8,0) -- (-3.8,0);
     \draw (-3.2,-0.7) node {$\bu_\infty$};
    \draw [-stealth](-4.8,-1) -- (-3.8,-1);
    \draw (0,-1.3) node {$\obstacle$};
    \draw (1.8,0.6) node {$\Omega$};
    \draw (-5,-1.74) -- (4,-1.74);
\end{tikzpicture} \hspace{1cm}
\end{center}
\caption{\, The Morawetz problem for two-dimensional steady supersonic flow past an obstacle $O$.}\label{fig1.1}
\end{figure}
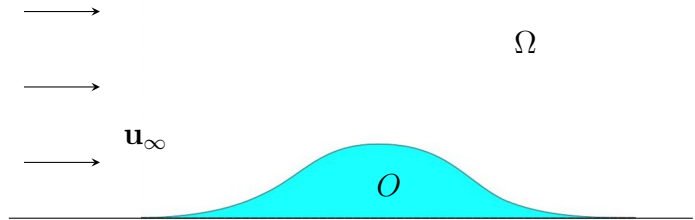

Introduce the \emph{angle function} $\vartheta$ associated to the velocity field $\bu$ such that
\begin{equation}\label{2.1a}
    \bu = (q \cos\vartheta, q \sin\vartheta),
\end{equation}
and define the \emph{renormalized density} $\nu$ by the formula:
\begin{equation}\label{eq:nu def}
    \nu(\varrho) := \int_0^\varrho \frac{\tau^{2}}{1-\tau^{2}} \d \tau.
\end{equation}
Then we construct entropy pairs $\mathbf{Q}=(Q_1,Q_2)$ of
system \eqref{eq:system in intro}--\eqref{eq:bernoulli general gamma intro},
which are functions of the phase space variables $\bu$,
via the \emph{Loewner--Morawetz relations}:
\begin{align}
\begin{split}\label{eq:lowener mor}
   \left( \begin{matrix}
       Q_1 \\ Q_2
   \end{matrix} \right) = \left( \begin{matrix}
       \varrho q \cos \vartheta & -q \sin\vartheta \\
       \varrho q \sin\vartheta & q \cos\vartheta
   \end{matrix} \right) \left(  \begin{matrix}
       H_\nu \\ H_\vartheta
   \end{matrix}\right),
    \end{split}
\end{align}
where the \emph{entropy generators} $H$ are determined by a linear singular mixed-type equation ---
\emph{the Tricomi--Kelydish equation}:
\begin{equation}\label{eq:ent gen intro}
    H_{\nu\nu} - \frac{M^2-1}{\varrho^2}H_{\vartheta\vartheta} = 0;
\end{equation}
we refer the reader to \S \ref{sec:entropic struc intro} for more details on this construction.
We emphasize that the function pairs $\Q$, which are technically
computed in the $(\varrho,\vartheta)$-coordinates
as above, are entropy pairs for the conservative
system \eqref{eq:system in intro}--\eqref{eq:bernoulli general gamma intro}
so that $\div_\x \Q(\bu)$ is used throughout the paper.

A particular solution of \eqref{eq:ent gen intro} is provided by the \emph{special entropy generator}:
\begin{equation}\label{eq:special entropy generator}
H^*(\nu,\vartheta) = \frac{\vartheta^2}{2} - \frac{\nu}{\varrho(\bar{\nu})}
+ \int_{\bar{\nu}}^\nu \int_{\bar{\nu}}^\tau \frac{M(\nu')^2 - 1}{\varrho(\nu')^2} \d \nu' \d \tau,
\end{equation}
for suitably chosen constant $\bar{\nu}>0$, provided that
the integral is well-defined (\textit{cf.}~\cite{thesis}).
We define the \emph{special entropy pair} $\Q_*$ to be the entropy pair generated by $H^*$
via the Loewner--Morawetz relations.
When restricted to the supersonic region, this special generator constitutes
a uniformly convex solution of \eqref{eq:ent gen intro};
however, it is not periodic with respect to $\vartheta$.

The starting point of the compensated compactness method is to find
a uniformly convex entropy of the system,
with which some uniform gradient estimates are obtained.
It is not clear \emph{a priori} that \eqref{eq:ent gen intro} possesses uniformly convex solutions
defined up to and including the subsonic region. For this reason,
as well as the invariant regions associated with our setting (\textit{cf.}~\S \ref{sec:viscous}),
in this paper, we focus on the supersonic regime.
Note that the supersonic condition $M>1$  is equivalent to $q > q_\cri$, where the \emph{sonic speed},
or equivalently \emph{critical speed}, is given by
$$q_{\cri}:=\sqrt{\frac{2}{\gamma+1}}; $$
this constraint is equivalent to $\varrho < \varrho_{\cri}:=\varrho(q_{\cri})$.
We also define  $\nu_\cri := \nu(\varrho_\cri)$ and consider equation \eqref{eq:ent gen intro} only
for $\nu \in (0,\nu_*]$ for $\nu_* < \nu_\cri$,
where this equation is a singular linear wave equation.
Much of the current paper is concerned with generating sufficiently many solutions
of \eqref{eq:ent gen intro}
to provide the associated entropy pairs as many as possible for the compensated compactness method to be applied.
We note that the vacuum is achieved at the \emph{cavitation speed}:
\begin{equation*}
    q_\cav := \sqrt{\frac{2}{\gamma-1}},
\end{equation*}
where, by virtue of the Bernoulli relation, the density is null.

As already mentioned, equation \eqref{eq:ent gen intro} is a singular
linear second-order wave equation
in the domain of interest, so that two linearly independent fundamental solutions
are expected to generate all possible solutions, which are denoted by $H^\r$ and $H^\s$.
Furthermore, these fundamental solutions should satisfy the Huygens principle.
To quantify this finite-speed of propagation, we rewrite equation \eqref{eq:ent gen intro}
in the familiar form:
\begin{equation}\label{eq:ent gen intro with k'}
    H_{\nu\nu} - k'(\nu)^2 H_{\vartheta\vartheta} = 0.
\end{equation}
The Bernoulli relation \eqref{eq:bernoulli general gamma intro} implies that $\nu$ satisfies
\begin{equation}\label{eq:nu prime eqn}
    \left\lbrace\begin{aligned}
        & \nu'(\varrho) =\frac{1}{M^{2}(\rho)} \qquad \text{in } (0,\varrho_\cri), \\
        & \nu(0) = 0,
    \end{aligned}\right.
\end{equation}
whence, by defining the \emph{characteristic speed} $k$ via
\begin{equation}\label{eq:k def}
    k(\nu(\varrho)) := \int_0^\varrho \frac{\sqrt{M^2(\varrho')-1}}{\varrho'}
    \frac{1}{M^{2}(\varrho')} \d \varrho' \qquad \text{for all } \varrho \in (0,\varrho_\cri),
\end{equation}
we see that
\begin{equation}\label{eq:what is k' squared}
    \left\lbrace \begin{aligned}
   & k'(\nu)^2 = \frac{M^2 - 1}{\varrho^2} \qquad \text{in } (0,\nu_\cri), \\
   &k(0) = 0.
    \end{aligned}\right.
\end{equation}
It follows that the entropy generator equation \eqref{eq:ent gen intro}
can be recast in form \eqref{eq:ent gen intro with k'} in the supersonic regime.
Moreover, by defining the \emph{light cone}:
$$
\mathcal{K} := \{(\nu,s)\, :\, |s| \leq k(\nu)\},
$$
we expect that $\supp H^\r \cup \supp H^\s \subset \mathcal{K}$.

With these notions in hand, we are ready to state our main theorems.

\subsection{Main theorems}\label{sec:main theorems}
Our first main theorem is concerned with the existence of entropy generator kernels in the hyperbolic region,
up to and including cavitation (also see Propositions \ref{prop:reg kernel exist} and \ref{prop:sing kernel exist}).

\begin{theorem}\label{thm:kernels}
There exist distributional solutions $H^\r$ and $H^\s$ with
supported in $\mathcal{K}$
of the initial value problems in $(0,\nu_{\cri}) \times \mathbb{R}${\rm :}
    \begin{equation*}
    \left\lbrace\begin{aligned}
         &H^\r_{\nu\nu} - \frac{M^2-1}{\varrho^2} H^\r_{ss} = 0,\\
        &H^\r|_{\nu = 0} = 0, \\
        &H^\r_\nu|_{\nu=0} = \delta_0,
    \end{aligned}\right. \qquad \left\lbrace\begin{aligned}
         &H^\s_{\nu\nu} - \frac{M^2-1}{\varrho^2} H^\s_{ss} = 0,\\
        &H^\s|_{\nu = 0} = \delta_0, \\
        &\varrho H^\s_\nu|_{\nu=0} = -\delta_0'',
    \end{aligned} \right.
    \end{equation*}
for which  $H^\r$ and $H^\s$
are respectively called the \emph{regular kernel} and the \emph{singular kernel}
of \eqref{eq:ent gen intro}.
\end{theorem}

\begin{remark}[Generating Entropy Pairs]\label{rmk:generation}
We say that a generator $H$ is \emph{generated by kernels $H^\r$ and $H^\s$}
if there exist smooth functions $\varphi_1$ and $\varphi_2$ such that
\begin{equation*}
    H = H^\r * \varphi_1 + H^\s * \varphi_2,
\end{equation*}
where the convolution is taken only for the second variable $s$.

If an entropy pair $\Q$ is determined from the Loewner--Morawetz relations
via a generator $H$ generated by kernels $H^\r$ and $H^\s$, we say
that \emph{$\Q$ is generated by kernels $H^\r$ and $H^\s$}. A suitable generator $H$ can also be determined
from $\Q$ by inverting the matrix in \eqref{eq:lowener mor}.
\end{remark}

Our goal is then to employ Theorem \ref{thm:kernels} to produce entropy pairs
that are endowed with additional compactness properties.
More precisely, we require the approximate solution
sequence $\{\bu^\varepsilon\}_{\varepsilon>0}$ to satisfy
that the sequence of \emph{entropy dissipation measures}
$\{\div_\x \mathbf{Q}^\varepsilon\}_{\varepsilon>0}$ is pre-compact in $H^{-1}_\loc$.
This requirement is critical in applying the compensated compactness method.

To state our main existence result, we must first introduce our notion of solutions.
This is encapsulated in the following definition.

\begin{definition}[Entropy Solutions]\label{def:ent sol}
We say that $\bu \in L^\infty$ is an \emph{entropy solution}
of \eqref{eq:system in intro}--\eqref{eq:bernoulli general gamma intro}
if the following conditions are satisfied:
\begin{enumerate}
\item[(i)] $\bu$ is a solution of system \eqref{eq:system in intro}--\eqref{eq:bernoulli general gamma intro}
in the sense of distributions;
\smallskip
\item[(ii)] Let $\Q \in \{\mathbf{Q}_*,\Q_H\}$, where $\Q_*$ is the special entropy pair
generated by $H^*$, and $\Q_H$ is any entropy pair determined from the Loewner--Morawetz
relations via a generator $H$ generated by kernels $H^\r$ and $H^\s$,
satisfying the convexity conditions{\rm :}
\begin{equation}\label{eq:generator conditions ent ineq}
0 \leq  H_{\vartheta \vartheta} -\varrho H_{\nu \vartheta\vartheta},
\qquad | \varrho H_{\nu \vartheta} + H_{\vartheta\vartheta\vartheta}| \leq \varrho (  H_{\vartheta \vartheta} - \varrho H_{\nu \vartheta\vartheta}).
 \end{equation}
Then
\begin{equation}\label{eq:ent ineq distr}
\div_\x \mathbf{Q}(\mathbf{u}) \geq 0 \qquad\,\, \mbox{in the sense of distributions}.
\end{equation}
\end{enumerate}
\end{definition}

We now state the main existence theorem.

\begin{theorem}[Global Existence of Entropy Solutions]\label{thm:main existence theorem}
Given the incoming uniform supersonic flow $\bu_\infty = (q_\infty,0)$ at $x=-\infty$ with constant $q_\infty \in (q_\cri,q_\cav)$,
then there exists a globally-defined entropy solution $\bu \in L^\infty$ of the Morawetz problem
for system \eqref{eq:system in intro} with $\gamma=3$ such that
    \begin{equation*}
    0 \leq \varrho(\bu(\x)) \leq \varrho_*, \quad q_{\cri} < q_* \leq |\bu(\x)| \leq q_{\cav}
    \qquad \text{a.e.
    } \x\in\Omega,
\end{equation*}
for some $\varrho_*>0$ and $q_*>0$.
\end{theorem}

We show in \S \ref{sec:entropy sol} that the entropy solution obtained
in Theorem \ref{thm:main existence theorem}
satisfies the boundary condition around
the obstacle boundary $\partial O$: $\bu \cdot\mathbf{n}|_{\partial O} \geq 0$ in the sense
of weak normal traces (\textit{cf.}~Chen-Frid \cite{ChenFrid}). Obtaining the more precise boundary
condition $\bu\cdot\mathbf{n}|_{\partial O}=0$ may be possible
for specific domains, which are beyond the scope of this paper.

As mentioned before, the existence result is proved by combining a vanishing viscosity
method with the following compensated compactness framework:

\begin{theorem}[Compensated Compactness Framework]\label{thm:compensated compactness framework intro}
Let $\{\bu^\epsi
\}_{\varepsilon>0} \subset L^\infty$
be a sequence of functions satisfying$:$
\begin{enumerate}
    \item[(i)] There exists $\varrho_*\in (0, \varrho_{\rm cr})$, independent of $\epsi>0$, such that
\begin{equation*}
  0 \le \varrho(|\bu^\epsi(\x)|) \leq \varrho_* < \varrho_{\cri} \qquad \mbox{a.e.~$\x
  $};
\end{equation*}
or equivalently, there exists $q^*\in (q_{\cri}, q_{\cav}]$, independent of $\varepsilon$, such that
\begin{equation}\label{5.1a}
    q_{\cri} < q_* \leq |\bu^\varepsilon(\x)| \leq q_{\cav}
    \qquad \mbox{a.e.~$\x.$}
\end{equation}
\item[(ii)] There exist $\vartheta_*,\vartheta^*\in\mathbb{R}$,
independent of $\epsi>0$,
such that
\begin{equation}\label{5.2a}
    \vartheta_* \leq \vartheta^\varepsilon(\x) \leq \vartheta^*
    \qquad \mbox{a.e.~$\x,$}
\end{equation}
where $\vartheta^\varepsilon$ is the angle function associated to the velocity field $\bu^\varepsilon.$
\item[(iii)] For any entropy pair $\Q$ via
the Loewner--Morawetz relations
from a generator $H$ generated by kernels $H^\r$ and $H^\s$,
\begin{equation}\label{5.3a}
\{\div_\mathbf{x} \mathbf{Q}(\bu^\epsi)\}_{\varepsilon>0}
\quad\mbox{is pre-compact in $H^{-1}_\loc$}.
\end{equation}
\end{enumerate}
Then there exist both a subsequence $($still denoted$)$ $\{\bu^\epsi\}_{\varepsilon>0}$ and
a function $\bu\in L^\infty$ such that
\begin{equation}\label{eqn-convergence-approx}
\bu^\epsi \,\to \, \bu
\qquad \mbox{
{\it a.e.}~and strongly in $L^p_\loc$ for all $p\in [1,\infty)$}.
\end{equation}
\end{theorem}

As a direct corollary, we obtain the compactness of entropy solutions
of the Morawetz problem.

\begin{theorem}[Compactness and Weak Continuity]\label{big-thm-3}
Let $\{\bu^\epsi\}_{\epsi>0}\subset L^\infty$ be entropy solutions of
system \eqref{eq:system in intro}--\eqref{eq:bernoulli general gamma intro} which admit the following bounds$:$
\begin{itemize}
\item[(i)]  There exists $\varrho_*\in (0, \varrho_{\rm cr})$, independent of $\epsi>0$, such that
\begin{equation*}
  0 \le \varrho(\bu^\epsi(\x)) \leq \varrho_* < \varrho_{\cri} \qquad \mbox{
  a.e.~$\x$}.
\end{equation*}

\item[(ii)] There exist $q^*\in (q_{\cri}, q_{\cav}]$ and $\vartheta_*,\vartheta^* \in \mathbb{R}$,
independent of $\epsi>0$
such that
\begin{equation*}
    q_{\cri} < q_* \leq |\bu^\epsi(\x)| \leq q_{\cav}
    , \quad \vartheta_* \leq \vartheta^\varepsilon(\x) \leq \vartheta^*
    \qquad \mbox{
    a.e.~$\x$},
\end{equation*}
where $\vartheta^\varepsilon$ is the angle function associated to the velocity field $\bu^\varepsilon.$
\end{itemize}
Then there exist both a
subsequence $($still denoted$)$ $\{\bu^\epsi\}_{\varepsilon>0}$
and a function $\bu \in L^\infty$ such that
\begin{itemize}
\item[(i)]  $\bu^\epsi\,\to\, \bu$ a.e.~and strongly in $L^p_\loc$ for all $p \in [1,\infty).$

\smallskip
\item[(ii)]  The limit function $\bu$
is also an entropy solution of system \eqref{eq:system in intro}--\eqref{eq:bernoulli general gamma intro}.
\end{itemize}
\end{theorem}

\smallskip
\subsection{Connections with the isentropic Euler equations}\label{subsec:connection euler}
As mentioned earlier, the Morawetz problem is deeply connected to the isentropic Euler equations \eqref{eq:euler intro}.
In particular, recall that an entropy pair $(E, F)$ for system \eqref{eq:euler intro} is a pair of functions
of $(\varrho,u)$ such that
\begin{equation*}
    E_t(\varrho,u) + F_x(\varrho,u) =0
\end{equation*}
for any smooth solution $(\varrho, u)$ of system \eqref{eq:euler intro}.
When the pressure law is given by a polytropic perfect gas, \textit{i.e.},
\begin{equation}
    p(\varrho)= \frac{(\tilde\gamma-1)^2}{4\tilde\gamma}\varrho^{\tilde{\gamma}},
\end{equation}
the entropy $E$ and entropy-flux $F$ may be generated by convolving test functions $\phi$ with kernels,
denoted for convenience $\chi(\varrho,u,s)$ and $\sigma(\varrho,u,s)$, respectively,
so that $(E(\varrho,u),F(\varrho,u))=(\chi(\varrho,u)*_s\phi,\,\sigma(\varrho,u)*_s\phi)$.
The entropy-flux generator itself is given
by $\sigma(\varrho,u,s)=(u+\frac{\tilde\gamma-1}{2}(s-u))\chi(\varrho,u,s)$,
while the entropy generator $\chi$ is determined by the Keldysh equation:
\begin{equation}
    \chi_{\varrho\varrho} - \frac{(\tilde\gamma-1)^2}{4}\varrho^{\tilde{\gamma}-3}\chi_{uu}=0.
\end{equation}
This equation admits two linearly independent fundamental solutions,
encoding the choices of admissible initial conditions.

More generally, when considering general pressure laws so that
$p(\varrho)$ is only \emph{asymptotically} polytropic, a similar Keldysh equation:
\begin{equation}
    \chi_{\varrho\varrho} - \tilde{k}'(\varrho)^2\chi_{uu}=0
\end{equation}
holds for the entropy generator $\chi(\varrho,u,s)$.
A general strategy for solving the equations of this form
was developed in \cite{ChenLeFloch}, and the first step in our approach is concerned with adapting this method.

Returning to the Morawetz problem, notice from \eqref{eq:nu def} that
$\nu(\varrho) = \tanh^{-1}(\varrho) - \varrho$ as a closed-form;
however, there is no hope in inverting this formula explicitly.
It follows that the only information available is an asymptotic description
of the characteristic speed $k(\nu)$ in the vicinity of the vacuum, a rough statement of which is
\begin{equation}\label{eq:asymp k intro rough}
k(\nu) = c_\sharp \nu^{\frac{1}{3}} + \bigo(\nu), \qquad k'(\nu) = \frac{1}{3}c_\sharp \nu^{-\frac{2}{3}} + \bigo(1)
\end{equation}
for some positive constant $c_\sharp$ (see Proposition \ref{prop:k expand}).

The leading term of the entropy generator equation \eqref{eq:ent gen intro with k'}
coincides with the coefficient of the entropy equation for the isentropic Euler equations
\eqref{eq:euler intro}
with a suitable $\gamma$-law. More precisely,
\begin{equation}
    k'^2_\text{Morawetz}(\nu
    )\sim (c_\sharp\theta/\gamma)^2 \nu^{-1-\frac{1}{\gamma}},
\end{equation}
which coincides with the coefficient of the entropy equation
for
system \eqref{eq:euler intro} when $\tilde\gamma=2-\frac{1}{\gamma}$.
We highlight that the case: $\gamma=3$ for the Morawetz problem corresponds to
the case: $\tilde{\gamma}=\frac{5}{3}$
for the isentropic Euler equations \eqref{eq:euler intro}.

We emphasize that our setting is different from that in \cite{ChenLeFloch}
for the isentropic Euler equations with general pressure laws.
While it is possible to construct an ad-hoc general pressure
law $\tilde{p}(\nu)$ in the sense of \cite{ChenLeFloch},
matching our $k(\nu)$ specified above, this pressure law fails to verify the genuine nonlinearity
requirement in \cite{ChenLeFloch}.
More precisely, coefficient $k$ is defined from a given pressure
law
$\tilde{p}(\varrho) := \kappa \varrho^{\tilde{\gamma}}(1+P(\varrho))$,
for $\kappa >0$ and $\tilde{\gamma}\in (1,3)$ with $P$ satisfying
$|P^{(j)}(\varrho)| \leq C\varrho^{\tilde{\gamma}-1-j}$ for $j=1,\dots,4$,
by setting $k'(\varrho) = \frac{\sqrt{\tilde{p}'(\varrho)}}{\varrho}$.
One is indeed able to reconstruct a pressure function $\tilde{p}(\nu)$ of this form
from our coefficient $k(\nu)$ defined in \eqref{eq:k def} by defining
\begin{equation*}
   \tilde{p}(\nu) := \int_0^\nu \tau^2 k'(\tau)^2 \d \tau = \frac{1}{9}c_\sharp^2 \int_0^\nu \big( \tau^{\frac{2}{3}}
   + \bigo(\tau^{\frac{4}{3}}) \big) \d \tau = \frac{1}{6}c_\sharp^2 \nu^{\frac{5}{3}}\big( 1 + \bigo(\nu^{\frac{2}{3}}) \big),
\end{equation*}
so that, by setting $\tilde{\gamma}=\frac{5}{3}$,
the perturbation $P(\nu) := 6 c_\sharp^{-2} \nu^{-\frac{5}{3}}p(\nu) - 1  = \bigo(\nu^{\frac{2}{3}})$
satisfies the boundedness assumptions in \cite{ChenLeFloch}.
However, the pressure functions admissible in the setting in \cite{ChenLeFloch} must also satisfy
the assumptions of strict hyperbolicity and genuine nonlinearity, \textit{i.e.},
$\tilde{p}'(\nu)>0$ and $2\tilde{p}'(\nu) + \nu \tilde{p}''(\nu)>0$ for $\nu>0$.
This second requirement is not true in our case since, in the vicinity of the vacuum,
\begin{equation*}
  \begin{aligned}  2\tilde{p}'(\nu) + \nu \tilde{p}''(\nu) = 4\nu^2 k'(\nu)^2
  + 2\nu^3 k''(\nu)k'(\nu) \sim \frac{8}{27}c_\sharp^2 \nu^{\frac{2}{3}}  \quad \text{as } \nu \to 0,
  \end{aligned}
\end{equation*}
which is positive; however,
along the sonic curve where $M=1$
(\textit{cf.}~Remark \ref{remark:k smooth on open}),
\begin{equation*}
   \begin{aligned}
   2\tilde{p}'(\nu_\cri) + \nu_\cri \tilde{p}''(\nu_\cri)
   = \Big(4\nu^2 \frac{(M^2-1)}{\varrho^2}
   - 2 \nu^3 \frac{M^2}{\varrho^3}\big( \frac{1}{\varrho^2} + (M^2-1) \big)\Big)\Big|_{\nu=\nu_\cri}
   \!\!\!\!=\! -2\frac{\nu_\cri^3}{\varrho_\cri^5},
   \end{aligned}
\end{equation*}
which is negative.
In turn, one cannot directly apply
\cite[Theorem 2.1]{ChenLeFloch},
concerned with the existence of the weak entropy kernel for the isentropic Euler equations,
to obtain the existence of the regular kernel $H^\r$.
Meanwhile, the existence of the strong entropy kernel has never been treated for the pressure
laws different from exactly
polytropic gas, so the analysis of the singular kernel $H^\s$ is different from what has been
covered thus far in the literature.

\section{Hyperbolicity and Entropy Structure}\label{sec:Section hyperbolicity and genuine nl}

This section aims to formulate system \eqref{eq:system in intro}--\eqref{eq:bernoulli general gamma intro}
as a genuinely nonlinear hyperbolic system in \S \ref{sec:hyp gen non}
and to contextualize the Loewner--Morawetz relations for entropy pairs
given in \eqref{eq:lowener mor}--\eqref{eq:ent gen intro} in \S \ref{sec:entropic struc intro}.
The main results of this section are the following:

\begin{prop}[Hyperbolicity and Genuine Nonlinearity]\label{lem:hyperbolic and gn}
For $\gamma=3$, system \eqref{eq:system in intro}--\eqref{eq:bernoulli general gamma intro}
is strictly hyperbolic in region $\{q_{\cri} < |\bu| < q_{\cav}\}$
and is genuinely nonlinear in region $\{q_{\cri} \leq |\bu| \leq q_{\cav}\}$.
\end{prop}

\begin{prop}[Riemann Invariants]\label{lem:riemann invariants}
For $\gamma=3$, system \eqref{eq:system in intro}--\eqref{eq:bernoulli general gamma intro}
is endowed with the Riemann invariants{\rm :}
    \begin{equation}\label{eq:Wpm expression 1}W_\pm = \vartheta \pm k,\end{equation}
    and the mapping{\rm :} $\bu \mapsto (W_-(\bu),W_+(\bu))$ is bijective on $\{q_{\cri} \leq |\bu| \leq q_{\cav}\}$.
\end{prop}

\subsection{Hyperbolicity and genuine nonlinearity}\label{sec:hyp gen non}
We first show Proposition \ref{lem:hyperbolic and gn}.

Following \cite{gq-transonic}, we first recast
system \eqref{eq:system in intro}--\eqref{eq:bernoulli general gamma intro} in polar coordinates.
By recalling the angle function $\vartheta$ associated with the velocity field $\bu$
as introduced in \eqref{2.1a},
system \eqref{eq:system in intro}--\eqref{eq:bernoulli general gamma intro}
can be formally rewritten
as
\begin{equation}\label{eq:system polar sec 2 before strict hyp and genuine nonlinear}
    \left( \begin{matrix}
        q \\
        \vartheta
    \end{matrix} \right)_x + \mathbf{M}  \left( \begin{matrix}
        q \\
        \vartheta
    \end{matrix} \right)_y = 0,
\end{equation}
where matrix $\mathbf{M}$ is uniquely characterized by its normalized right-eigenvectors $\mathbf{R}_\pm$
and corresponding eigenvalues $\Lambda_\pm$, given by the explicit formulas:
\begin{equation}\label{eq:eigenvalues and eigenvec q theta}
   \mathbf{R}_\pm = \frac{qc}{\sqrt{q^2 c^2 + q^2-c^2}}
   \left( \begin{matrix}
        1 \\
        \mp \frac{\sqrt{q^2-c^2}}{qc}
    \end{matrix} \right), \quad
    \Lambda_\pm = -\frac{\cos\vartheta \pm \frac{\sqrt{q^2-c^2}}{c}\sin\vartheta}
     {\sin\vartheta \mp \frac{\sqrt{q^2-c^2}}{c}\cos\vartheta},
\end{equation}
with the local sonic speed $c$ defined in \eqref{eq:local sound speed}: $c(\varrho)=\varrho$
for $\gamma=3$.
As per \cite[\S 3]{gq-transonic}, further linear transformations may be applied to formally
rewrite \eqref{eq:system polar sec 2 before strict hyp and genuine nonlinear} as a nonlinear system
for the unknown vector field $(\varrho,\vartheta)$,
the setting in which we compute entropy pairs (\textit{cf.}~\S \ref{sec:entropic struc intro}).

System \eqref{eq:system polar sec 2 before strict hyp and genuine nonlinear},
as well as the proposed variants in \cite{gq-transonic}, is \emph{a priori} not equivalent to
the conservative system \eqref{eq:system in intro}--\eqref{eq:bernoulli general gamma intro};
indeed we do not expect the solutions to be regular enough to apply the chain rule.
To this end, we define the vector field:
\begin{equation}\label{eq:unknown for conservative sys}
    \Z(\bu) = \left( \begin{matrix}
       Z_1(\bu) \\
        Z_2(\bu)
    \end{matrix} \right) :=  \left( \begin{matrix}
       \varrho(u,v) u \\
        v
    \end{matrix} \right),
\end{equation}
where $\varrho$ is prescribed by the Bernoulli law \eqref{eq:bernoulli general gamma intro}
so that \eqref{eq:system in intro} reads as
\begin{equation}\label{eq:conservative form u v ii}
      \Z(\bu)_x + \bF (\Z(\bu))_y = 0
\end{equation}
with
$$
\bF (\Z(\bu)) = \left(\begin{matrix}\varrho(u,v) v \\  -u \end{matrix}\right).
$$
We note in passing that, for $\gamma=3$, the mapping: $\bu \mapsto \Z(\bu)$
in \eqref{eq:unknown for conservative sys} is well-defined
in region $\{q_{\cri} \leq |\bu| \leq q_{\cav}\}$ and is invertible
in subregion $\{q_{\cri} < |\bu| < q_{\cav}\}$.
Indeed, if the system is constrained to this subregion,
we may reconstruct the velocity field $\bu$ from $\Z$ by the explicit formulas:
\begin{equation}\label{eq:reconstructing u from U}
   \varrho(\Z) = \frac{1}{\sqrt{2}}\sqrt{(1-Z_2^2) - \sqrt{(1-Z_2^2)^2 - 4Z_1^2}}, \quad
   u(\Z) = \frac{Z_1}{\varrho(\Z)}, \quad v(\Z) = Z_2.
\end{equation}
The relations above are amount to solving for $\varrho$ in terms of $\Z$ in the Bernoulli law:
\begin{equation}\label{eq:bernoulli for U1 U2}
    \varrho^{\gamma-1} = 1 - \frac{\gamma-1}{2}\Big( \big(\frac{Z_1}{\varrho}\big)^2 + Z_2^2 \Big),
\end{equation}
which may be inverted explicitly for $\gamma=3$,
provided that the condition: $0<\varrho < \varrho_{\cri}$ is imposed.

We are now ready to give the proof of Proposition \ref{lem:hyperbolic and gn}.

\begin{proof}[Proof of Proposition \ref{lem:hyperbolic and gn}]
The proof is divided into two steps.

\smallskip
\noindent 1. \textit{Hyperbolicity}: First, matrix $\nabla_\Z \bF$ may be computed by comparing
the change of coordinates in \eqref{eq:system polar sec 2 before strict hyp and genuine nonlinear}
and the chain rule in \eqref{eq:conservative form u v ii}:
\begin{equation*}
    \nabla_\Z \bF = \mathbf{J}^{-1} \mathbf{M} \mathbf{J}, \qquad \mathbf{J} := \left( \begin{matrix}
        q_{Z_1} & q_{Z_2} \\
        \vartheta_{Z_1} & \vartheta_{Z_2}
    \end{matrix} \right).
\end{equation*}
Matrix $\mathbf{J}$ is obtained by first calculating the Jacobian for the change of
variables: $(u,v) \mapsto (Z_1,Z_2)$ via implicitly differentiating \eqref{eq:bernoulli for U1 U2}, \textit{i.e.},
\begin{equation*}
    \mathbf{K} := \left( \begin{matrix}
       u_{Z_1}  & u_{Z_2} \\
      v_{Z_1} & v_{Z_2}
    \end{matrix} \right) = \left( \begin{matrix}
       -\frac{c^2 }{\varrho(u^2-c^2)}  & -\frac{uv}{u^2-c^2} \\
      0 & 1
    \end{matrix} \right),
\end{equation*}
and then multiplying by the corresponding matrix for the mapping: $(q,\vartheta)\mapsto (u,v)$, \textit{i.e.},
\begin{equation*}
  \mathbf{L} := \left( \begin{matrix}
       q_u  & q_v \\
       \vartheta_u  & \vartheta_v
    \end{matrix} \right) = \left( \begin{matrix}
       \cos\vartheta  & \sin\vartheta \\
      -\frac{1}{q}\sin\vartheta  & \frac{1}{q}\cos\vartheta
    \end{matrix} \right),
\end{equation*}
whence
\begin{equation*}
    \mathbf{J} = \mathbf{L}\mathbf{K}.
\end{equation*}
Notice that, while the matrices above are not well-defined in terms of the coordinates $(Z_1,Z_2)$
at the vacuum,
they are well-defined on the locus: $\{\varrho=0\}$
in terms of either coordinate system $(u,v)$ or $(q,\vartheta)$.

It follows from the previous computations that
matrix $\nabla_\Z \bF$ is uniquely characterized
by its normalized right-eigenvectors:
\begin{equation*}
    \mathbf{e}_\pm = \frac{\mathbf{J}^{-1}\mathbf{R}_\pm}{|\mathbf{J}^{-1}\mathbf{R}_\pm|}
    = \frac{1}{\sqrt{1+\varrho^2c^{-2} (q^2-c^2)} } \left( \begin{matrix}
        \varrho c^{-2}( uv + \Lambda_\pm (c^2-u^2)) \\
        -1
    \end{matrix} \right),
\end{equation*}
and corresponding eigenvalues are given again by $\Lambda_\pm$ that are distinct and real
when $\varrho>0$ and $q>c$,
and are equal on locus $\{\varrho=0\}$ when $c=0$.
Thus, the conservative system is strictly hyperbolic in the supersonic region away from the vacuum.

\smallskip
\noindent 2. \textit{Genuine nonlinearity}: A direct computation yields
\begin{equation*}
\partial_q \Lambda_\pm
=\mp\frac{\frac{qc}{\sqrt{q^2-c^2}}\big( 1+ \frac{\gamma-1}{2}M^2 \big)}
{(c\sin\vartheta \mp \sqrt{q^2-c^2} \cos \vartheta)^2},
\quad \partial_\vartheta \Lambda_\pm = \frac{q^2}{(c\sin\vartheta \mp \sqrt{q^2-c^2} \cos\vartheta)^2},
\end{equation*}
whence
\begin{equation*}
   \begin{aligned}
   \nabla_{\Z} \Lambda_\pm \!\cdot\! \mathbf{e}_\pm
   \! &= (\mathbf{J}^\top \nabla_{(q,\vartheta)}\Lambda_\pm) \cdot \frac{\mathbf{J}^{-1}\mathbf{R}_\pm}{|\mathbf{J}^{-1}\mathbf{R}_\pm|} \\
   &= \frac{1}{|\mathbf{J}^{-1}\mathbf{R}_\pm|} \big( \nabla_{(q,\vartheta)} \Lambda_\pm \cdot \mathbf{R}_\pm \big) \\
   &= \pm \frac{\gamma+1}{2(c \sin\vartheta \mp \sqrt{q^2-c^2} \cos\vartheta)^3}
   \frac{q^3}{\sqrt{(1+\varrho^{3-\gamma} (q^2-c^2))(q^2-c^2)}}.
   \end{aligned}
\end{equation*}
Since the inequality
\begin{equation*}
    \frac{q^3}{\sqrt{(1+\varrho^{3-\gamma} (q^2-c^2))(q^2-c^2)}} > 0
\end{equation*}
holds everywhere in the hyperbolic region including at the vacuum for $\gamma=3$, the conservative system \eqref{eq:conservative form u v ii}
is genuinely nonlinear in the supersonic region $\{q \geq c\}$
including the vacuum states $\{\varrho=0\}$.
\end{proof}

\subsection{Riemann invariants}\label{sec:riemann invariants} We now show Proposition \ref{lem:riemann invariants}.

As per \cite[\S 4]{gq-transonic}, system \eqref{eq:system polar sec 2 before strict hyp and genuine nonlinear}
is endowed with Riemann invariants $W_\pm$ that satisfy the conditions:
$\nabla_{(q,\vartheta)}W_\pm \cdot \mathbf{R}_\pm = 0$, \textit{i.e.},
\begin{equation}\label{eq:Riem pdes}
    \partial_\vartheta W_\pm = 1, \quad \partial_q W_\pm = \mp \frac{\sqrt{q^2 - c^2}}{qc}
\end{equation}
when
$q>c$.
Using the abuse of notation $k(q)$ to denote $k$ as a function of $q$ when composed with $\nu \circ \varrho$, we compute
\begin{equation*}
    k'(q) = \frac{\d k}{\d \nu}\frac{\d\nu}{\d\varrho}\frac{\d\varrho}{\d q} = -\frac{\sqrt{q^2 - c^2}}{qc}.
\end{equation*}
It follows from \eqref{eq:Riem pdes} that the closed-form expression \eqref{eq:Wpm expression 1}
for $W_\pm$ holds, where,  by comparing with \cite[\S 5]{gq-transonic} for $\gamma=3$,
    \begin{equation}\label{eq:Wpm expression 2}
    k(q) = (\sqrt{2}-1)\frac{\pi}{2}-\Big( \sqrt{2} \arcsin \sqrt{ 2q^2 - 1 } - \arcsin \sqrt{2-\frac{1}{q^2}} \Big);
\end{equation}
a general formula for $k$ for all $\gamma>1$ may be found in \textit{e.g.}~\cite[\S 117]{landau}.

The previous formula implies
\begin{equation}\label{eq:k' of q}
    k'(q) = -\frac{1}{q}\sqrt{\frac{2q^2-1}{{1-q^2}}} < 0
\end{equation}
for all speeds $q \in (q_{\cri},q_{\cav}) = (\frac{1}{\sqrt{2}},1)$.
Thus, $k:(q_{\cri},q_{\cav}) \to \mathbb{R}$ is strictly decreasing
and hence invertible on this subdomain. It follows that, for all $q \in [q_{\cri},q_{\cav}]$,
\begin{equation*}
    (\sqrt{2}-1)\frac{\pi}{2} = k(q_{\cri}) \geq k(q) \geq k(q_{\cav}) = 0.
\end{equation*}

We are now ready to prove Proposition \ref{lem:riemann invariants}.

\begin{proof}[Proof of Proposition \ref{lem:riemann invariants}]
The proof is divided into two steps.

\smallskip
\noindent 1. \textit{Invariant coordinates}:
A direct computation shows that
\begin{equation}\label{eq:Wpm change of basis}
\nabla_\Z W_\pm \cdot \mathbf{e}_\pm
= (\mathbf{J}^\top \nabla_{(q,\vartheta)}W_\pm) \cdot \frac{\mathbf{J}^{-1}\mathbf{R}_\pm}{|\mathbf{J}^{-1}\mathbf{R}_\pm|}
=\frac{1}{|\mathbf{J}^{-1}\mathbf{R}_\pm|} \big( \nabla_{(q,\vartheta)}W_\pm \cdot \mathbf{R}_\pm \big) =0,
\end{equation}
whence we deduce that functions $W_\pm$ of \eqref{eq:Wpm expression 1} are Riemann invariants for the conservative system \eqref{eq:conservative form u v ii}.

\smallskip
\noindent 2. \textit{Invertible coordinate transformation}: We consider the bijectivity of the curvilinear
coordinate transformation $\bu \mapsto (W_+(\bu),W_-(\bu))$.
This is manifestly well-defined via formulas \eqref{eq:Wpm expression 1} and \eqref{eq:Wpm expression 2};
it is then a matter of verifying whether they are invertible.

By recalling computation \eqref{eq:k' of q}, the velocity vector $\bu$ may be reconstructed
away from the vacuum and the sonic curve (where $k$ is invertible) as follows:
 \begin{equation*}
     \begin{aligned}
        & q(W_-,W_+) = k^{-1}\big( \frac{1}{2}(W_+ -W_-) \big) , \qquad \vartheta(W_-,W_+) = \frac{1}{2}\big( W_+ + W_- \big), \\
        & \bu(W_-,W_+) = q(W_-,W_+) \left(\begin{matrix} \cos\vartheta(W_-,W_+) \\
        \sin\vartheta(W_-,W_+) \end{matrix}\right).
     \end{aligned}
 \end{equation*}
 Meanwhile, at the vacuum and on the sonic curve, respectively, we write
 \begin{equation*}
     \bu(W_-,W_+)|_{\cav} = q_{\cav} \left(\begin{matrix} \cos\vartheta(W_-,W_+) \\
        \sin\vartheta(W_-,W_+) \end{matrix}\right), \quad \bu(W_-,W_+)|_{\cri} = q_{\cri} \left(\begin{matrix} \cos\vartheta(W_-,W_+) \\
        \sin\vartheta(W_-,W_+) \end{matrix}\right),
 \end{equation*}
 and the result follows.
\end{proof}

\subsection{Entropy structure}\label{sec:entropic struc intro}
The purpose of this
subsection is to provide a derivation of the entropy generator equation \eqref{eq:ent gen intro}.
We emphasize that the results of this
subsection were already contained in \cite{gq-transonic}; we recall them to make the present paper self-contained.
In what follows, to simplify the notation, we omit the $\varepsilon$-superscript.

We recast the approximate system \eqref{eq:potential approx}
in the matrix form, analogous to \eqref{eq:system polar sec 2 before strict hyp and genuine nonlinear},
and expand $q$ in terms of $\varrho$ by using the Bernoulli relation with $q'(\varrho) = - \frac{c^2}{\varrho q}$.
As per
\cite[\S 3]{gq-transonic}, we obtain
\begin{equation}\label{eq:polar system}
    \mathbf{A} \left(\begin{matrix}
    \varrho \\
    \vartheta
    \end{matrix}\right)_x + \mathbf{B} \left( \begin{matrix}
    \varrho \\
    \vartheta
    \end{matrix} \right)_y = \left( \begin{matrix}
    -V_2 \\
    -\frac{q^2}{q^2 - c^2} V_1
    \end{matrix} \right),
\end{equation}
where
\begin{equation}\label{eq:polar system matrix}
   \begin{aligned}
       \mathbf{A} =& \,  \left( \begin{matrix}
   \frac{c^2}{\varrho q}\sin \vartheta & -q \cos \vartheta \\
  -q\cos \vartheta & -\frac{\varrho q^3}{c^2-q^2}\sin \vartheta
    \end{matrix}\right) , \quad     \mathbf{B} =  \left( \begin{matrix}
  - \frac{c^2}{\varrho q}\cos \vartheta & -q \sin \vartheta \\
  -q\sin \vartheta & \frac{\varrho q^3}{c^2-q^2}\cos \vartheta
    \end{matrix}\right).
   \end{aligned}
\end{equation}

Computing the scalar product of system \eqref{eq:polar system} with
the vector field $(\Phi_\vartheta, \Phi_\varrho)^\top$ for a scalar function $\Phi$
of the phase-space variables, we have
\begin{equation}\label{eq:ent dissipation from polar sys}
    \begin{aligned}
        \div_\x \Q = - \Phi_\vartheta V_2 + \frac{q^2}{c^2 - q^2}\Phi_\varrho V_1,
    \end{aligned}
\end{equation}
where the pair of functions $\Q=(Q_1,Q_2)$ is required to satisfy
\begin{equation}\label{eq:what is needed from Q}
    \begin{aligned}
       & \partial_\varrho Q_1 = \frac{c^2}{\varrho q} \sin\vartheta \Phi_\vartheta - q \cos \vartheta \Phi_\varrho,
       \quad &&\partial_\vartheta Q_1 =  -q \cos\vartheta \Phi_\vartheta - \frac{\varrho q^3}{c^2-q^2}\sin\vartheta \Phi_\varrho , \\
        &\partial_\varrho Q_2 =  -\frac{c^2}{\varrho q}\cos\vartheta \Phi_\vartheta - q \sin \vartheta \Phi_\varrho,
        \quad &&\partial_\vartheta Q_2 = -q \sin\vartheta \Phi_\vartheta + \frac{\varrho q^3}{c^2-q^2}\cos\vartheta \Phi_\varrho.
    \end{aligned}
\end{equation}
We emphasize for clarity that the divergence in \eqref{eq:ent dissipation from polar sys}
is understood as
$$
\div_\x \Q
= \nabla_{(\varrho,\vartheta)}Q_1 \left(
\begin{matrix}
    \varrho \\
    \vartheta
\end{matrix}\right)_x
+ \nabla_{(\varrho,\vartheta)}Q_2 \left(\begin{matrix}
    \varrho \\
    \vartheta
\end{matrix}\right)_y.
$$
Along with \eqref{eq:what is needed from Q}, the requirement that
$\partial_\vartheta \partial_\varrho Q_j = \partial_\varrho \partial_\vartheta Q_j$, $j=1,2$,
implies
\begin{equation*}
    \sin \vartheta \Big(\frac{c^2}{\varrho q} \Phi_{\vartheta \vartheta} + q \Phi_\varrho
    + \Big(  \frac{\varrho q^3}{c^2 - q^2} \Phi_\varrho \Big)_\varrho \Big) = 0,
\end{equation*}
which is satisfied,
provided that $\Phi$ solves the Tricomi--Keldysh equation
\begin{equation}\label{eq:tricomi for D dissipation}
    \frac{c^2}{\varrho^2 q^2}\Phi_{\vartheta\vartheta} + \Big( \frac{q^2}{c^2-q^2}\Phi_\varrho \Big)_\varrho = 0.
\end{equation}
Define $H$ to be a solution of the system:
\begin{equation}\label{eq:sys H from D}
   \left\lbrace \begin{aligned}
    & \varrho H_{\nu \vartheta} - H_\vartheta = - \Phi_\vartheta, \\
    & H_\nu + \frac{1}{\varrho}H_{\vartheta\vartheta} = \frac{q^2}{c^2-q^2}\Phi_\varrho.
    \end{aligned}\right.
\end{equation}
We deduce from \eqref{eq:tricomi for D dissipation} that $H$ satisfies the entropy generator
equation \eqref{eq:ent gen intro};
note also for consistency that equality $\Phi_{\vartheta\nu}=\Phi_{\nu\vartheta}$ implies
\begin{equation*}
   \Big( H_{\nu\nu} - \frac{M^2-1}{\varrho^2}H_{\vartheta\vartheta} \Big)_\vartheta = 0,
\end{equation*}
which is manifestly satisfied, provided that $H$ is a solution of \eqref{eq:ent gen intro}.

Conversely, if $H$ is a solution of the entropy generator equation \eqref{eq:ent gen intro},
then, defining $\Phi$ such that \eqref{eq:sys H from D} is satisfied,
we see by a direct computation that $\Phi$ satisfies the Tricomi--Keldysh equation \eqref{eq:tricomi for D dissipation}.
Moreover, letting $\Q$ be the pair of functions generated from this $H$ via the Loewner--Morawetz
relations \eqref{eq:lowener mor}, it is direct
to check that relations \eqref{eq:what is needed from Q}
are verified and the dissipation equality \eqref{eq:ent dissipation from polar sys} holds in the form:
\begin{equation}\label{eq:ent dissipation from polar sys with H}
    \begin{aligned}
        \div_\x \Q =  \big(\varrho H_{\nu \vartheta} - H_\vartheta\big) V_2
 + \big( H_\nu + \frac{1}{\varrho}H_{\vartheta\vartheta} \big) V_1.
    \end{aligned}
\end{equation}

Finally, we emphasize that formula \eqref{eq:ent dissipation from polar sys with H} holds
for entropy pairs $\Q$ understood as functions of the conservative variables of \eqref{eq:conservative form u v ii}.
Indeed, by setting $\Q'(\Z) := \Q(\varrho(\Z),\vartheta(\Z))$ and
\begin{equation*}
    \mathbf{N} := \left( \begin{matrix}
        \varrho_{Z_1} & \varrho_{Z_2} \\
        \vartheta_{Z_1} & \vartheta_{Z_2}
    \end{matrix} \right),
\end{equation*}
a direct computation yields that, for entropy pairs generated via the Loewner--Morawetz
relations \eqref{eq:lowener mor},
\begin{equation*}
\begin{aligned}
       \div_\x \Q'(\Z)  &= \nabla_{(\varrho,\vartheta)}Q_1 \, \mathbf{N}\Z_x
       + \nabla_{(\varrho,\vartheta)}Q_2 \,\mathbf{N} \Z_y  \\
       &= \nabla_{(\varrho,\vartheta)}Q_1  \left(\begin{matrix}
    \varrho \\
    \vartheta
    \end{matrix}\right)_x + \nabla_{(\varrho,\vartheta)}Q_2 \left(\begin{matrix}
    \varrho \\
    \vartheta
    \end{matrix}\right)_y \\
    &= \big(\varrho H_{\nu \vartheta} - H_\vartheta\big) V_2
 + \big( H_\nu + \frac{1}{\varrho}H_{\vartheta\vartheta} \big) V_1,
  \end{aligned}
\end{equation*}
by using \eqref{eq:ent dissipation from polar sys}.

\section{Entropy Generator I: Regular Kernel}\label{sec:reg kernel}
In this and the next section, we prove the existence of two
linearly
independent fundamental solutions of
the entropy generator equation \eqref{eq:ent gen intro} and
identify
a cancellation property that is crucial for establishing
the $H^{-1}_\loc$-compactness of the
entropy dissipation
measures (\textit{cf.}~ \S \ref{section:what estimates}).
To achieve this, we exploit the linearity of the equation
and proceed by the Fourier methods, which were employed, for instance,
in \cite{ChenLeFloch}.
That is to say, for an appropriately
chosen $\lambda \in \mathbb{R}$,
we postulate an expansion of the form:
\begin{equation*}
H(\nu,s) = G_\lambda(\nu,s) + G_{\lambda+1}(\nu,s) + \text{remainder term},
\end{equation*}
where
\begin{equation}\label{eq:Glambda first def}
    G_\lambda (\nu,\vartheta) := [k(\nu)^2-\vartheta^2]_+^\lambda
    \qquad\,\, \mbox{for $\lambda > -1$},
\end{equation}
and the remainder term is a more regular function. It is well-known (\textit{cf.}~\cite{gelfandshilov}) that,

\begin{equation*}
\hat{G}_\lambda(\nu,\xi)
= c_{\lambda } k(\nu)^{\lambda+ \frac{1}{2}}|\xi|^{-\lambda -\frac{1}{2}}J_{\lambda +\frac{1}{2}}(|\xi|k(\nu))
\qquad\,\, \mbox{for $\lambda \in (-1,\infty)$},
\end{equation*}
with $c_{\lambda} := \sqrt{\pi} 2^{\lambda+\frac{1}{2}} \Gamma(\lambda+1)$
and the Bessel function $J_{\lambda+\frac{1}{2}}$ of the first kind.
In the case when $\lambda$ is integer-valued, these expressions
simplify to the weighted sums of trigonometric functions; see Appendix \ref{appendix:Glambda}.

The goal of this section is to prove the following two results concerning the regular kernel.
The proof of Proposition \ref{prop:reg kernel exist} readily generalizes
to a much wider range of $\gamma$,
while Proposition \ref{prop:est for h-1 cpct reg} is particular to the case: $\gamma=3$.

\begin{prop}[Existence of the Regular Kernel]\label{prop:reg kernel exist}
There exists a distributional solution $H^\r$, supported in $\mathcal{K}$, of the Cauchy problem{\rm :}
\begin{equation}\label{eq:reg kernel ivp}
    \left\lbrace\begin{aligned}
         &H^\r_{\nu\nu} - k'^2 H^\r_{ss} = 0 \qquad \text{in $(0,\nu_*] \times \mathbb{R}$},\\
        &H^\r|_{\nu = 0} = 0, \\
        &H^\r_\nu|_{\nu=0} = \delta_0,
    \end{aligned}\right.
\end{equation}
given by the formula:
\begin{equation}\label{eq:expansion reg kernel}
    \begin{aligned}
        H^\r(\nu,s) = a_{0}(\nu) G_1(\nu,s) + a_1(\nu) G_2(\nu,s) + g(\nu,s),
    \end{aligned}
\end{equation}
where $a_j(\nu), j=0,1$, satisfy that there exists $C>0$ independent of $\nu$ such that
\begin{equation}\label{eq:bounds on a0 a1}
    \begin{aligned}
       &|a_j(\nu)| + \varrho(\nu)|a_j'(\nu)| \leq C \qquad\mbox{for $j=0,1$},
    \end{aligned}
\end{equation}
and the remainder term $g$ is twice continuously differentiable,
with all of its first and second derivatives being $\alpha$-H\"older continuous
in
variable $s$ for all $\alpha \in [0,1)$.
\end{prop}

Our second main result of this section is concerned with the estimates
required to show the $H^{-1}_\loc$--compactness of the entropy dissipation measures.

\begin{prop}\label{prop:est for h-1 cpct reg}
For any $\varphi \in C^\infty(\mathbb{R})$, there exists a constant $C_\varphi>0$ independent of $(\nu,s)$
such that, for all $(\nu,s)\in [0,\nu_*]\times\mathbb{R}$ and $j=0,1,2$,
\begin{align}
&\big| \partial^j_s \big(\varrho(\nu)H^\r_{\nu }(\nu,\cdot) + H^\r_{ss}(\nu,\cdot)\big)*\varphi(s)\big|
\leq C_\varphi \varrho(\nu), \label{eq:reg cpct est 1}\\[1mm]
&\big| \varrho(\nu)\partial^j_sH^\r_{\nu}(\nu,\cdot)*\varphi(s) \big|
+ \big| \partial^j_s H^\r_{s}(\nu,\cdot)*\varphi(s)\big| \leq C_\varphi. \label{eq:reg cpct est 2}
\end{align}
\end{prop}

We now provide the proofs of Propositions \ref{prop:reg kernel exist}--\ref{prop:est for h-1 cpct reg}.
We start with the analysis of regular kernel coefficients.

\smallskip
\subsection{Computing the regular kernel coefficients}\label{sec:computing reg kernel coeff}
We postulate an expansion of the Fourier transform of the regular kernel
with respect to
variable $s\in \mathbb{R}$ in the form:
\begin{equation*}
    \hat{H}^\r(\nu,\xi) = \sum_{j=0}^1 \hat{H}^{\r,j}(\nu,\xi)+\hat{g}(\nu,\xi),
\end{equation*}
where
$\hat{H}^{\r,j}(\nu,\xi) = \alpha_j(\nu) \hat{f}_{j+1}(\xi k(\nu))$ for $j=0,1$.

In what follows, for simplicity of notation, we write $\hat{f}$ for
$\hat{f}(\xi k(\nu))$.
We use the formulas provided in Lemma \ref{lemma:flambda relations} to compute
\begin{equation*}
\hat{H}^{\r,0}_{\nu\nu} + \xi^2 k'^2 \hat{H}^{\r,0}
=  \,\alpha_0'' \hat{f}_1 +\Big( \alpha_0 k'^2 - \frac{1}{4}\big( \alpha_0 k'' k
  + 2 \alpha_0' k' k\big) \Big) \xi^2 \hat{f}_2,
\end{equation*}
and impose that
\begin{equation*}
        \alpha_0 k'^2 - \frac{1}{4}\big( \alpha_0 k'' k + 2 \alpha_0' k' k \big) = 0,
\end{equation*}
that is,
\begin{equation*}
    \begin{aligned}
        \frac{\alpha_0'}{\alpha_0} = 2 \frac{k'}{k}-\frac{1}{2}\frac{k''}{k'},
    \end{aligned}
\end{equation*}
whose closed-form solution is given by
\begin{equation}\label{eq:alpha0}
\alpha_0(\nu) = c_0 k(\nu)^{2} k'(\nu)^{-\frac{1}{2}}
\end{equation}
with $c_0$ chosen to satisfy the condition: $\hat{H}^\r_\nu(0,\xi) = 1$,
\textit{i.e.}, $H_\nu^\r|_{\nu=0}=\delta_0$; see \S \ref{sec:initial datum reg}.

For the second term, using again the formulas in Lemma \ref{lemma:flambda relations}, we have
\begin{equation*}
    \begin{aligned}
        \hat{H}^{\r,1}_{\nu\nu} + \xi^2 k'^2 \hat{H}^{\r,1} &= \Big( \alpha_1'' - 5 \big( 2\alpha_1' \frac{k'}{k}
        + \alpha_1 \frac{k''}{k} - 6 \alpha_1 \frac{k'^2}{k^2} \big) \Big) \hat{f}_2 \\
        &\quad\, + 4 \Big( 2\alpha_1' \frac{k'}{k} + \alpha_1 \frac{k''}{k} - 6 \alpha_1 \frac{k'^2}{k^2} \Big) \hat{f}_1,
    \end{aligned}
\end{equation*}
whence we impose
\begin{equation}\label{eq:eqn for alpha1}
        4 \Big( 2\alpha_1' \frac{k'}{k} + \alpha_1 \frac{k''}{k} - 6 \alpha_1 \frac{k'^2}{k^2} \Big) + \alpha_0'' = 0,
\end{equation}
that is,
\begin{equation*}
\big( k^{-3} k'^{\frac{1}{2}} \alpha_1 \big)' = -\frac{1}{8}k^{-2}k'^{-\frac{1}{2}}\alpha_0'',
\end{equation*}
whose solution is given by
\begin{equation}\label{eq:alpha1}
\alpha_1(\nu) = -\frac{1}{8}k(\nu)^3 k'(\nu)^{-\frac{1}{2}}
  \int_0^\nu k(\tau)^{-2}k'(\tau)^{-\frac{1}{2}}\alpha_0''(\tau) \d \tau.
\end{equation}
We therefore have
\begin{equation*}
        \hat{H}^{\r,1}_{\nu\nu} + \xi^2 k'^2 \hat{H}^{\r,1}
        = \,  \big( \alpha_1'' + \frac{5}{4} \alpha_0'' \big) \hat{f}_2 -\alpha_0'' \hat{f}_1.
\end{equation*}

It follows that
the remainder term $\hat{g}$ must satisfy the following equation:
\begin{equation}\label{eq:remainder eqn}
        \hat{g}_{\nu\nu} + \xi^2 k'^2 \hat{g} = \ell\hat{f}_2,
\end{equation}
where
\begin{equation}\label{eq:ell def}
    \ell(\nu) = -\big( \alpha_1''(\nu) + \frac{5}{4} \alpha_0''(\nu) \big).
\end{equation}
Therefore, we define
\begin{equation}\label{eq:big Rr def}
    R^\r(\nu,s) := \ell(\nu)k(\nu)^{-5}G_2(\nu,s)
\end{equation}
such that $\hat{R}^\r(\nu,\xi) = \ell(\nu)f_2(\xi k(\nu))$, whence we impose that the remainder term satisfies
\begin{equation*}
    \left\lbrace\begin{aligned}
        & g_{\nu\nu} - k'^2 g_{ss}  = R^\r \qquad \text{in $(0,\nu_*]\times\mathbb{R}$}, \\
        &g|_{\nu = 0} = 0, \\
        &g_\nu|_{\nu=0} = 0.
    \end{aligned}\right.
\end{equation*}

In turn, the expansion for $H^\r$ reads as
\begin{equation}
 H^\r(\nu,s) = a_0(\nu) G_1(\nu,s) + a_1(\nu)G_2(\nu,s) + g(\nu,s)
\end{equation}
with
\begin{equation}\label{eq:a0 and a1 def from alphas}
    a_0(\nu) := \alpha_0(\nu) k(\nu)^{-3}, \qquad a_1(\nu) := \alpha_1(\nu) k(\nu)^{-5}.
\end{equation}

\smallskip
\subsection{Asymptotic analysis of the regular kernel}\label{sec:asymp analysis reg kernel}
The purpose of this section is to collect the results
concerning the asymptotic behavior of the characteristic speed function $k(\nu)$ in the vicinity of the vacuum.

\subsubsection{Properties of the characteristic speed}

For the sake of being self-contained, we provide a brief outline of
the proof for the case:
$\gamma = 3$;
further details may be found in \cite[Appendix C]{thesis}.

\begin{prop}[Asymptotic Behavior of the Characteristic Speed]\label{prop:k expand}
The characteristic speed function $k(\nu)$ admits the decomposition{\rm :}
\begin{equation}
    k(\nu) = c_\sharp \nu^{\frac{1}{3}} + c_\flat \nu + c_l\nu^{\frac{5}{3}} + L(\nu)
    \qquad \text{for $\nu \in [0,\nu_*]$},
\end{equation}
where $c_\sharp = 3^{\frac{1}{3}}$, $c_\flat$ and $c_l$ are the constants that can be determined.
Moreover, there exists a constant $C=C(\nu_*)>0$ such that
\begin{equation}
    |L^{(j)}(\nu)| \leq C\nu^{\frac{7}{3}-j}
    \qquad \text{for $\nu \in (0,\nu_*]$ and $j=0,\cdots,4$}.
\end{equation}
\end{prop}

\begin{proof}
We write an asymptotic description of $\nu$, defined in \eqref{eq:nu def}, near the vacuum.
Using the Binomial Theorem and the Weierstra{\ss} M-test for $\varrho$ sufficiently small, we have
\begin{equation*}
\begin{aligned}
\nu(\varrho) = \sum_{j=1}^\infty (-1)^{j-1}\int_0^{\varrho} \tau^{2j} \d \tau
&= \frac{1}{3}\varrho^3 \Big( 1 + \varrho^2 \sum_{j=0}^\infty \frac{3}{2j+5} (-1)^{j+1} \varrho^{2j} \Big) \\
&= \frac{1}{3}\varrho^3 \big( 1 - \frac{3}{5} \varrho^2 + \bigo(\varrho^4) \big).
    \end{aligned}
\end{equation*}
In turn, it is direct
to verify that $\nu(\varrho) = \bigo(\varrho^3)$, whence
$\varrho(\nu) = \bigo(\nu^{\frac{1}{3}})$.
More precisely, using the Binomial Theorem again, we write
\begin{equation*}
    \begin{aligned}
        3^{\frac{1}{3}}\nu(\varrho)^{\frac{1}{3}} = \varrho \big( 1 - \frac{3}{5} \varrho^2
        + \bigo(\varrho^4)\big)^{\frac{1}{3}} = \varrho - \frac{1}{5} \varrho^3 + \bigo(\varrho^5)
        = \varrho - \frac{3}{5}\nu(\varrho) + \bigo(\varrho^5),
    \end{aligned}
\end{equation*}
which implies
\begin{equation}\label{eq:rho expansion precise}
    \varrho(\nu) = 3^{\frac{1}{3}}\nu^{\frac{1}{3}} + \frac{3}{5} \nu + \bigo(\nu^{\frac{5}{3}}).
\end{equation}
It follows from
\eqref{eq:what is k' squared}, the Binomial Theorem, and
the expression for the Mach number $M(\varrho)=\varrho^{-1}\sqrt{1-\varrho^2}$ that,
for sufficiently small density,
\begin{equation*}
    \begin{aligned}
        k'(\nu) =& \,  \varrho(\nu)^{-2} \big(1-2\varrho(\nu)^2\big)^{\frac{1}{2}}
         = 3^{-\frac{2}{3}} \nu^{-\frac{2}{3}}\big(1 +  3^{\frac{2}{3}}c_\flat \nu^{\frac{2}{3}} + \bigo(\nu^{\frac{4}{3}})\big),
    \end{aligned}
\end{equation*}
whence, by integrating the above,
the result with $c_\sharp$ explicitly determined
can directly be obtained.
By refining the present strategy, it is direct to compute $c_\flat$ and $c_l$ explicitly
and to obtain the precise bounds on the remainder $L$.
\end{proof}

We remark in passing the following regularity and monotonicity properties of $k$.

\begin{remark}[Regularity and Monotonicity of the Characteristic Speed]\label{remark:k smooth on open}
Using representation \eqref{eq:k def},
$k \in C([0,\nu_{\cri}]) \cap C^1((0,\nu_{\cri}]) \cap C^\infty((0,\nu_{\cri}))$
and
\begin{equation*}
    0 < k(\nu) \leq k(\nu_{\cri}) \quad \text{for } \nu \in (0,\nu_{\cri}], \qquad k'(\nu) > 0 \quad \text{for } \nu \in (0,\nu_{\cri}),
\end{equation*}
while $k(0)=0=k'(\nu_{\cri})$.
Moreover, a direct calculation (\textit{cf.}~\cite[Lemma C.4]{thesis}) shows
\begin{equation*}
    k''(\nu) = -\frac{M^2}{\varrho^2\sqrt{M^2-1}}\big( \varrho^{-2} + (M^2-1) \big) < 0
    \qquad \text{for } \nu \in (0,\nu_{\cri}).
\end{equation*}
Note that the above explodes at the extremities $\nu=0$ and $\nu = \nu_{\cri}$, where the hyperbolicity of the system is lost.
\end{remark}

From Proposition \ref{prop:k expand} and Remark \ref{remark:k smooth on open}, we derive the following corollary.

\begin{corollary}\label{cor:above and below}
There exists $C=C(\nu_*)>0$ such that, for all $\nu \in (0,\nu_*]$,
\begin{equation}\label{eq:above and below}
    \begin{aligned}
        & C^{-1} \nu^{\frac{1}{3}} \leq \varrho(\nu) \leq C\nu^{\frac{1}{3}}, \quad && C^{-1} \nu^{\frac{1}{3}} \leq k(\nu) \leq C\nu^{\frac{1}{3}} , \\
        & C^{-1} \nu^{-\frac{2}{3}} \leq k'(\nu) \leq C\nu^{-\frac{2}{3}}, \quad && C^{-1} \nu^{-\frac{5}{3}} \leq -k''(\nu) \leq C\nu^{-\frac{5}{3}}.
    \end{aligned}
\end{equation}
\end{corollary}

The following cancellation result is a direct corollary of Proposition \ref{prop:k expand},
which is a refinement of \cite[Lemma 4.16]{thesis} for the case: $\gamma = 3$:

\begin{corollary}\label{cor:OG cancellation derivs}
There exists $C=C(\nu_*)>0$ such that,  for all $\nu \in (0,\nu_*]$,
\begin{equation}
\Big|\big((2k'^2 + kk'')-\frac{10}{9}c_\sharp c_\flat \nu^{-\frac{2}{3}}-\tilde{c} \big)^{(j)}(\nu)\Big|
    \leq C\nu^{\frac{2}{3}-j} \qquad\mbox{for $j=0,1,2$},
\end{equation}
where $\tilde{c} = \frac{28}{9}c_\sharp c_l + 2c_\flat^2$.
\end{corollary}

\smallskip
\subsubsection{Coefficients of the regular kernel in the vicinity of the vacuum}\label{sec:asymptotics reg}

Throughout this subsection, we use, with a slight abuse of notation, the shorthand $\bigo(\nu^d)$
to denote a function $\phi(\nu)$ for which $|\phi^{(j)}(\nu)|\leq C\nu^{d-j}$ hold for integer $j$, for some constant $C>0$.
We adopt this convention to simplify the calculations contained in the proofs that follow.

\begin{lemma}\label{lem:asymp exp alpha0 and alpha1}
The following asymptotic expansions hold{\rm :}
    \begin{equation*}
        \begin{aligned}
            &\alpha_0(\nu) = C_{\alpha_0,0} \nu
 + C_{\alpha_0,1} \nu^{\frac{5}{3}} + C_{\alpha_0,2} \nu^{\frac{7}{3}} + r_{\alpha_0}(\nu), \\
 &\alpha_1(\nu) = C_{\alpha_1,0} \nu^{\frac{5}{3}} + C_{\alpha_1,1} \nu^{\frac{7}{3}} + r_{\alpha_1}(\nu),
        \end{aligned}
    \end{equation*}
with
\begin{equation*}
\begin{aligned}
& C_{\alpha_0,0} = c_0 3^{\frac{1}{2}} c_\sharp^{\frac{3}{2}}\neq 0,
\quad C_{\alpha_0,1} = C_{\alpha_0,0} \frac{c_\flat}{2c_\sharp}\neq 0, \quad
C_{\alpha_0,2}
= C_{\alpha_0,0}\Big( \frac{11}{8}(\frac{c_\flat}{c_\sharp})^2 - \frac{c_l}{2c_\sharp} \Big), \\
& C_{\alpha_1,0} = - \frac{5}{4}C_{\alpha_0,1}\neq 0, \quad C_{\alpha_1,1}
  = C_{\alpha_1,0} \Big( \frac{7}{3}\big(\frac{2C_{\alpha_0,2}}{5C_{\alpha_0,1}}- \frac{ c_\flat}{2 c_\sharp}\big) + \frac{3c_\flat}{2c_\sharp}\Big),
        \end{aligned}
    \end{equation*}
where
$r_{\alpha_0}$ and $r_{\alpha_1}$ satisfy the estimates{\rm :}
    \begin{equation*}
        \begin{aligned}
            |r^{(j)}_{\alpha_0}(\nu)| + |r^{(j)}_{\alpha_1}(\nu)| \leq C\nu^{3-j}
            \qquad \mbox{for $j=0,1,2,3$}.
        \end{aligned}
    \end{equation*}
In addition,
$\ell(\nu)$ satisfies the bounds{\rm :}
\begin{equation}\label{eq:ell bounds from asymp}
    |\ell(\nu)| + \nu|\ell'(\nu)| \leq C\nu^{\frac{1}{3}}.
\end{equation}
\end{lemma}

\begin{proof}
Using the previous formula for $\alpha_0$ and the asymptotic expansion for $k(\nu)$ near the vacuum, we obtain
\begin{equation*}
    \alpha_0(\nu) = c_0 c_\sharp^2 \nu^{\frac{2}{3}}\big( 1 + \frac{c_\flat}{c_\sharp}\nu^{\frac{2}{3}}
    + \frac{c_l}{c_\sharp}\nu^{\frac{4}{3}}
    + \bigo(\nu^2) \big)(\frac{c_\sharp}{3}\nu^{-\frac{2}{3}})^{-\frac{1}{2}}
    \big( 1 + \frac{3c_\flat}{c_\sharp}\nu^{\frac{2}{3}} + \frac{5 c_l}{c_\sharp}\nu^{\frac{4}{3}} + \bigo(\nu^2) \big)^{-\frac{1}{2}}.
\end{equation*}
Using the Binomial Theorem,
we expand the terms in the brackets and obtain the result for $\alpha_0$, provided that
$\nu$ is sufficiently small.
Furthermore, we have
\begin{equation*}
    \alpha_0''(\nu) = C_{\alpha_0'',0} \nu^{-\frac{1}{3}} + C_{\alpha_0'',1}\nu^{\frac{1}{3}} + \bigo(\nu)
\end{equation*}
with
\begin{equation*}
    C_{\alpha_0'',0} := \frac{10}{9}C_{\alpha_0,1} , \quad C_{\alpha_0'',1} = \frac{28}{9}C_{\alpha_0,2}.
\end{equation*}
Similarly, we obtain
\begin{equation*}
    \begin{aligned}
        \alpha_1(\nu) =& \!  -\! \frac{3}{8}C_{\alpha_0'',0}  \nu^{\frac{4}{3}} \big( 1 + \frac{3c_\flat}{2 c_\sharp} \nu^{\frac{2}{3}}
        + \bigo(\nu^{\frac{4}{3}}) \big) \! \int_0^\nu \!  \tau^{-\frac{2}{3}}
        \Big(1 + \big(\frac{C_{\alpha_0'',1}}{C_{\alpha_0'',0}}- \frac{7 c_\flat}{2 c_\sharp}\big)\tau^{\frac{2}{3}} + \bigo(\tau^{\frac{4}{3}})\Big) \d \tau \\
        =& \!  -\! \frac{9}{8}C_{\alpha_0'',0}  \nu^{\frac{5}{3}} \big( 1 + \frac{3c_\flat}{2 c_\sharp} \nu^{\frac{2}{3}} + \bigo(\nu^{\frac{4}{3}}) \big)
        \Big( 1 + \frac{1}{3}\big(\frac{C_{\alpha_0'',1}}{C_{\alpha_0'',0}}- \frac{7 c_\flat}{2 c_\sharp}\big) \nu^{\frac{2}{3}} + \bigo(\nu^{\frac{4}{3}})  \Big),
    \end{aligned}
\end{equation*}
so that
\begin{equation*}
    \alpha_1''(\nu) = C_{\alpha_1'',0}\nu^{-\frac{1}{3}} + \bigo(\nu^{\frac{1}{3}})
\end{equation*}
with $C_{\alpha_1'',0} := \frac{10}{9}C_{\alpha_1,0}$. It follows that
\begin{equation*}
    \alpha_1'' + \frac{5}{4}\alpha_0''
    = \big( \underbrace{C_{\alpha_1'',0} + \frac{5}{4}C_{\alpha_0'',0} }_{=0}\big) \nu^{-\frac{1}{3}} + \bigo(\nu^{\frac{1}{3}}),
\end{equation*}
and then the bounds of $\ell(\nu)$ are directly obtained.
\end{proof}

Then estimates \eqref{eq:bounds on a0 a1} on
coefficients $a_0$ and $a_1$ are directly deduced from formulas \eqref{eq:a0 and a1 def from alphas}
and Lemma \ref{lem:asymp exp alpha0 and alpha1}.

\subsection{Analysis of the remainder term for the regular kernel}\label{sec:remainder for reg}

It remains to show that there exists a classical solution $g$ of equation \eqref{eq:remainder eqn}.
This is the focus of this section, which also provides some estimates on the H\"older regularity
of the remainder term.

Note that the singularity at the vacuum precludes us directly from appealing to
the Cauchy--Lipschitz Theorem to show the existence of a solution $g$ of \eqref{eq:remainder eqn}.
Therefore, we study the family of ordinary differential equations posed away from the vacuum for $\varepsilon>0$:
\begin{equation}\label{eq:g eps eqn}
    \left\lbrace\begin{aligned}
    & \hat{g}^\varepsilon_{\nu\nu} + k'^2 \xi^2 \hat{g}^\varepsilon = \hat{R}^\r
    \quad\,\, &&\text{for } \nu \in [\varepsilon,\nu_*], \\
    &\hat{g}^\varepsilon|_{\nu = \varepsilon} = 0, \\
    &\hat{g}^\varepsilon_\nu |_{\nu=\varepsilon}= 0.
    \end{aligned}
    \right.
\end{equation}
Our strategy is as follows: the Cauchy--Lipschitz Theorem yields the existence
of a unique solution $\hat{g}^\varepsilon(\cdot,\xi) \in C^2([\varepsilon,\nu_*])$ for each fixed $\varepsilon>0$
and each fixed $\xi \in \mathbb{R}$. We then prove the uniform estimates
of sequence $\{g^\varepsilon\}_{\varepsilon>0}$ to pass to the limit strongly
by using the Ascoli--Arzel\`a Theorem, thereby obtaining a classical solution of \eqref{eq:remainder eqn}.

\subsubsection{Fourier estimates for the remainder term}\label{subsection:fourier rem}

This subsection is independent from the specific choice of kernel $H^\r$.
We also emphasize that, throughout this subsection, the $\varepsilon$-superscript
bears no relation to the sequence of approximate problems.

For each $\varepsilon>0$,
$h^\varepsilon(\cdot,\xi) \in C^2([\varepsilon,\nu_*])$
is the unique solution of the Cauchy problem:
\begin{equation}\label{eq:away from vacuum remainder}
    \left\lbrace\begin{aligned}
    & h^\varepsilon_{\nu\nu}(\nu,\xi) + k'(\nu)^2 \xi^2 h^\varepsilon(\nu,\xi) = r^\varepsilon(\nu,\xi) \quad
    &&\text{for } \nu \in [\varepsilon,\nu_*], \\
    &h^\varepsilon(\varepsilon,\xi) = 0, \\
    &h^\varepsilon_\nu (\varepsilon,\xi) = 0,
    \end{aligned}
    \right.
\end{equation}
provided by the Cauchy--Lipschitz Theorem,
where $r^\varepsilon(\cdot,\xi)$  is assumed to be continuously differentiable.
The Fourier variable $\xi$ is interpreted as a parameter,
and the right-hand side of the above is assumed to take values in $\mathbb{R}$ for all $\xi \in\mathbb{R}$,
so that $h^\varepsilon$ is also real-valued for all such $\xi$.

The following is a version of \cite[Lemma 3.2]{ChenLeFloch} applicable to our setting.

\begin{lemma}[Energy Estimate]\label{lemma:energy est 1}
Fix $\varepsilon > 0$. Let $h^\varepsilon(\cdot,\xi) \in C^2([\varepsilon,\nu_*])$
be the unique solution of \eqref{eq:away from vacuum remainder}.
Then there exists $C>0$ independent of $(\varepsilon, \xi, r^\varepsilon)$ such that
\begin{equation}\label{eq:first energy est}
    h^\varepsilon_\nu(\nu,\xi)^2 + k'(\nu)^2 \xi^2 h^\varepsilon(\nu,\xi)^2
    \leq C \sum_{j=1}^3 I^\varepsilon_j(\nu,\xi) \qquad
    \text{for } (\nu,\xi) \in [\varepsilon,\nu_*]\times (\mathbb{R}\setminus\{0\}),
\end{equation}
where
\begin{equation}\label{eq:I 1 to 3 def}
    \left.\begin{aligned}
    I^\varepsilon_1(\nu,\xi) &:= k'(\nu)^{-2}\xi^{-2} r^\varepsilon(\nu,\xi)^2,\\
    I^\varepsilon_2(\nu,\xi) &:= \xi^{-2} \int_\varepsilon^\nu k'(y)^{-2} r^\varepsilon(\tau,\xi)^2 \d \tau, \\
    I^\varepsilon_3(\nu,\xi) &:= \xi^{-2} \int_\varepsilon^\nu \frac{r_\nu^\varepsilon(\tau,\xi)^2}{|k'(\tau)k''(y)| + k'(\tau)^2} \d \tau.
    \end{aligned}\right.
\end{equation}
In addition,
the following estimate holds{\rm :}
\begin{equation}\label{eq:no xi dep}
    h^\varepsilon_\nu(\nu,\xi)^2 + k'(\nu)^2 \xi^2 h^\varepsilon(\nu,\xi)^2
    \leq C \nu \int_\varepsilon^\nu r^\varepsilon(\tau,\xi)^2 \d \tau \qquad \text{for } \nu \in [\varepsilon,\nu_*]
\end{equation}
for any $\xi \in \mathbb{R}$ with constant $C>0$ independent of $(\varepsilon, \xi, r^\varepsilon)$.
\end{lemma}

\begin{proof} Throughout this proof, we omit the $\varepsilon$-superscript for clarity of exposition.
We divide the proof into two steps:

\smallskip
\noindent 1.  Multiplying \eqref{eq:away from vacuum remainder} with $2h^\varepsilon_\nu$ and using the product rule,
we have
\begin{equation}\label{eq:prelim mult 2 mu nu}
    \partial_\nu \left( h_\nu^2 + k'^2 \xi^2 h^2 \right)
    = 2 r h_\nu + 2 \xi^2 k' k'' h^2 \qquad\,\, \text{for } \nu \in [\varepsilon,\nu_*].
\end{equation}
Note that the second term on the right-hand side is negative, since $k'(\nu)>0$ and $k''(\nu) < 0$ for all $\nu \in (0,\nu_*]$.
Using the Cauchy--Young inequality, we have
\begin{equation*}
    \partial_\nu \left( h_\nu^2 + k'^2 \xi^2 h^2 \right) \leq \nu^{-1}\left( h_\nu^2 + k'^2 \xi^2 h^2 \right) + \nu r^2.
\end{equation*}
Multiplying by the integrating factor $\nu^{-1}$, we obtain
\begin{equation*}
    \partial_\nu \left( \nu^{-1} \left( h_\nu^2 + k'^2 \xi^2 h^2 \right) \right) \leq r^2,
\end{equation*}
which yields \eqref{eq:no xi dep} upon integrating over interval $[\varepsilon,\nu]$.

\smallskip
\noindent
2. We now go through a refined estimate starting from \eqref{eq:prelim mult 2 mu nu}
by employing
the signed term on the right-hand side.
An integration by parts shows that
\begin{equation*}
    \int_\varepsilon^\nu r h_\nu \d \tau= r(\nu,\xi) h(\nu,\xi) - \int^\nu_\varepsilon r_\nu h \d \tau.
\end{equation*}
Then we have
\begin{equation}
    \frac{1}{2}\!\left( h_\nu(\nu,\xi)^2 + k'(\nu)^2 \xi^2 h(\nu,\xi)^2 \right)
     \! = r(\nu,\xi) h(\nu,\xi) - \!\int^\nu_\varepsilon \! r_\nu h \d \tau + \xi^2 \int_\varepsilon^\nu \! k' k'' h^2 \d \tau.
\end{equation}
Multiplying and dividing the integrand in second term on the right-hand side
by $\xi \! \left( c_1|k' k''|\! + \! k'^2 \right)^{\frac{1}{2}}$ for any $c_1 \geq 0$ and using the Cauchy--Young inequality,
we find the right-hand side is bounded above by
\begin{equation*}
    \begin{aligned}
    &\frac{1}{k'(\nu)^2 \xi^2}r(\nu,\xi)^2 + \frac{1}{4}k'(\nu)^2 \xi^2 h(\nu,\xi)^2
    + \frac{1}{2 c_2}\xi^{-2}\int_\varepsilon^\nu \left( c_1|k' k''| + k'^2 \right)^{-1} r_\nu^2 \d\tau\\
    &+ \int_\varepsilon^\nu \left( k' k'' + \frac{c_2}{2}\left( c_1|k' k''| + k'^2 \right) \right) \xi^2 h^2 \d\tau
    \end{aligned}
\end{equation*}
for any choices of $c_1 \geq 0$ and $c_2 > 0$.
Fixing $c_1 = 1$ and $c_2 = 2$, and observing that $k' k'' + |k' k''|=0$
due to the opposite signs of $k'$ and $k''$ (\textit{cf.}~Remark \ref{remark:k smooth on open}), we obtain
\begin{equation*}
    \begin{aligned}
    &\frac{1}{4}\left( h_\nu(\nu,\xi)^2 + k'(\nu)^2 \xi^2 h(\nu,\xi)^2 \right) \\
    &\leq I_1(\nu,\xi) + \frac{1}{4}\xi^{-2}\int_\varepsilon^\nu \left( |k' k''| + k'^2 \right)^{-1} r_\nu^2 \d \tau
    + \int_\varepsilon^\nu k'^2\xi^2 h^2 \d \tau.
    \end{aligned}
\end{equation*}
Define $\phi(\nu) := \frac{1}{4}\big( h_\nu(\nu,\xi)^2 + k'(\nu)^2 \xi^2 h(\nu,\xi)^2 \big)$.
Then
\begin{equation}\label{eq:take a break}
    \begin{aligned}
    \phi(\nu) \leq  I_1(\nu,\xi) + \frac{1}{4} I_3(\nu,\xi) +  4\int_\varepsilon^\nu \phi(\tau) \d \tau.
    \end{aligned}
\end{equation}
An application of Gr\"{o}nwall's inequality
yields
\begin{equation*}
    \phi(\nu) \leq \big(I_1(\nu,\xi) + \frac{1}{4}I_3(\nu,\xi)\big)
    + \int_\varepsilon^\nu \big( I_1(\tau,\xi) + \frac{1}{4}I_3(\tau,\xi) \big) e^{4\tau} \d \tau.
\end{equation*}
It follows that there exists $C_0>0$, independent of $\varepsilon$ and $\xi$,
such that
\begin{equation*}
    \phi(\nu) \leq I_1(\nu,\xi) + I_3(\nu,\xi)
    + C_0 \int_\varepsilon^\nu \big( I_1(\tau,\xi) + I_3(\tau,\xi) \big) \d \tau
    \qquad \text{for } \nu \in [\varepsilon,\nu_*].
\end{equation*}
Finally, we observe that $I_2(\nu,\xi) = \int_\varepsilon^\nu I_1(\tau,\xi) \d \tau$ and
\begin{equation*}
    \begin{aligned}
    \int_\varepsilon^\nu I_3(y,\xi) \d y &= \int_\varepsilon^\nu \xi^{-2}
    \Big( \int_\varepsilon^{\tau'} \frac{r_\nu(\tau,\xi)^2}{|k'(\tau)k''(\tau)| + k'(\tau)^2} \d \tau \Big) \d \tau' \\
    &\leq (\nu-\varepsilon) \xi^{-2}  \int_\varepsilon^\nu \frac{r_\nu(\tau,\xi)^2}{|k'(\tau)k''(\tau)| + k'(\tau)^2} \d \tau \\
    &\leq \nu_* I_3(\nu,\xi).
    \end{aligned}
\end{equation*}
Then the result follows.
\end{proof}

\subsubsection{Existence and boundedness of the remainder term for the regular kernel}\label{sec:exist rem reg}

\begin{lemma}
There exists $\hat{g}(\cdot,\xi) \in C^2([0,\nu_*])$ that is the solution of the Cauchy problem{\rm :}
\begin{equation}\label{eq:limit pb g reg}
    \left\lbrace\begin{aligned}
    & \hat{g}_{\nu\nu} + k'^2 \xi^2 \hat{g} = \hat{R}^\r \qquad &&\text{for } \nu \in [\varepsilon,\nu_*], \\
    &\hat{g}|_{\nu = 0} = 0, \\
    &\hat{g}_\nu |_{\nu=0}= 0,
    \end{aligned}
    \right.
\end{equation}
such that
the uniform estimates hold{\rm :}
    \begin{equation}\label{eq:fourier bounds remainder}
    |\partial^j_\nu \hat{g}(\nu,\xi)| \leq C\frac{\nu^{\frac{7}{3}-j}}{(1+|\xi k(\nu)|)^{4-j}}
    \qquad \mbox{for $j=0,1,2$},
    \end{equation}
where $C>0$ is some constant independent of $(\nu,\xi)$.
\end{lemma}

\begin{proof}  The proof consists of three steps.

\smallskip
\noindent
1. \textit{Uniform estimates}:
Using the bounds in \eqref{eq:ell bounds from asymp} and
the asymptotic expansions for the Bessel function contained in $\hat{f}_2$,
we obtain that the inhomogeneity on the right-hand side of \eqref{eq:remainder eqn},
$\hat{R}^\r(\nu,\xi) = A(\nu) \hat{f}_2(\xi k(\nu))$, satisfies
\begin{equation}\label{eq:big R reg est}
    |\partial^j_\nu \hat{R}^\r(\nu,\xi)| \leq \frac{C\nu^{\frac{1}{3}-j}}{(1+|\xi k(\nu)|)^{3-j}}
    \qquad \text{for } j=0,1.
\end{equation}
Recall terms $\{I_j\}_{j=1}^3$ in Lemma \ref{lemma:energy est 1},
and substitute $r^\varepsilon  = \hat{R}^\r$ for all $\varepsilon>0$ into the lemmas
in \S \ref{subsection:fourier rem}.
Then, for $\xi \neq 0$, we have
\begin{equation*}
    \begin{aligned}
        I_1 \leq C k'^{-2}\xi^{-2}\nu^{\frac{2}{3}}(1+|\xi k|)^{-6}
           \leq C\nu^2 k^2|\xi k|^{-8} \leq C\nu^{\frac{8}{3}}|\xi k|^{-8},
    \end{aligned}
\end{equation*}
while
\begin{equation*}
    I_2 \leq C\xi^{-2}\int_0^\nu \tau^2 |\xi k(\tau)|^{-6} \d \tau \leq C\xi^{-8} \nu
    \leq C\nu^{\frac{11}{3}}|\xi k|^{-8}.
\end{equation*}
Finally, we have
\begin{equation*}
    I_3 \leq C \xi^{-2} \int_0^\nu \tau |\xi k(\tau)|^{-4} \d \tau
    \leq C\xi^{-6} \int_0^\nu \tau^{-\frac{1}{3}} \d\tau \leq C\nu^{\frac{8}{3}}|\xi k|^{-6}.
\end{equation*}
It follows that there exists $C>0$ independent of $(\nu,\xi,\varepsilon)$ such that
\begin{equation}\label{eq:better est xi rem}
    |\hat{g}^\varepsilon_\nu|^2 + \xi^2 k'^2 |\hat{g}^\varepsilon|^2 \leq C\nu^{\frac{8}{3}}|\xi k|^{-6},
\end{equation}
which is accurate for
$|\xi k| \geq 1$.
Meanwhile, for $|\xi k| < 1$, the second estimate of Lemma \ref{lemma:energy est 1} implies
\begin{equation}\label{eq:worse est xi rem}
    |\hat{g}_\nu^\varepsilon|^2 + \xi^2 k'^2 |\hat{g}^\varepsilon|^2
    \leq C\nu \int_0^\nu \tau^{\frac{2}{3}}(1+|\xi k(\tau)|)^{-6} \d \tau \leq C\nu^{\frac{8}{3}},
\end{equation}
so that
$|\hat{g}^\varepsilon_\nu|\leq C\nu^{\frac{4}{3}}$ and,
by the Fundamental Theorem of Calculus, $|g^\varepsilon| \leq \nu^{\frac{7}{3}}$.
Combining \eqref{eq:better est xi rem} with \eqref{eq:worse est xi rem}, we obtain
\begin{equation*}
    |\hat{g}^\varepsilon_\nu|^2 + \xi^2 k'^2 |\hat{g}^\varepsilon|^2 \leq C\nu^{\frac{8}{3}}(1+|\xi k|)^{-6},
\end{equation*}
from which we deduce
\begin{equation*}
    |\partial^j_\nu \hat{g}^\varepsilon(\nu,\xi)| \leq \frac{C\nu^{\frac{7}{3}-j}}{(1+|\xi k(\nu)|)^{4-j}}
 \qquad \mbox{for $j=0,1$}.
\end{equation*}
In turn, using equation \eqref{eq:remainder eqn} to estimate the second derivative,
we have
\begin{equation}\label{eq:fourier bounds remainder eps}
    |\partial^j_\nu \hat{g}^\varepsilon(\nu,\xi)| \leq \frac{C\nu^{\frac{7}{3}-j}}{(1+|\xi k(\nu)|)^{4-j}}
    \qquad\mbox{for $j=0,1,2$}.
\end{equation}

\smallskip

\noindent
2. \textit{Equicontinuity of second derivatives}:
We now prove the uniform equicontinuity of the second derivatives
by employing a strategy akin to the one in \cite[Lemma 4.35]{thesis}.
For $\nu_1,\nu_2 \in [\varepsilon,\nu_*]$ with $\nu_1 \leq \nu_2$,
using equation \eqref{eq:g eps eqn} yields
\begin{align}\label{eq:expand g eps double}
 &|\hat{g}^\varepsilon_{\nu\nu}(\nu_1,\xi) - \hat{g}^\varepsilon_{\nu\nu}(\nu_2,\xi)| \nonumber\\
 &\leq \xi^2 |k'(\nu_1)^2 \hat{g}^\varepsilon(\nu_1,\xi) - k'(\nu_2)^2 \hat{g}(\nu_2,\xi)|
 + |\hat{R}^\r(\nu_1,\xi) - \hat{R}^\r(\nu_2,\xi)|.
 \end{align}
We bound both terms on the right-hand side of the previous inequality.
The first term may be rewritten and estimated as
\begin{equation*}
    \begin{aligned}
    \xi^2 \left| \int_{\nu_1}^{\nu_2} \partial_\nu\big( k'^2 \hat{g}^\varepsilon \big) \d \nu \right|
    &\leq  \xi^2 \left| \int_{\nu_1}^{\nu_2} 2 k' k'' \hat{g}^\varepsilon \d \nu \right|
       + \xi^2 \left| \int_{\nu_1}^{\nu_2} k'^2 \hat{g}^\varepsilon_\nu \d \nu \right| \\
    &\leq  C\xi^2 \int_{\nu_1}^{\nu_2} \d \nu = C\xi^2 |\nu_1 - \nu_2|,
    \end{aligned}
\end{equation*}
where we have used the uniform bounds \eqref{eq:fourier bounds remainder eps}.

Meanwhile, using estimates \eqref{eq:big R reg est} and the H\"older inequality, we have
\begin{equation*}
    \begin{aligned}
        |\hat{R}^\r(\nu_1,\xi) - \hat{R}^\r(\nu_2,\xi)|
        &\leq \int_{\nu_1}^{\nu_2}|\partial_\nu \hat{R}^\r| \d \nu \leq  C\int_{\nu_1}^{\nu_2} \nu^{-\frac{2}{3}} \d \nu \\
        &\leq  C\left( \int_0^{\nu_\cri} \nu^{-\frac{5}{6}} \d \nu \right)^{\frac{4}{5}}
         \left( \int_{\nu_1}^{\nu_2} \d \nu \right)^{1-\frac{4}{5}}
 \leq  C|\nu_1 - \nu_2|^{\frac{1}{5}}.
    \end{aligned}
\end{equation*}
Then,
returning to \eqref{eq:expand g eps double}, we deduce
\begin{equation}\label{eq:unif equi double reg}
    |\hat{g}^\varepsilon_{\nu\nu}(\nu_1,\xi) - \hat{g}^\varepsilon_{\nu\nu}(\nu_2,\xi)| \leq C_\xi |\nu_1 - \nu_2|^{\frac{1}{5}}
\end{equation}
for some constant $C_\xi>0$ depending on $\xi$, but independent of $\varepsilon$.

\smallskip
\noindent
3. \textit{Passage to the limit:}
Using the Ascoli--Arzel\`a Theorem with the uniform derivative bounds
of $\hat{g}^\varepsilon$ and $\hat{g}_\nu^\varepsilon$
and
the uniform equicontinuity estimate \eqref{eq:unif equi double reg},
we find that, for each fixed $\xi \in \mathbb{R}$, there exists
a subsequence $\{ \hat{g}^{\varepsilon'}\}_{\varepsilon'>0}$
such that $\hat{g}^{\varepsilon'}(\cdot,\xi) \to \hat{g}(\cdot,\xi)$ strongly in $C^2([0,\nu_*])$.
The limit function $\hat{g}$ is manifestly a classical solution of \eqref{eq:limit pb g reg}
and satisfies estimates \eqref{eq:fourier bounds remainder}.
\end{proof}

\subsubsection{H\"older continuity of the remainder term for the regular kernel}

\begin{lemma}\label{lem:holder remainder reg}
For any $\alpha \in [0,1)$, there exists $C_\alpha>0$ such that the remainder term satisfies
\begin{equation*}
\Vert \partial_s^{j}g(\nu,\cdot) \Vert_{C^{0,\alpha}(\mathbb{R})}
\leq C_\alpha \nu^{2-\frac{j+\alpha}{3}}, \quad \Vert \partial^{j}_\nu g(\nu,\cdot) \Vert_{C^{0,\alpha}(\mathbb{R})}
\leq C_\alpha\nu^{2-j-\frac{\alpha}{3}}
\qquad\,\, \mbox{for $j=0,1,2$},
\end{equation*}
\end{lemma}
\begin{proof}
We write
\begin{equation*}
     \Vert g(\nu,\cdot) \Vert_{C^{0,\alpha}(\mathbb{R})} \leq \int_\mathbb{R} |\xi|^\alpha |\hat{g}(\nu,\xi)| \d \xi.
\end{equation*}
It then follows from the Fourier estimates provided by \eqref{eq:fourier bounds remainder} that,
for any fixed $\alpha \in [0,1)$,
\begin{equation*}
    \Vert g(\nu,\cdot) \Vert_{C^{0,\alpha}(\mathbb{R})}
    \leq C\nu^{\frac{7}{3}}k(\nu)^{-1-\alpha} \int_\mathbb{R} \frac{|\xi|^\alpha}{(1+|\xi|)^4} \d \xi \leq C_\alpha \nu^{2-\frac{\alpha}{3}}.
\end{equation*}
An analogous strategy leads to the results for the higher derivatives.
\end{proof}

We remark from Lemma \ref{lem:holder remainder reg} that the second derivatives of $g$ are well-defined away
from the vacuum in the classical sense. The following corollary is then an immediate consequence of \eqref{eq:limit pb g reg}.

\begin{corollary}
The remainder term $g$ is the unique classical solution of the Cauchy problem:
\begin{equation}\label{eq:eqn for remainder g reg phys}
    \left\lbrace\begin{aligned}
    & g_{\nu\nu} - k'^2 g_{ss} = R^\r \qquad &&\text{for } \nu \in (0,\nu_*], \\
    &g|_{\nu = 0} = 0, \\
    &g_\nu |_{\nu=0}= 0.
    \end{aligned}
    \right.
\end{equation}
\end{corollary}

\subsection{Proof of Proposition \ref{prop:reg kernel exist}}\label{sec:initial datum reg}
With the understanding of the regular kernel obtained in \S \ref{sec:computing reg kernel coeff}--\S \ref{sec:remainder for reg},
we now complete the proof
of Proposition \ref{prop:reg kernel exist}.
First, we have

\begin{lemma}[Initial Data for the Regular Kernel]\label{lem:initial data reg} Let $H^\r$ be the regular kernel constructed
in {\rm \S\ref{sec:computing reg kernel coeff}--\S \ref{sec:asymp analysis reg kernel}}. Then
$$
  H^\r|_{\nu=0}=0, \quad H^\r_\nu|_{\nu=0}=\delta_0.
$$
\end{lemma}

\begin{proof}
Using that $\hat{f}_1,\hat{f}_2 \in L^\infty$, we estimate
\begin{equation*}
    \begin{aligned}
        \big| \hat{H}^{\r}(\nu,\xi)\big| \leq  C\big( |\alpha_0| + k^2|\alpha_1| + |\hat{g}|\big) \leq C\nu \to 0
    \end{aligned}
\end{equation*}
in the limit as $\nu \to 0^+$, which implies that $H^\r|_{\nu=0}=0$.

Similarly, using the formula: $\hat{f}_1'(\xi) = -\frac{3}{\xi}f_1(\xi) + \frac{2}{\xi}f_0(\xi)$,
abiding again by $\hat{f}$ representing for $\hat{f}(\xi k(\nu))$,
    \begin{equation}\label{eq:weak kernel nu deriv}
    \begin{aligned}
        \hat{H}^\r_\nu(\nu,\xi) = & \, 2\alpha_0 \frac{k'}{k} \hat{f}_0
        + \Big( \underbrace{\alpha_0' - 3 \alpha_0 \frac{k'}{k} + 4 \alpha_1 k' k}_{= \bigo(\nu^{\frac{2}{3}})} \Big) \hat{f}_1
        + \Big( \underbrace{\alpha_1' k^2 - 3 \alpha_1 k' k}_{=\bigo(\nu^{\frac{2}{3}})} \Big) \hat{f}_2 + \hat{g}_\nu,
    \end{aligned}
\end{equation}
where we have used
the cancellation in the asymptotic expansions of $\alpha_0$ and $\alpha_1$ to obtain the bound
on the factor premultiplying $\hat{f}_1$.
The precise value of $c_0$ can be computed from the previous line, since
\begin{equation}\label{eq:Hreg nu deriv expression}
    \hat{H}^\r_\nu(\nu,\xi) = \frac{2C_{\alpha_0} }{3} \hat{f}_0(\xi k(\nu)) + \bigo(\nu^{\frac{2}{3}}),
\end{equation}
whence, using the asymptotic relation for $\hat{f}_0$ near the origin, we have
\begin{equation*}
    \lim_{\nu \to 0^+}\hat{H}^\r_\nu(\nu,\xi) = \frac{4}{3}C_{\alpha_0}.
\end{equation*}
Thus, by letting $c_0 = \frac{3^{\frac{1}{2}}}{4}c_\sharp^{-\frac{3}{2}}$,
$\lim_{\nu \to 0^+} \hat{H}^\r_\nu(\nu,\xi) = 1$.
\end{proof}

\begin{lemma}[Huygens Principle for the Regular Kernel]\label{lem:huygens reg}
For the regular kernel $H^\r$ constructed in {\rm \S3.1--\S 3.2},
 $\, \supp H^\r \subset \mathcal{K}$.
\end{lemma}

\begin{proof}
For fixed $\nu \in (0,\nu_*]$, we see that
$\supp G_n(\nu,\cdot) = [-k(\nu),k(\nu)]$ for $n=1,2$,
while the standard theory of linear wave equations applied to \eqref{eq:eqn for remainder g reg phys}
implies that $\supp g(\nu,\cdot) \subset [-k(\nu),k(\nu)]$.
In turn, all the terms contained in expansion \eqref{eq:expansion reg kernel}
are supported inside $\mathcal{K}$.
\end{proof}

By combining the results of \S \ref{sec:computing reg kernel coeff} and \S\ref{sec:remainder for reg}
with Lemmas \ref{lem:initial data reg}--\ref{lem:huygens reg}, the proof
of Proposition \ref{prop:reg kernel exist} is complete.

\subsection{Proof of Proposition \ref{prop:est for h-1 cpct reg}}
We now prove
the following result.

\begin{lemma}\label{lem:cpct est reg}
There exists $C>0$ independent of $(\nu,\xi)$ such that,
for all $(\nu,\xi) \in [0,\nu_*]\times\mathbb{R}$,
\begin{equation*}
\big| \varrho(\nu)\hat{H}^\r_\nu(\nu,\xi) - \xi^2 \hat{H}^\r(\nu,\xi) \big| \leq C(1+\xi^2)\nu^{\frac{1}{3}}.
\end{equation*}
\end{lemma}

\begin{proof}
Using that $\hat{f}_1,\hat{f}_2 \in L^\infty$ and omitting the remainder term,
we first have
\begin{equation}\label{eq:est H xi squared reg}
    \begin{aligned}
        \big|\xi^2 \hat{H}^{\r}(\nu,\xi)\big| \leq  C\xi^2\big( |\alpha_0| + k^2|\alpha_1| + |\hat{g}|\big)
        \leq C\nu\xi^2 \leq C\nu^{\frac{1}{3}} \xi^2
    \end{aligned}
\end{equation}
for
$\nu \in [0,\nu_*]$.
Recall equation \eqref{eq:weak kernel nu deriv}, \textit{i.e.},
\begin{equation*}
    \begin{aligned}
        \hat{H}^\r_\nu(\nu,\xi) = & \, 2\alpha_0 \frac{k'}{k} \hat{f}_0
        + \Big( \underbrace{\alpha_0' - 3 \alpha_0 \frac{k'}{k} + 4 \alpha_1 k' k}_{= \bigo(\nu^{\frac{2}{3}})} \Big) \hat{f}_1
        + \Big( \underbrace{\alpha_1' k^2 - 3 \alpha_1 k' k}_{=\bigo(\nu^{\frac{2}{3}})} \Big) \hat{f}_2 + \hat{g}_\nu.
    \end{aligned}
\end{equation*}
It follows that $\big|\hat{H}^\r_\nu(\nu,\xi)\big| \leq  C$ so that
\begin{equation}\label{eq:est rho Hnu reg}
    \begin{aligned}
        \big|\varrho(\nu)\hat{H}^\r_\nu(\nu,\xi)\big| \leq  C\varrho(\nu) \leq C\nu^{\frac{1}{3}}
    \end{aligned}
\end{equation}
and
\begin{equation*}
    \big| \varrho(\nu)\hat{H}^\r_\nu(\nu,\xi) - \xi^2 \hat{H}^\r(\nu,\xi) \big| \leq C(1+\xi^2)\nu^{\frac{1}{3}},
\end{equation*}
as required.
\end{proof}

Estimate \eqref{eq:reg cpct est 1} in Proposition \ref{prop:est for h-1 cpct reg}
follows directly from Lemma \ref{lem:cpct est reg},
the asymptotic description of $\varrho$ near the vacuum, the compact support of $H^\r(\nu,\cdot)$,
and the fact that $H^\r_\nu(\nu,\cdot) \in L^1(\mathbb{R})$ for all $\nu$.
This final fact allows us to relax the assumption on the smoothness of $\varphi$,
though we do not write this explicitly.

Furthermore, estimate \eqref{eq:reg cpct est 2} in Proposition \ref{prop:est for h-1 cpct reg}
follows immediately from the estimates derived in the previous proof.
Indeed, estimate \eqref{eq:est rho Hnu reg} implies
$$
|\partial^j_s \varrho(\nu)H^\r_{\nu}(\nu,\cdot)*\varphi(s)| \leq C_\varphi
\qquad\mbox{for $j=0,1,2$},
$$
while \eqref{eq:est H xi squared reg} implies that
$|\partial^j_s H^\r_{s}(\nu,\cdot)*\varphi(s)| \leq C_\varphi$.
Then the proof of Proposition \ref{prop:est for h-1 cpct reg} is complete.

\section{Entropy Generator II: Singular Kernel}\label{sec:sing kernel}
Similarly, in this section,
we establish the existence of the singular kernel and obtain the Fourier estimates
required for the $H^{-1}_\loc$--compactness of the entropy dissipation measures.
This singular kernel requires the definition of $G_{\lambda}$ for negative integer $\lambda$; see Appendix \ref{appendix:Glambda}.

\begin{prop}[Existence of the Singular Kernel]\label{prop:sing kernel exist}
There exists a distributional solution $H^\s$, supported in $\mathcal{K}$, of the initial value problem
\begin{equation*}
    \left\lbrace\begin{aligned}
         &H^\s_{\nu\nu} - k'^2 H^\s_{ss} = 0 \qquad \text{in } (0,\nu_*] \times \mathbb{R},\\
        &H^\s|_{\nu = 0} = \delta_0, \\
        & \varrho H^\s_\nu|_{\nu=0} = -\delta_0'',
    \end{aligned}\right.
\end{equation*}
given by the formula:
\begin{equation}\label{eq:sing expansion}
    \begin{aligned}
        H^\s(\nu,s) = b_0(\nu) G_{-2}(\nu,s) + b_1(\nu)G_{-1}(\nu,s) + b_2(\nu) G_0(\nu,s) + h(\nu,s),
    \end{aligned}
\end{equation}
where there exists $C>0$ independent of $\nu$ such that
\begin{equation*}
        \nu^{-1}|b_j(\nu)| + |b_j'(\nu)| \leq C  \qquad\mbox{for $j= 0,1,2$},
\end{equation*}
and the remainder term $h(\nu,s)$ is $\alpha$-H\"older continuous with respect
to $s$ for all $\alpha \in [0,1)$.
\end{prop}

We remark that Propositions \ref{prop:reg kernel exist} and \ref{prop:sing kernel exist}
imply Theorem \ref{thm:kernels}.

\begin{prop}\label{prop:cpctness est sing}
For any $\varphi \in C^\infty(\mathbb{R})$,
there exists $C_\varphi>0$ independent of $(\nu,s)$ such that,
for all $(\nu,s)\in [0,\nu_*]\times\mathbb{R}$ and $j=0,1,2$,
\begin{align}
&\big| \partial^j_s\big(\varrho(\nu)H^\s_{\nu }(\nu,\cdot) + H^\s_{ss}(\nu,\cdot)\big)*\varphi(s)\big|
\leq C_\varphi \varrho(\nu), \label{eq:sing cpct est 1}\\[1mm]
&\big| \varrho(\nu)\partial^j_s H^\s_{\nu}(\nu,\cdot)*\varphi(s)\big| + \big| \partial^j_s H^\s_{s}(\nu,\cdot)*\varphi(s)\big|
\leq C_\varphi. \label{eq:sing cpct est 2}
\end{align}
\end{prop}

We now provide the proofs of Propositions \ref{prop:sing kernel exist}--\ref{prop:cpctness est sing}
in this section. We start with the analysis of the singular kernel coefficients.

\subsection{Computing the singular kernel coefficients}\label{sec:computing sing kernel coeff}

Similarly, we postulate an expansion of the form:
\begin{equation*}
    \hat{H}^\s(\nu,\xi) = \sum_{j=0}^2 \hat{H}^{\s,j}(\nu,\xi) + \hat{h}(\nu,\xi),
\end{equation*}
where $\hat{H}^{\s,j}(\nu,\xi) = \beta_j(\nu) k(\nu)^{2j} \hat{f}_{j-2}(\xi k(\nu))$
for $j =0,1,2$.

As in \S 3.1, we write $\hat{f}$ for $\hat{f}(\xi k(\nu))$ for simplicity.
We use the formulas provided in Lemma \ref{lem:f hat relations} to compute
\begin{equation*}
\begin{aligned}
\hat{H}^{\s,0}_{\nu\nu}& + k'^2 \xi^2 \hat{H}^{\s,0} = \beta_0'' \hat{f}_{-2}
+ \Big( \frac{1 }{2}k\big(\beta_0 k'' + 2\beta_0' k'\big) + \beta_0 k'^2 \Big)\xi^2 \hat{f}_{-1},
    \end{aligned}
\end{equation*}
and impose that
\begin{equation*}
\frac{1}{2}k\big(\beta_0 k'' + 2\beta_0' k'\big) + \beta_0 k'^2 = 0,
\end{equation*}
that is,
\begin{equation}\label{eq:beta0 def}
    \beta_0(\nu) = d_0 k(\nu)^{-1} k'(\nu)^{-\frac{1}{2}}
\end{equation}
with $d_0$ chosen such that $\lim_{\nu \to 0^+}\hat{H}^\s(\nu,\xi) = 1$; see \S \ref{sec:initial datum sing}. Then we obtain
\begin{equation}\label{eq:H0 once alpha0 is picked}
\begin{aligned}
    \hat{H}^{\s,0}_{\nu\nu}& + k'^2 \xi^2 \hat{H}^{\s,0} = \beta_0'' \hat{f}_{-2}.
    \end{aligned}
\end{equation}

Meanwhile, using the formulas provided in Lemma \ref{lem:f hat relations},
\begin{equation}\label{eq:general Hj}
\begin{aligned}
    \hat{H}^{\s,1}_{\nu\nu} + k'^2 \xi^2  \hat{H}^{\s,1} =& \,  \Big( \beta_1'' k^2 + 6 \beta_1' k' k
      + 6 \beta_1 k'^2 + 3\beta_1 k'' k \Big) \hat{f}_{-1} \\
    &-2 \Big( 4\beta_1 k'^2 + 2\beta_1' k' k +\beta_1 k'' k\Big) \hat{f}_{-2}.
    \end{aligned}
\end{equation}
We therefore set $\beta_1$ to be the solution of
\begin{equation*}
    2\beta_1' k' k +\beta_1 \big(4k'^2 + k'' k\big) = \frac{1}{2}\beta_0'',
\end{equation*}
that is,
\begin{equation*}
    \frac{\de}{\de\nu}\big(\beta_1 k^2 k'^{\frac{1}{2}}\big) = \frac{1}{4}\beta_0'' k k'^{-\frac{1}{2}},
\end{equation*}
whose closed-form
solution is
\begin{equation}\label{eq:beta1 def}
    \begin{aligned}
        \beta_1(\nu) =& \,  \frac{1}{4}k(\nu)^{-2} k'(\nu)^{-\frac{1}{2}} \int_0^\nu \beta_0''(\tau) k(\tau) k'(\tau)^{-\frac{1}{2}} \d \tau.
    \end{aligned}
\end{equation}
Then
\begin{equation*}
    \begin{aligned}
    \big(\hat{H}^{\s,0} + \hat{H}^{\s,1}\big)_{\nu\nu}+ k'^2 \xi^2 \big( \hat{H}^{\s,0} + \hat{H}^{\s,1} \big) = \ell_1 \hat{f}_{-1},
    \end{aligned}
\end{equation*}
where the coefficient is given by the exact formula:
\begin{equation}\label{eq:nasty coeff def}
\ell_1(\nu)
= \beta_1''(\nu)k(\nu)^2 + 6 \beta_1'(\nu) k'(\nu) k(\nu) + 6 \beta_1(\nu) k'(\nu)^2 + 3\beta_1(\nu) k''(\nu) k(\nu).
\end{equation}

Using again the formulas provided in Lemma \ref{lem:f hat relations}, we have
\begin{equation}\label{eq:f0 term}
\begin{aligned}
    \hat{H}^{\s,2}_{\nu\nu} + k'^2 \xi^2  \hat{H}^{\s,2}
    =& \,  \big( \beta_2'' k^4 + 6 \beta_2' k' k^3 + 6 \beta_2 k'^2 k^2 + 3\beta_2 k'' k^3 \big)\hat{f}_{0} \\
    &+ 2\big( 6\beta_2 k'^2 k^2 + 2\beta_2' k' k^3 +\beta_2 k'' k^3 \big) \hat{f}_{-1}.
    \end{aligned}
\end{equation}
Thus, we choose $\beta_2$ to be the solution of the equation:
\begin{equation*}
    2\big( 6\beta_2 k'^2 k^2 + 2\beta_2' k' k^3 +\beta_2 k'' k^3 \big) = -\ell_1,
\end{equation*}
which can be rewritten as
\begin{equation*}
    \frac{\de}{\de\nu}\big( \alpha_2 k^3 k'^{\frac{1}{2}} \big) = -\frac{1}{4}\beta_1 k'^{-\frac{1}{2}},
\end{equation*}
whose solution is given by the explicit formula
\begin{equation}\label{eq:beta2}
\beta_2(\nu)=-\frac{1}{4} k(\nu)^{-3} k'(\nu)^{-\frac{1}{2}} \int_0^\nu \ell_1(\tau) k'(\tau)^{-\frac{1}{2}} \d \tau.
\end{equation}

We define
\begin{align}
&R^\s(\nu,s) := \ell_2(\nu)k(\nu)^{-1}G_0(\nu,s), \label{eq:big Rs def}\\
&\ell_2(\nu):= -k(\nu)^2\big( \beta_2''(\nu)k(\nu)^2
      + 6 \beta_2'(\nu) k'(\nu) k(\nu) + 6 \beta_2(\nu) k'(\nu)^2+ 3\beta_2(\nu) k''(\nu) k(\nu) \big),
      \label{eq:ell2 def}
\end{align}
such that $\hat{R}^\s(\nu,\xi) = \ell_2(\nu)f_0(\xi k(\nu))$.

We therefore define the remainder term $\hat{h}$ to be the solution of the Cauchy problem:
\begin{equation}\label{eq:sing rem eqn}
    \left\lbrace\begin{aligned}
        & h_{\nu\nu} - k'^2 h_{ss}  = R^\s \qquad \text{in $(0,\nu_*]\times\mathbb{R}$}, \\
        &h|_{\nu = 0} = 0, \\
        &h_\nu|_{\nu=0} = 0.
    \end{aligned}\right.
\end{equation}

In turn, the expansion for $H^\s$ reads as
\begin{equation}
    H^\s(\nu,s) = \beta_0(\nu) k(\nu)^{3} G_{-2}(\nu,s) + \beta_1(\nu) k(\nu)^{3} G_{-1}(\nu,s) + \beta_2(\nu)^{3} G_0(\nu,s) + h(\nu,s)
\end{equation}
with
\begin{equation}\label{eq:b0 and b1 def from betas}
    b_0(\nu) := \beta_0(\nu) k(\nu)^{3}, \quad b_1(\nu) := \beta_1(\nu) k(\nu)^{3}, \quad b_2(\nu) := \beta_2(\nu) k(\nu)^3.
\end{equation}

\subsection{Asymptotic analysis of the singular kernel}\label{sec:asymp sing}

As per \S \ref{sec:asymptotics reg}, to simplify the calculations contained in the proofs that follow,
with slight abuse of notation, the shorthand $\bigo(\nu^d)$ is understood
to denote a function $\phi(\nu)$ for which the bounds $|\phi^{(j)}(\nu)|\leq C\nu^{d-j}$ hold for integer $j$,
for some positive constant $C$.

\begin{lemma}\label{lem:asymp beta0 and 1}
The coefficient functions $\beta_0(\nu)$ and $\beta_1(\nu)$ admit the following
decompositions{\rm :}
\begin{equation*}
\begin{aligned}
\beta_0(\nu) &= C_{\beta_0,0} + C_{\beta_0,1} \nu^{\frac{2}{3}} + C_{\beta_0,2} \nu^{\frac{4}{3}}
  + r_{\beta_0}(\nu), \\
\beta_1(\nu) &= C_{\beta_1,0} + C_{\beta_1,1}\nu^{\frac{2}{3}} + C_{\beta_1,2}\nu^{\frac{4}{3}}
   + r_{\beta_1}(\nu),
\end{aligned}
\end{equation*}
with
\begin{equation*}
    \begin{aligned}
        &C_{\beta_0,0} = d_0 c_\sharp^{-\frac{3}{2}} 3^{\frac{1}{2}}, \quad C_{\beta_0,1} = -\frac{5 c_\flat}{2 c_\sharp}C_{\beta_0,0}, \\
        &C_{\beta_1,0} = -\frac{C_{\beta_0,1}}{2}c_\sharp^{-2}, \quad C_{\beta_1,1}
        = -C_{\beta_1,0}\big( \frac{11 c_\flat}{3 c_\sharp} + \frac{2 C_{\beta_0,2}}{3 C_{\beta_0,1}} \big),
    \end{aligned}
\end{equation*}
where constants
$C_{\beta_0,0},C_{\beta_0,1},C_{\beta_1,0} \neq 0$,
and functions $r_{\beta_0}(\nu)$ and $r_{\beta_1}(\nu)$ satisfy that there exists some
constant $C>0$ such that
\begin{equation*}
    |r_{\beta_0}^{(j)}(\nu)| + |r_{\beta_1}^{(j)}(\nu)| \leq C\nu^{2-j}
    \qquad \text{for $j=0,1,2,3$}.
\end{equation*}
\end{lemma}

\begin{proof}
Using the asymptotic expansion for $k$, we have
\begin{equation*}
\begin{aligned}
\beta_0(\nu) =& \,  \big(c_\sharp \nu^{\frac{1}{3}} + c_\flat \nu + c_l \nu^{\frac{5}{3}}
+\bigo(\nu^{\frac{7}{3}})\big)^{-1} \big(\frac{1}{3}c_\sharp \nu^{-\frac{2}{3}} + c_\flat
+ \frac{5}{3}c_l \nu^{\frac{2}{3}}+ \bigo(\nu^{\frac{4}{3}})\big)^{-\frac{1}{2}}.
\end{aligned}
\end{equation*}
By using the Binomial Theorem, for sufficiently small $\nu$,
this can be rewritten as
\begin{equation*}
\begin{aligned}
\beta_0(\nu) &=\sqrt{3} d_0 c_\sharp^{-\frac{3}{2}}
\big( 1 - \frac{c_\flat}{c_\sharp}\nu^{\frac{2}{3}}
- \frac{c_l}{c_\sharp}\nu^{\frac{4}{3}} + \frac{c_\flat^2}{c_\sharp^2}\nu^{\frac{4}{3}} + \bigo(\nu^2) \big) \\
&\quad\, \times \big( 1 - \frac{3 c_\flat}{2 c_\sharp}\nu^{\frac{2}{3}} - \frac{5 c_l}{2 c_\sharp}\nu^{\frac{4}{3}}
+ \frac{27 c_\flat^2}{8c_\sharp^2}\nu^{\frac{4}{3}} + \bigo(\nu^2) \big),
\end{aligned}
\end{equation*}
and then the result for $\beta_0$ follows. Analogously, by writing
\begin{equation*}
        \beta_1(\nu) = -\frac{C_{\beta_0,1}}{6}c_\sharp^{-2} \nu^{-\frac{1}{3}} \Big( 1 - \frac{7 c_\flat}{2 c_\sharp} \nu^{\frac{2}{3}} + \bigo(\nu^{\frac{4}{3}}) \Big)
        \!\int_0^\nu \!\!\tau^{-\frac{2}{3}} \Big( 1 - \big(\frac{2C_{\beta_0,2}}{C_{\beta_0,1}}
          +\frac{c_\flat}{2c_\sharp}\big)\tau^{\frac{2}{3}} + \bigo(\tau^{\frac{4}{3}}) \Big) \d \tau,
\end{equation*}
the result for $\beta_1$ follows.
\end{proof}

\begin{lemma}
The following asymptotic expansion holds for $\ell_1${\rm :}
\begin{equation*}
    \ell_1(\nu) = C_{\ell_1,0} \nu^{-\frac{2}{3}} + C_{\ell_1,1} + r_{\ell_1}(\nu),
\end{equation*}
where $C_{\ell_1,0} = \frac{10}{3}C_{\beta_1,0}c_\sharp c_\flat + \frac{10}{9}C_{\beta_1,1}c_\sharp^2$,
$C_{\ell_1,1}\in\mathbb{R}$ is some constant that can be determined explicitly,
and
$r_{\ell_1}$ admits the bounds{\rm :}
\begin{equation*}
    |r^{(j)}_{\ell_1}(\nu)| \leq C\nu^{\frac{2}{3}-j} \qquad \text{for $j=0,1,2,3$},
\end{equation*}
for some constant $C>0$.
\end{lemma}

\begin{proof}
To begin with,
in view of Corollary \ref{cor:OG cancellation derivs}, we observe
\begin{equation*}
    \begin{aligned}
        6 \beta_1(\nu) k'(\nu)^2 + 3\beta_1(\nu) k''(\nu) k(\nu)
        &= 3 \beta_1(\nu) \big( 2k'(\nu)^2 + k''(\nu)k(\nu) \big) \\
        &= \frac{10}{3}C_{\beta_1,0}c_\sharp c_\flat \nu^{-\frac{2}{3}} + C + \bigo(\nu^\frac{2}{3})
    \end{aligned}
\end{equation*}
for some constant $C$.
Meanwhile,
\begin{equation*}
        \beta_1''(\nu)k(\nu)^2 + 6 \beta_1'(\nu) k'(\nu) k(\nu)
        = \frac{10}{9}C_{\beta_1,1}c_\sharp^2 \nu^{-\frac{2}{3}} + C + \bigo(\nu^{\frac{2}{3}})
\end{equation*}
for some constant $C$.
The result follows.
\end{proof}

\begin{lemma}\label{lem:alpha2 asymp}
The following asymptotic expansion holds for $\beta_2${\rm :}
\begin{equation*}
    \beta_2(\nu) = C_{\beta_2,0} + C_{\beta_2,1}\nu^{\frac{2}{3}}+ r_{\beta_2}(\nu),
\end{equation*}
where $C_{\beta_2,0} = -\frac{9}{8}c_\sharp^{-4}C_{\ell_1,0}$,
$C_{\beta_2,1} \in \mathbb{R}$ is some constant that
can be determined explicitly,
and $r_{\beta_2}$ admits the bounds{\rm :}
\begin{equation*}
    |r_{\beta_2}^{(j)}(\nu)| \leq C\nu^{\frac{4}{3}-j} \qquad \text{for $j =0,1,2,3$},
\end{equation*}
for some constant $C>0$.
\end{lemma}

\begin{proof}
Using the previous asymptotic relations, we obtain
\begin{equation*}
    \begin{aligned}
\beta_2(\nu)
        &= -\frac{3}{4}c_\sharp^{-4}C_{\ell_1,0}\nu^{-\frac{2}{3}}\big( 1 - \frac{3c_\flat}{c_\sharp}\nu^{\frac{2}{3}}
         + \bigo(\nu^{\frac{4}{3}}) \big) \big(  1 - \frac{3c_\flat}{2c_\sharp}\nu^{\frac{2}{3}} + \bigo(\nu^{\frac{4}{3}}) \big)\\
        &\quad\,\, \times\int_0^\nu \tau^{-\frac{1}{3}} \big( 1 +C_1\tau^{\frac{2}{3}} + \bigo(\tau^{\frac{4}{3}}) \big) \d \tau \\
        &=-\frac{9}{8}c_\sharp^{-4}C_{\ell_1,0} \big(  1 + C_2\nu^{\frac{2}{3}} + \bigo(\nu^{\frac{4}{3}}) \big)
    \end{aligned}
\end{equation*}
for some constants $C_1$ and $C_2$ that can be determined explicitly. Then the result follows.
\end{proof}

\begin{corollary}\label{cor:beta2 bounds}
There exist constants $C_{\ell_2,0}$ and $C>0$,
and a function $r_{\ell_2}$ such that
\begin{equation*}
    \ell_2(\nu) = C_{\ell_2,0} + r_{\ell_2}(\nu)
\end{equation*}
with
\begin{equation*}
    |r_{\ell_2}^{(j)}(\nu)| \leq C\nu^{\frac{2}{3}-j} \qquad \text{for $j = 0,1$}.
\end{equation*}
In turn, $\ell_2$ satisfies the bound{\rm :}
\begin{equation*}
    |\ell_2(\nu)| + \nu^{\frac{1}{3}}|\ell_2'(\nu)| \leq C
\end{equation*}
for some constant $C>0$.
\end{corollary}

\begin{proof}
In view of Corollary \ref{cor:OG cancellation derivs} and Lemma \ref{lem:alpha2 asymp}, we have
\begin{equation*}
    \begin{aligned}
        &6 \beta_2(\nu) k'(\nu)^2 k(\nu)^2 + 3\beta_2(\nu) k''(\nu) k(\nu)^3 \\
        &= 3 \beta_2(\nu) k(\nu)^2 \big( 2 k'(\nu)^2 + k''(\nu) k(\nu) \big) \\
        &= \frac{4}{3}C_{\beta_2,0} c_\sharp^3 \big(c_\flat + (2\frac{c_\flat}{c_\sharp}+\frac{C_{\beta_2,1}}{C_{\beta_2,0}})c_\flat
        + \frac{9}{2}c_l)\nu^{\frac{2}{3}} + \bigo(\nu^{\frac{4}{3}}) \big).
    \end{aligned}
\end{equation*}
Meanwhile,
\begin{equation*}
    \begin{aligned}
        &\beta_2''(\nu)k(\nu)^4 + 6 \beta_2'(\nu) k'(\nu) k(\nu)^3 \\
        &=-\frac{2}{9}C_{\beta_2,1}c_\sharp^4 \big(1 + \nu^{\frac{4}{3}}r''_{\beta_2}(\nu) \big) \big( 1 + \frac{4c_\flat}{c_\sharp}\nu^{\frac{2}{3}}
         + \bigo(\nu^{\frac{4}{3}})\big) \\
        &\quad\,\, + \frac{4}{3}C_{\beta_2,1}c_\sharp^4 \big( 1 + \nu^{\frac{1}{3}}r'_{\beta_2}(\nu) \big)
         \big( 1 + \frac{6c_\flat}{c_\sharp}\nu^{\frac{2}{3}} + \bigo(\nu^{\frac{4}{3}}) \big),
    \end{aligned}
\end{equation*}
so that
\begin{equation*}
    \begin{aligned}
    &\beta_2''(\nu)k(\nu)^4 + 6 \beta_2'(\nu) k'(\nu) k(\nu)^3 \\
    &= \frac{10}{9}C_{\beta_2,1}c_\sharp^4 + \frac{2}{9}C_{\beta_2,1}c_\sharp^4 \nu^{\frac{1}{3}} \big( r_{\beta_2}''(\nu) \nu + 6 r_{\beta_2}'(\nu)  \big)
    +\frac{8}{9}C_{\beta_2,1}c_\sharp^3c_\flat \nu\big( r_{\beta_2}''(\nu) \nu + 9 r_{\beta_2}'(\nu)  \big) \\
        &\quad\,\, + \bigo(\nu^{\frac{5}{3}}) \big(  r_{\beta_2}''(\nu) \nu +  r_{\beta_2}'(\nu) \big).
    \end{aligned}
\end{equation*}
It follows from Lemma \ref{lem:alpha2 asymp} that
\begin{equation*}
    \begin{aligned}
       &\beta_2''(\nu)k(\nu)^4 + 6 \beta_2'(\nu) k'(\nu) k(\nu)^3 \\
       &= \frac{10}{9}C_{\beta_2,1}c_\sharp^4 + \frac{2}{9}C_{\beta_2,1}c_\sharp^4 \nu^{\frac{1}{3}} \big( r_{\beta_2}''(\nu) \nu + 6 r_{\beta_2}'(\nu)  \big)
       +\frac{8}{9}C_{\beta_2,1}c_\sharp^3c_\flat \nu\big( r_{\beta_2}''(\nu) \nu + 9 r_{\beta_2}'(\nu)  \big) \\
        &\quad\,\, + \bigo(\nu^{\frac{5}{3}}) \big(  r_{\beta_2}''(\nu) \nu +  r_{\beta_2}'(\nu) \big),
    \end{aligned}
\end{equation*}
whence the result follows.
\end{proof}

\subsection{Analysis of the remainder term for the singular kernel}
We now analyze the remainder term $h(\nu,s)$ for the singular kernel.

\subsubsection{Existence and boundedness of the remainder term for the singular kernel}\label{sec:remainder for sing}

As in \S \ref{sec:exist rem reg}, we show the existence of $h(\nu,s)$ solving the Cauchy problem \eqref{eq:sing rem eqn}
by studying the limit of solutions $\{h^\varepsilon\}_{\varepsilon>0}$ of the translated problems:
\begin{equation}\label{eq:approx rem reg eqn}
 \left\lbrace \begin{aligned}
 & \hat{h}^\varepsilon + k'^2 \xi^2 \hat{h}^\varepsilon = \hat{R}^\s \qquad\,\, \text{for $\nu\in [\varepsilon,\nu_*]$}, \\
 &\hat{h}^\varepsilon|_{\nu=\varepsilon} = 0, \\
 &\hat{h}^\varepsilon_\nu|_{\nu=\varepsilon} = 0.
 \end{aligned}  \right.
\end{equation}

Once again, we stress that the $\varepsilon$-superscript
bears no relation to the sequence
of approximate problems throughout this subsection.

\begin{lemma}
 There exists $\hat{h}(\cdot,\xi) \in C^2([0,\nu_*])$ satisfying the Cauchy problem:
\begin{equation}\label{eq:limit pb rem sing}
    \left\lbrace\begin{aligned}
    & \hat{h}_{\nu\nu} + k'^2 \xi^2 \hat{h} = \hat{R}^\s \qquad\,\, \text{for } \nu \in [\varepsilon,\nu_*], \\
    &\hat{h}|_{\nu = 0} = 0, \\
    &\hat{h}_\nu |_{\nu=0}= 0,
    \end{aligned}
    \right.
\end{equation}
such that
the uniform estimates hold{\rm :}
\begin{equation}\label{eq:fourier bounds remainder sing}
    |\partial^j_\nu \hat{h}(\nu,\xi)| \leq C \frac{\nu^{2-j}}{(1+|\xi k(\nu)|)^{2-j}}, \qquad j=0,1,2,
    \end{equation}
where $C>0$ is some constant independent of $(\nu,\xi)$.
\end{lemma}

\begin{proof}  The proof consists of three steps.

\smallskip
\noindent
1. \textit{Uniform estimates}:
 Using the explicit formulas of Lemmas \ref{lem:G hat in terms of f hat and f hat precise} and \ref{lem:f hat relations},
 and the estimate of $\ell_2$ provided by Corollary \ref{cor:beta2 bounds}, we deduce that the forcing term $\hat{R}^\s$ satisfies
    \begin{equation}
        |\partial^j_\nu \hat{R}^\s(\nu,\xi)| \leq C \frac{\nu^{-j}}{(1+|\xi k(\nu)|)^{1-j}} \qquad \mbox{for $j=0,1$},
    \end{equation}
for some constant $C>0$ independent of $(\nu,\xi,\varepsilon)$.

Recall terms $\{I_j^\varepsilon\}_{j=1}^3$ in Lemma \ref{lemma:energy est 1}
and substitute $r^\varepsilon = \hat{R}^\s$ for all $\varepsilon>0$.
Using the asymptotic description of $k(\nu)$, we obtain that, for $\xi \neq 0$,
    \begin{equation*}
    \begin{aligned}
    |I_1^\varepsilon(\nu,\xi)| \leq C \nu^{\frac{4}{3}} \xi^{-2}(1+|\xi k(\nu)|)^{-2} \leq C\nu^{2} |\xi k(\nu)|^{-4}.
    \end{aligned}
\end{equation*}
Similarly,
\begin{align*}
&|I_2^\varepsilon(\nu,\xi)| \leq C \xi^{-2} \int_\varepsilon^\nu \tau^{\frac{4}{3}}(1+|\xi k(\tau)|)^{-2} \d \tau
\leq C\nu^3 |\xi k(\nu)|^{-4}, \\
&|I_3^\varepsilon(\nu,\xi)| \leq C \xi^{-2} \int_\varepsilon^\nu \tau^{-2+\frac{7}{3}} \d \tau
\leq C\xi^{-2} \nu^{\frac{4}{3}} \leq C\nu^2 |\xi k(\nu)|^{-2}.
\end{align*}
By Lemma \ref{lemma:energy est 1}, we see that, for some constant $C>0$ independent of $(\nu,\xi,\varepsilon)$,
\begin{equation}\label{eq:checkpoint 1 energy est 2}
    |\hat{h}^\varepsilon_\nu|^2 + k'^2 \xi^2 |h^\varepsilon|^2 \leq C\nu^2 |\xi k|^{-2},
\end{equation}
which is accurate for $|\xi k| \geq 1$.
On the other hand, for $|\xi k|<1$, we use the second estimate in Lemma \ref{lemma:energy est 1} to obtain
\begin{equation}\label{eq:checkpoint 2 energy est 2}
    |\hat{h}^\varepsilon_\nu|^2 + k'^2 \xi^2 |h^\varepsilon|^2 \leq C\nu^2,
\end{equation}
so that $|\hat{h}^\varepsilon_\nu| \leq C\nu$,
which implies that
$|\hat{h}^\varepsilon| \leq C\nu^2$.
In turn, by combining \eqref{eq:checkpoint 1 energy est 2} with \eqref{eq:checkpoint 2 energy est 2},
we have
\begin{equation*}
    |\hat{h}^\varepsilon_\nu|^2 + k'^2 \xi^2 |h^\varepsilon|^2 \leq C\nu^2(1+|\xi k|)^{-2},
\end{equation*}
so that
\begin{equation*}
    |\partial^j_\nu \hat{h}^\varepsilon(\nu,\xi)| \leq C \frac{\nu^{2-j}}{(1+|\xi k(\nu)|)^{2-j}}
    \qquad\mbox{for $j=0,1$}.
\end{equation*}
Finally, using equation \eqref{eq:approx rem reg eqn} to estimate the second derivative,
we obtain
\begin{equation}\label{eq:unif eps indep deriv bounds h}
    |\partial^j_\nu \hat{h}^\varepsilon(\nu,\xi)| \leq C \frac{\nu^{2-j}}{(1+|\xi k(\nu)|)^{2-j}}
    \qquad \mbox{for $j=0,1,2$},
\end{equation}
where $C>0$ is independent of $(\nu,\xi,\varepsilon)$.

\smallskip
\noindent
2. \textit{Equicontinuity of second derivatives}:
For $\nu_1,\nu_2 \in [\varepsilon,\nu_*]$ with $\nu_1 \leq \nu_2$,
we use equation \eqref{eq:g eps eqn} to obtain
\begin{equation*}
    \begin{aligned}
        &|\hat{h}^\varepsilon_{\nu\nu}(\nu_1,\xi) - \hat{h}^\varepsilon_{\nu\nu}(\nu_2,\xi)| \\
        &\leq \xi^2 |k'(\nu_1)^2 \hat{h}^\varepsilon(\nu_1,\xi) - k'(\nu_2)^2 \hat{h}(\nu_2,\xi)|
        + |\hat{R}^\s(\nu_1,\xi) - \hat{R}^\s(\nu_2,\xi)|.
    \end{aligned}
\end{equation*}
We now bound both terms on the right-hand side.
The first term may be rewritten and estimated as
\begin{equation*}
    \begin{aligned}
    \xi^2 \bigg| \int_{\nu_1}^{\nu_2} \partial_\nu\big( k'^2 \hat{h}^\varepsilon \big) \d \nu \bigg|
    &\leq  \xi^2 \bigg| \int_{\nu_1}^{\nu_2} 2 k' k'' \hat{h}^\varepsilon \d \nu \bigg|
      + \xi^2 \bigg| \int_{\nu_1}^{\nu_2} k'^2 \hat{h}^\varepsilon_\nu \d \nu \bigg| \\
    &\leq  C\xi^2 \int_{\nu_1}^{\nu_2} \nu^{-\frac{1}{3}} \d \nu.
    \end{aligned}
\end{equation*}
For the second term, we use $\hat{R}^\r(\nu,\xi) = \ell_2(\nu)\hat{f}_0(\xi k(\nu))$,
the connection formula $\hat{f}_0'(\xi) = -\frac{\xi}{2} \hat{f}_1(\xi)$
(\textit{cf.}~Lemma \ref{lem:f hat relations}), and Corollary \ref{cor:beta2 bounds} to write
\begin{equation*}
    \begin{aligned}
        |\hat{R}^\s(\nu_1,\xi) - \hat{R}^\s(\nu_2,\xi)|
        &= \bigg| \int_{\nu_1}^{\nu_2} \frac{\de}{\de\nu}\big( \ell_2(\nu) \hat{f}_{0}(\xi k(\nu)) \big) \d \nu \bigg| \\
        &\leq  \bigg|\int_{\nu_1}^{\nu_2} \ell_2'(\nu) \hat{f}_{0}(\xi k(\nu)) \d \nu\bigg|
         + \bigg| \int_{\nu_1}^{\nu_2} \ell_2(\nu) \xi k'(\nu) \hat{f}_{0}'(\xi k(\nu)) \d \nu \bigg| \\
        &\leq  C\int_{\nu_1}^{\nu_2} \nu^{-\frac{1}{3}} \d \nu
        + C\xi^2\int_{\nu_1}^{\nu_2} k'(\nu)k(\nu) |\hat{f}_{1}(\xi k(\nu))| \d \nu.
    \end{aligned}
\end{equation*}
Employing $\hat{f}_1 \in L^\infty(\mathbb{R})$ and the H\"older inequality, we deduce
 \begin{align}
        |\hat{h}^\varepsilon_{\nu\nu}(\nu_1,\xi) - \hat{h}^\varepsilon_{\nu\nu}(\nu_2,\xi)|
        &\leq C(1+\xi^2)\int_{\nu_1}^{\nu_2} \nu^{-\frac{1}{3}} \d \nu \nonumber \\
        &\leq C(1+\xi^2)\bigg( \int_{\nu_1}^{\nu_2} \nu^{-\frac{2}{3}} \d \nu \bigg)^{\frac{1}{2}} \bigg( \int_{\nu_1}^{\nu_2} \d \nu \bigg)^{\frac{1}{2}} \nonumber\\
        &\leq C(1+\xi^2)|\nu_1 - \nu_2|^{\frac{1}{2}}, \label{eq:equi h double int form}
    \end{align}
where $C>0$ is independent of $(\nu,\xi, \varepsilon)$.

\smallskip
\noindent
3. \textit{Passage to the limit}:
Using the uniform equicontinuity estimate \eqref{eq:equi h double int form},
we invoke the Ascoli--Arzel\`a Theorem to obtain that,
for each fixed $\xi \in \mathbb{R}$, there exists a subsequence $\{ \hat{h}^{\varepsilon'}\}_{\varepsilon'>0}$
such that $\hat{h}^{\varepsilon'}(\cdot,\xi) \to \hat{h}(\cdot,\xi)$ strongly in $C^2([0,\nu_*])$.
The strong convergence ensures that the limit function $\hat{h}$ is a classical solution
of \eqref{eq:limit pb rem sing} and satisfies the estimates in \eqref{eq:fourier bounds remainder sing}.
\end{proof}

\subsubsection{H\"older continuity of remainder term for the singular kernel}
\begin{lemma}\label{lem:holder remainder sing}
For any $\alpha \in [0,1)$, there exists $C_\alpha>0$ such that the remainder term $h(\nu,s)$ satisfies
    \begin{equation*}
        \Vert h(\nu,\cdot) \Vert_{C^{0,\alpha}(\mathbb{R})} \leq C_\alpha \nu^{\frac{5-\alpha}{3}}.
    \end{equation*}
\end{lemma}

\begin{proof}
Using estimates \eqref{eq:fourier bounds remainder sing}, we have
\begin{equation*}
     \Vert h(\nu,\cdot) \Vert_{C^{0,\alpha}(\mathbb{R})} \leq \int_\mathbb{R} |\xi|^\alpha |\hat{h}(\nu,\xi)| \d \xi
\leq C\nu^2 k(\nu)^{-1-\alpha} \int_\mathbb{R} \frac{|\xi|^\alpha}{(1+|\xi|)^2} \d \xi.
\end{equation*}
Then the result follows.
\end{proof}

It follows from the previous result that, since $\hat{h} \in L^1(\mathbb{R})$
satisfies \eqref{eq:limit pb rem sing}, the remainder term $h(\nu,s)$
is a distributional solution of the Cauchy problem \eqref{eq:sing rem eqn}.

\subsection{Proof of Proposition \ref{prop:sing kernel exist}}\label{sec:initial datum sing}

We first observe
the initial data of $H^\s$ at $\nu=0$,
the proof of which follows directly from the asymptotic description of
coefficients $\{\beta_j\}_{j=0}^2$ in Lemma \ref{lem:asymp beta0 and 1}.

\begin{lemma}[Initial Data for the Singular Kernel]\label{lem:initial data sing ker no deriv}
$\, H^\s|_{\nu=0} = \delta_0$.
\end{lemma}

\begin{proof}
Notice that
\begin{equation*}
    \hat{H}^\s(\nu,\xi) = \beta_0 \hat{f}_{-2} + \underbrace{\beta_1 k^2}_{=\bigo(\nu^{\frac{2}{3}})} \hat{f}_{-1}
    + \underbrace{\beta_2 k^4}_{=\bigo(\nu^{\frac{4}{3}})} \hat{f}_0 + \hat{g}.
\end{equation*}
Since $\hat{f}_{-1},\hat{f}_0 \in L^\infty(\mathbb{R})$
by Lemma \ref{lem:G hat in terms of f hat and f hat precise},
\begin{equation*}
    \hat{H}^\s(\nu,\xi) = C_{\beta_0,0} \hat{f}_{-2}(\xi k(\nu)) + \bigo(\nu^{\frac{2}{3}}).
\end{equation*}
Then, using the precise form of $\hat{f}_{-2}$ provided by Lemma \ref{lem:G hat in terms of f hat and f hat precise},
we have
\begin{equation*}
    \big|\hat{H}^\s(\nu,\xi) -  \frac{1}{2}C_{\beta_0,0} \big| \leq C|\xi k(\nu)|^2 + \bigo(\nu^{\frac{2}{3}}).
\end{equation*}
Using the expression for $C_{\beta_0,0}$ in Lemma \ref{lem:asymp beta0 and 1} with
$d_0 = 3^{-\frac{1}{2}}c_\sharp^{\frac{3}{2}}$, we conclude
\begin{equation*}
    \lim_{\nu \to 0^+} \hat{H}^\s(\nu,\xi) = \frac{1}{2}C_{\beta_0,0} = 1,
\end{equation*}
where the limit is taken pointwise for each fixed $\xi \in \mathbb{R}$.
\end{proof}

The proof that $\varrho H^\s_\nu|_{\nu=0} = -\delta_0''$ is contained in \S \ref{sec:cpctness est sing}.

\smallskip
\begin{prop}[Huygens Principle for the Singular Kernel]\label{lem:huygens sing}
$\,\,\supp H^\s \subset \mathcal{K}$.
\end{prop}

\begin{proof}
It is clear that $\supp G_{n}(\nu,\cdot) = [-k(\nu),k(\nu)]$ when $n=0,1,2$.
In addition, Lemma \ref{lem:G0 to G negatives} shows that $\supp G_n(\nu,\cdot) \subset [-k(\nu),k(\nu)]$
for $n=-3,-2,-1$.
By applying the standard theory of the linear wave equation to the Cauchy problem \eqref{eq:sing rem eqn},
it follows that $\supp h(\nu,\cdot) \subset [-k(\nu),k(\nu)]$ for each fixed $\nu>0$.

It follows that all the terms in formula \eqref{eq:sing expansion} are supported
inside $[-k(\nu),k(\nu)]$, whence the result follows.
\end{proof}

By combining the results of \S\ref{sec:computing sing kernel coeff} and \S\ref{sec:remainder for sing}
with Lemmas \ref{lem:initial data sing ker no deriv}--\ref{lem:huygens sing},
the proof of Proposition \ref{prop:sing kernel exist} is complete.

\subsection{Proof of Proposition \ref{prop:cpctness est sing}}\label{sec:cpctness est sing}
As $H^\s(\nu,\cdot) \in \mathcal{E}'(\mathbb{R})$,
we see that the convolution:
$H^\s(\nu,\cdot)*\varphi(s) = \langle H^\s(\nu,\cdot), \varphi(s-\cdot) \rangle$
defines a smooth function of $s$ for each $\varphi \in C^\infty(\mathbb{R})$.
Moreover, this function satisfies the entropy generator equation \eqref{eq:ent gen intro}.
In view of $\mathscr{F}\big(H^\s(\nu,\cdot)*\varphi\big)(\xi) = \hat{H}^\s(\nu,\xi)\hat{\varphi}(\xi)$,
the result that follows confirms
the first statement \eqref{eq:sing cpct est 1}
in Proposition \ref{prop:cpctness est sing}.

\begin{lemma}\label{lemma:kernel compactness estimates fourier}
There exists $C>0$ independent of $(\nu,\xi)$ such that,
for all $(\nu,\xi) \in [0,\nu_*]\times\mathbb{R}$,
\begin{equation}
    |\varrho(\nu)\hat{H}^\s_\nu(\nu,\xi) - \xi^2 \hat{H}^\s(\nu,\xi)| \leq C(1+\xi^4)\nu^{\frac{2}{3}}.
\end{equation}
\end{lemma}

\begin{proof}
To simplify notation, in some parts of this proof,
the functions that depend solely on $\nu$ are written without their argument stated explicitly.

Observe that, in view of Lemma \ref{lem:G hat in terms of f hat and f hat precise},
$\hat{f}_{-2} \in L^\infty_{\loc}(\mathbb{R})$
and $\hat{f}_{-1}, \hat{f}_0, \hat{f}_1 \in L^\infty(\mathbb{R})$.
Combining this with the asymptotic expansions for coefficients $\{\beta_j\}_{j=0}^2$
with \eqref{eq:fourier bounds remainder sing} shows that there exists $C>0$ independent
of $(\nu,\xi)$ such that
\begin{equation*}
\big| \beta_1 k^{2} \hat{f}_{-1} (\xi k(\nu)) + \beta_2 k^{4} \hat{f}_{0} (\xi k(\nu))
+ \hat{h}(\nu,\xi) \big| \leq {C\nu^{\frac{2}{3}}}
\end{equation*}
for all $(\nu,\xi)\in [0,\nu_*]\times\mathbb{R}$.
Similarly, using also the relation: $\hat{f}_{-1}'(\xi) = -\frac{\xi}{2}\hat{f}_0(\xi)$
(\textit{cf.}~Lemmas \ref{lem:G hat in terms of f hat and f hat precise} and \ref{lem:f hat relations}),
we have
\begin{equation*}
    \Big| \partial_\nu \Big( \beta_1 k^{2} \hat{f}_{-1} (\xi k(\nu)) + \beta_2 k^{4} \hat{f}_{0} (\xi k(\nu))
    + \hat{h}(\nu,\xi) \Big) - 2\beta_1kk'\hat{f}_{-1}(\xi k(\nu)) \Big|
    \leq C(1+\xi^2)\nu^{\frac{1}{3}}.
\end{equation*}
Therefore, it follows that there exists $C>0$ independent of $(\nu,\xi)$ such that,
for all $(\nu,\xi)\in [0,\nu_*]\times\mathbb{R}$,
\begin{align}
&|\hat{H}^\s(\nu,\xi) - \beta_0\hat{f}_{-2}(\xi k(\nu))| \leq {C\nu^{\frac{2}{3}}}, \label{eq:close approx Hhat}\\
&\big| \hat{H}^\s_\nu(\nu,\xi) - \partial_\nu\big( \beta_0 \hat{f}_{-2}(\xi k(\nu)) \big)
- 2\beta_1k k' \hat{f}_{-1}(\xi k(\nu)) \big| \leq {C(1+\xi^2)\nu^{\frac{1}{3}}}. \label{eq:close approx Hhat deriv}
\end{align}
Then we have
\begin{align}
       &|\varrho(\nu)\hat{H}^\s_\nu(\nu,\xi) - \xi^2 \hat{H}^\s(\nu,\xi)| \nonumber \\
        &\leq  \varrho(\nu)\big| \beta_0'(\nu)\hat{f}_{-2}(\xi k(\nu)) + 2\beta_1(\nu) k(\nu) k'(\nu) \hat{f}_{-1}(\xi k(\nu)) \big| \nonumber\\
        &\quad +\! C\big|\varrho(\nu) \partial_\nu (\hat{f}_{-2}(\xi k(\nu)))\! -\! \xi^2 \hat{f}_{-2}(\xi k(\nu))\big|
       \! + \! C(1+\xi^2)\nu^{\frac{2}{3}} \nonumber \\
        &=: |K_1(\nu,\xi)|+ C|K_2(\nu,\xi)| + C(1+\xi^2)\nu^{\frac{2}{3}}, \label{eq:some of the way there fourier compact kernel}
    \end{align}
where we have used the asymptotic expansions to bound $\beta_0$ by a constant
in the first line, as well as the estimate for $\varrho(\nu)$.
It remains to estimate terms $K_1$ and $K_2$.

We begin with $K_2$, which can be rewritten as
\begin{equation*}
\begin{aligned}
K_2(\nu,\xi)
&=\xi^2\Big( \frac{1}{2}\varrho(\nu)k'(\nu)k(\nu) \hat{f}_{-1}(\xi k(\nu)) - \hat{f}_{-2}(\xi k(\nu)) \Big) \\
&= \frac{1}{2}\xi^2 \Big( \big(\varrho(\nu)k'(\nu)k(\nu) - 1 \big) \cos(\xi k(\nu)) - \xi^2 k(\nu)^2\frac{\sin (\xi k(\nu))}{\xi k(\nu)} \Big).
\end{aligned}
\end{equation*}
Using the asymptotic expansions provided by Proposition \ref{prop:k expand}
for $k$ and $\varrho$ (\textit{cf.}~\eqref{eq:rho expansion precise}
and the exact value for $c_\sharp$), we obtain
\begin{equation*}
    \varrho(\nu) k'(\nu) k(\nu) = \frac{1}{3}\varrho(\nu) c_\sharp^2\nu^{-\frac{1}{3}} + \bigo(\nu^{\frac{2}{3}}) = 1 + \bigo(\nu^{\frac{2}{3}}),
\end{equation*}
whence we obtain that $|K_2(\nu,\xi)| \leq C\xi^2( 1+\xi^2) \nu^{\frac{2}{3}} \leq C(1+\xi^4)\nu^{\frac{2}{3}}$
for some constant $C>0$ independent of $(\nu,\xi)$.

The estimate for $K_1$ is similar. Notice that
\begin{equation*}
    K_1(\nu,\xi) = \varrho(\nu) \Big( \frac{1}{2}\beta_0'(\nu) \xi^2 k(\nu)^2 \frac{\sin (\xi k(\nu))}{\xi k(\nu)}
     + \big( \frac{1}{2}\beta_0'(\nu) + 2\beta_1(\nu)k(\nu)k'(\nu) \big) \cos(\xi k(\nu)) \Big),
\end{equation*}
and $|\varrho(\nu) \beta_0'(\nu) \xi^2 k(\nu)^2| \leq C\xi^2 \nu^{\frac{2}{3}}$.
On the other hand, using the asymptotic expansions for $\beta_0$ and $\beta_1$, we have
\begin{equation*}
    \frac{1}{2}\beta_0'(\nu) + 2\beta_1(\nu)k(\nu)k'(\nu) = \big( \underbrace{\frac{1}{3}C_{\beta_0,1}
    + \frac{2}{3}C_{\beta_1,0}c_\sharp^2}_{=0} \big)\nu^{-\frac{1}{3}} + \bigo(\nu^{\frac{1}{3}}) = \bigo(\nu^{\frac{1}{3}}),
\end{equation*}
whence $|K_2(\nu,\xi)| \leq C(1+\xi^2)\nu^{\frac{2}{3}}$ for some constant $C>0$ independent of $(\nu,\xi)$.
The proof is complete.
\end{proof}

\smallskip
Note that the second estimate \eqref{eq:sing cpct est 2} in Proposition \ref{prop:cpctness est sing}
is obtained directly from the previous calculations.
Indeed, it follows from the asymptotic expansions in \S \ref{sec:asymp sing},
equation \eqref{eq:close approx Hhat}, and
the bound: $|\hat{f}_{-2}(\xi)| \leq C(1+\xi^2)$
(\textit{cf.}~Lemma \ref{lem:G hat in terms of f hat and f hat precise})
that
$$
|\hat{H}^\s(\nu,\xi)| \leq C(1+\xi^2)
$$
for some constant $C>0$ independent of $(\nu,\xi)$.
Meanwhile, \eqref{eq:close approx Hhat deriv}
and $\hat{f}_{-2}'(\xi) = \frac{\xi}{2}\hat{f}_{-1}(\xi)$
imply that $|\varrho(\nu)\hat{H}^\s_\nu(\nu,\xi)| \leq C(1+\xi^2)$.
Estimate \eqref{eq:sing cpct est 2} follows directly from elementary properties
of the Fourier transform, which concludes the proof of Proposition \ref{prop:cpctness est sing}.

Finally, we show that the initial condition for $\varrho H^\s_\nu$ is achieved.
This is an immediate consequence of Lemma \ref{lemma:kernel compactness estimates fourier}.

\begin{corollary}\label{cor:initial data rho Hnu is achieved}
$\,\,\varrho H^\s_\nu|_{\nu=0} = -\delta_0''$.
\end{corollary}
\begin{proof}
    Taking the pointwise limit as $\nu \to 0^+$
    in the estimate in Lemma \ref{lemma:kernel compactness estimates fourier}
    and using the initial data for $H^\s|_{\nu=0}$ proved
    in Lemma \ref{lem:initial data sing ker no deriv} in the Fourier coordinates, we obtain
    \begin{equation*}
        \lim_{\nu \to 0^+} \varrho(\nu) \hat{H}^\s_\nu(\nu,\xi) = \xi^2
        \qquad \mbox{for each fixed $\xi \in \mathbb{R}$}.
    \end{equation*}
    Then the result follows.
\end{proof}

\section{Compensated Compactness Framework}\label{sec:comp comp framework and red}

This section is devoted to the proof of
Theorem \ref{thm:compensated compactness framework intro},
based on our detailed entropy analysis in \S3--\S5.
In doing so, we establish a compactness framework
that leads to the strong convergence of the approximate or exact solutions
to an entropy solution of the limit problem
for system \eqref{eq:system in intro}--\eqref{eq:bernoulli general gamma intro};
see \S \ref{sec:entropy sol}.

\begin{proof}[Proof of Theorem \ref{thm:compensated compactness framework intro}]
The uniform bounds in \eqref{5.1a}--\eqref{5.2a} imply the existence of a subsequence (still denoted)
$\{\bu^\varepsilon\}_{\varepsilon>0}$ converging weakly-* to a function $\bu
\in L^\infty$ and the existence of a Young measure $\{\young_\x\}_{\x}$
characterizing this limit and its interactions with continuous functions.
The rest of the proof is concerned with upgrading the weak-* convergence
to almost everywhere convergence (taking a further subsequence when necessary)
by reducing the Young measure to the Dirac mass
concentrated at $\bu(\x)$ away from the vacuum. Once the almost everywhere convergence is known,
obtaining the strong convergence in $L^p_\loc$ for all $p \in [1,\infty)$ is an elementary argument
by using the uniform boundedness of $\{\bu^\varepsilon\}_{\varepsilon>0}$ and the required interpolation.

We divide the rest of the proof into three steps.

\smallskip
\noindent 1. \textit{Hyperbolicity, genuine nonlinearity, and Riemann invariant coordinates}:
Recall from Proposition \ref{lem:hyperbolic and gn} (\S \ref{sec:hyp gen non}) that,
since the sequence of approximate solutions $\bu^\epsi$ satisfies the bounds in \eqref{5.1a}--\eqref{5.2a},
the nonlinear system \eqref{eq:conservative form u v ii} written in terms of $\Z(\bu)$
is strictly hyperbolic away from the vacuum;
however, the strict hyperbolicity fails at the vacuum.
Additionally, again by Lemma \ref{lem:hyperbolic and gn},
the system is genuinely nonlinear,
including the states at the vacuum.

Recall from Proposition \ref{lem:riemann invariants} that
\begin{equation*}
    W_\pm(\bu) = \vartheta(\bu) \pm k(|\bu|),
\end{equation*}
provided by \eqref{eq:Wpm expression 1} and \eqref{eq:Wpm expression 2},
are Riemann invariants
for the conservative system \eqref{eq:conservative form u v ii} for $\Z(\bu)$.
Furthermore, the change of coordinates: $\bu \mapsto (W_-(\bu),W_+(\bu))$ is bijective
on region $\{q_{\cri}\le |\bu| \leq q_{\cav}\}$.
We henceforth employ the curvilinear coordinate system $(W_-,W_+)$ in the phase-space and
denote the image measure of $\young_\x$ for this invertible coordinate change by $\young'_\x$,
which is precisely the Young measure generated
by sequence $\{(W_-(\bu^\varepsilon),W^\varepsilon_+(\bu^\varepsilon))\}_{\varepsilon>0} \subset L^\infty$.

\smallskip
\noindent 2. \textit{Reduction of the Young measure}:
Assumption \eqref{5.3a} indicates that the entropy dissipation measures for entropy pairs
generated by $H^\r$ and $H^\s$ satisfy the $H^{-1}$-compactness including at the vacuum state.
In turn, the Young measure satisfies the usual Tartar--Murat commutation relation
for all entropy pairs generated by kernels $H^\r$ and $H^\s$,
while these kernels $H^\r$ and $H^\s$ provide two linearly independent families of entropy pairs.
Then we can employ the arguments in DiPerna \cite{diperna2} or Serre \cite{SS0} for
the genuinely nonlinear conservative system \eqref{eq:conservative form u v ii}
to show that
the support of the family of Young measures generated by the sequence of approximate solutions
is one point in the Riemann invariant coordinates on the phase-plane, \textit{i.e.},
\begin{equation*}
    \young_\x' = \delta_{(W_-(\x),W_+(\x))} \qquad \text{a.e.~}\x,
\end{equation*}
for some functions $(W_-,W_+)\in L^\infty$.

\smallskip
\noindent 3. \textit{Strong convergence of the velocity}:
In turn, using the coordinate transformation: $(W_-,W_+) \mapsto \bu(W_-,W_+)$
in Lemma \ref{lem:riemann invariants}, which is well-defined everywhere, including at the vacuum,
we deduce that the original measure $\young_\x$ is also a Dirac mass in the $\bu$-phase plane.
By the uniqueness of weak-* limits, we have
$$
\young_\x = \delta_{\bu(\x)} \quad \text{a.e.~}\x.
$$
It follows that
there exists a subsequence (still denoted by) $\{\bu^\epsi\}_{\varepsilon>0}$
such that \eqref{eqn-convergence-approx} holds. The proof is complete.
\end{proof}

\begin{remark}\label{rem:vacuum cannot go from U back to u}
Observe that the mapping: $\bu \mapsto \Z(\bu)$ provided by \eqref{eq:unknown for conservative sys}
is not invertible at the vacuum, since
$|u(\Z)| = \sqrt{q^2_{\cav}-Z^2_2}$.
However, the sign cannot be determined solely from the entries of $\Z$,
whence $\bu(\Z)$ is not well-defined.
We emphasize that this is not a limitation for us, as our compensated compactness framework
is devised
to obtain $\bu$ directly as a solution of \eqref{eq:conservative form u v ii};
that is, it is not the case that we first obtain $\Z$
solving \eqref{eq:conservative form u v ii}
and then claim that $\bu(\Z)$ solves \eqref{eq:system in intro}.
\end{remark}

\section{Compactness of Entropy Solutions and Weak Continuity of the Euler Equations
for Potential Flow}\label{sec:weak continuity}

In this section, as a direct application of the compensated compactness framework
in \S \ref{sec:comp comp framework and red},
we prove Theorem \ref{big-thm-3},
concerned with the compactness of entropy solutions
and the weak continuity of system \eqref{eq:system in intro}
for the velocity field $\bu$.

\begin{proof}[Proof of Theorem \ref{big-thm-3}]  The proof is divided into four steps.

\smallskip
\noindent 1. First, assumptions (i)--(ii)
indicate that
the solutions sequence
$\{\bu^\epsi\}_{\varepsilon>0}$
satisfies conditions (i)--(ii) in the compensated compactness
framework (Theorem \ref{thm:compensated compactness framework intro}).
The key is to prove
the compactness of the entropy dissipation measures
for generic entropy pairs generated by kernels $H^\r$ and $H^\s$,
after which a direct application of Theorem \ref{thm:compensated compactness framework intro}
implies the convergence almost everywhere $\bu^\varepsilon \to \bu$.

\smallskip
\noindent 2.  For the special entropy pair  $\Q_*$,
it follows from the definition of entropy solutions (Definition \ref{def:ent sol}) that
the entropy inequality holds:
$$
\mu^\epsi_*:=-{\div_\x}\, \Q_*(\bu^\epsi) \le 0
$$
distributionally.
Furthermore, since sequence $\{\bu^\epsi\}_{\varepsilon>0}$ is uniformly bounded as indicated in Step 1,
we conclude that
$$
\{\mu^\epsi_*\}_{\varepsilon>0} \quad \mbox{is a bounded subset in $W^{-1,\infty}_{\rm loc}$}.
$$
An application of Murat’s Lemma, which states that the embedding of the positive cone of $W^{-1,q}$
into $W^{-1,p}$ is compact for $p<q$ (\textit{cf.}~\cite{muratcone}),  indicates that
$$
\{\mu^\epsi_*\}_{\varepsilon>0} \quad \mbox{is pre-compact in $W^{-1,p}_{\rm loc}$
for $p\in (1,\infty)$}.
$$
The same argument shows that, for any entropy pair $\Q_H$ determined
from the Loewner--Morawetz relations
via a generator $H$ generated by kernels $H^\r$ and $H^\s$ and satisfying
the convexity conditions in \eqref{eq:generator conditions ent ineq},
$$
\{{\div_\x}\, \Q_H(\bu^\epsi)\}_{\varepsilon>0} \qquad
\mbox{is pre-compact in $W^{-1,p}_{\rm loc}$ for $p\in (1,\infty)$}.
$$

\smallskip
\noindent 3.  In this step, we apply an idea from Chen \cite{C4} to transfer the compactness of
the entropy dissipation measures $\{\mu^\varepsilon_*\}_{\varepsilon>0}$
to the dissipation measures for any regular entropy pair,
including those that do not satisfy the convexity conditions
in \eqref{eq:generator conditions ent ineq}.

Let $\Q_H$ be any entropy pair determined from the Loewner--Morawetz relations
via a generator $H$ generated by kernels $H^\r$ and $H^\s$.
It is immediate from Propositions \ref{prop:est for h-1 cpct reg} and \ref{prop:cpctness est sing}
that there exists a positive constant $C_H$ such that
\begin{equation*}
\sum_{j=0}^2 (|\varrho \partial^j_\vartheta H_\nu| + |\partial^j_\vartheta H|) \leq C_H ,
\qquad \sum_{j=0}^2 |\partial^j_\vartheta(\varrho H_{\nu } + H_{\vartheta\vartheta})| \leq C_H \varrho.
\end{equation*}
Meanwhile, the special entropy generator $H^*$ satisfies the strict convexity
condition $H^*_{\vartheta\vartheta} = 1$.
In turn, defining $\widetilde{H} := H + C_0 H^*$ with $C_0 \geq 2 C_H$ a positive constant, we find that
\begin{equation*}
 \begin{aligned}   &\widetilde{H}_{\vartheta\vartheta} - \varrho \widetilde{H}_{\nu\vartheta\vartheta} \geq C_0 - C_H > 0, \\
 &\varrho (\widetilde{H}_{\vartheta\vartheta} - \varrho \widetilde{H}_{\nu\vartheta\vartheta}) \geq \varrho(C_0 - C_H)
  \geq \varrho C_H \geq |\varrho \widetilde{H}_{\nu\vartheta} + \widetilde{H}_{\vartheta\vartheta\vartheta}|.
 \end{aligned}
\end{equation*}
Using the linearity of the Loewner--Morawetz relations,
it follows that the entropy pair $\Q_H + C_0 \Q_*$, which is determined from generator $\widetilde{H}$,
satisfies the convexity conditions \eqref{eq:generator conditions ent ineq}.
Therefore, by setting $\mu^\epsi := {\div_\x}\,\Q_H(\bu^\epsi)$,
it follows from the conclusion of Step 2 that, for all $p \in (1,\infty)$,
$$
\{\mu^\epsi +C_0 \mu^\epsi_*\}_{\varepsilon>0}
\qquad\mbox{is pre-compact in $W^{-1,p}_{\rm loc}$}.
$$
By linearity, we conclude that $\{\mu^\varepsilon\}_{\varepsilon>0}$ is pre-compact in $W^{-1,p}_{\rm loc}$ for $p\in (1,\infty)$.
Thus, we conclude that, given any entropy pair  $\Q_H$ determined from the Loewner--Morawetz relations
via a generator $H$ generated by kernels $H^\r$ and $H^\s$, we have
\begin{equation*}
    \{ \div_\x \Q_H(\bu^\varepsilon) \}_{\varepsilon>0} \quad \text{is pre-compact in } H^{-1}_\loc.
\end{equation*}

\smallskip
\noindent 4.  Combining the results of the previous steps, we conclude that the solution
sequence $\{\bu^\epsi\}_{\varepsilon>0}$ satisfies
all the conditions in the compactness framework (Theorem \ref{thm:compensated compactness framework intro}).
This implies that
there exists a subsequence (still denoted by) $\{\bu^\epsi\}_{\varepsilon>0}$
that converges to the limit function $\bu$ almost everywhere
and strongly in $L^p_\loc$ for all $p \in [1,\infty)$.
We then apply the Dominated Convergence Theorem to show that $\bu$ is an entropy solution
of \eqref{eq:system in intro}--\eqref{eq:bernoulli general gamma intro} in the sense of Definition \ref{def:ent sol};
that is, $\bu$ satisfies \eqref{eq:system in intro} and \eqref{eq:ent ineq distr} in the sense of distributions.
\end{proof}

\section{Construction and Uniform Estimates of Viscous Approximate Solutions}\label{sec:viscous}
The goal of this section is to present an example to show how a sequence of approximate solutions
$\{\bu^\varepsilon\}_{\varepsilon>0}$ for system \eqref{eq:potential approx}
can be constructed to satisfy the conditions of the compensated compactness framework
established in \S \ref{sec:comp comp framework and red} in a fixed bounded domain $\mathscr{D}$.
To state the main result of this section, we begin by rewriting the potential flow
system \eqref{eq:system in intro}--\eqref{eq:bernoulli general gamma intro}
in the concise divergence-form:
\begin{equation*}
    \begin{cases}
         \div\big(\varrho q(\varrho)\e(\vartheta)\big) = 0, \\
         \div\big(q(\varrho)\e(\vartheta-\frac{\pi}{2})\big)= 0,
    \end{cases}
\end{equation*}
where $q(\varrho) = \sqrt{1-\varrho^2}$ is consistent with the Bernoulli
relation \eqref{eq:bernoulli general gamma intro}, and $\e(\vartheta)=(\cos\vartheta,\sin\vartheta)$.

The problem is posed in a domain $\mathscr{D} \subset \mathbb{H}$, an open connected subset with inward unit
normal $\mathbf{n}$ pointing into $\mathscr{D}$, which is filled with the fluid.
The boundary is written $\partial\mathscr{D} = \partial\mathscr{D}_1 \cup \partial\mathscr{D}_2$, where
$\partial\mathscr{D}_1$ is composed of only the obstacle boundary $\partial\obstacle$,
while $\partial\mathscr{D}_2$ is the remaining part of the rectangular box constituting $\partial\mathscr{D}$;
see Figure 8.1.

\vspace{-20pt}
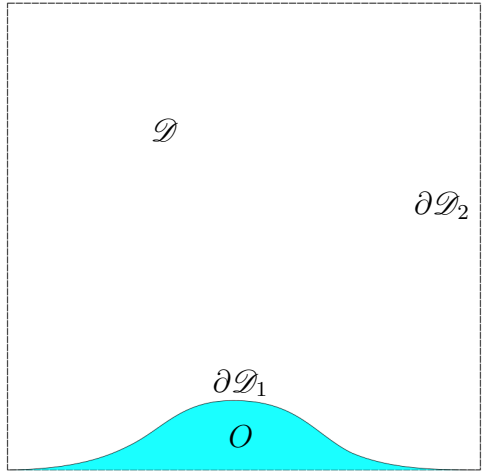
\begin{figure}[ht]
\begin{center}
\begin{tikzpicture}
   \draw (0, 0) node[inner sep=0]
   {\includegraphics[width=6.5cm]{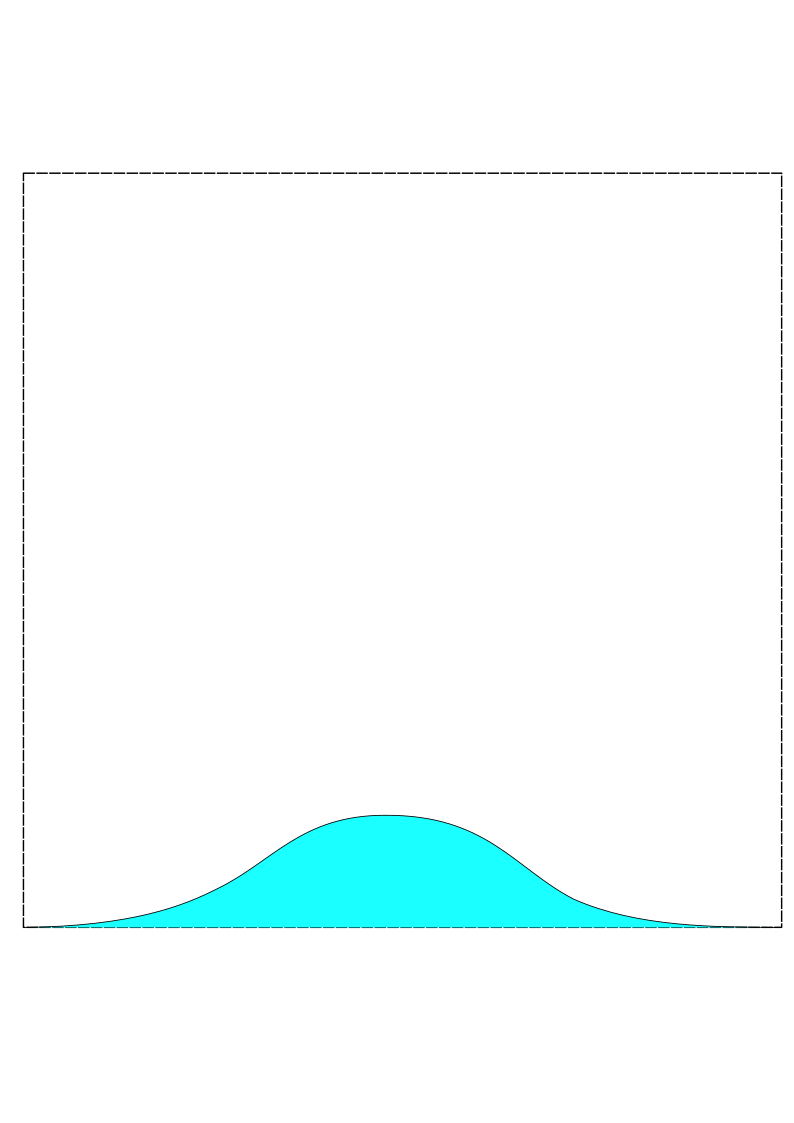}};
   \draw (0,-1.85) node {$\partial\mathscr{D}_1$};
    \draw (0,-2.55) node {$\obstacle$};
    \draw (-1,1.5) node {$\mathscr{D}$};
    \draw (2.65,0.5) node {$\partial\mathscr{D}_2$};
\end{tikzpicture} \hspace{1cm}
\end{center}
\vspace{-45pt}
\caption{\, The posed domain $\mathscr{D}$ for the viscous approximate problems.}\label{fig8.1}
\end{figure}

As per \cite{gq-transonic}, we add the viscosity terms $V_1^\varepsilon$ and $V_2^\varepsilon$
to the right-hand sides of the previous equations with
\begin{equation}\label{eq:choice of viscosities}
    \begin{cases}
        V_1^\varepsilon := \epsi\div\big((1-\displaystyle\frac{c(\varrho^\varepsilon)^2}{q(\varrho^\varepsilon)^2})\nabla \varrho^\varepsilon\big), \\
        V_2^\varepsilon := \epsi\Delta\vartheta^\varepsilon,
    \end{cases}
\end{equation}
for $\varepsilon \in (0,1)$ throughout this section.
In turn, the approximate problems for system \eqref{eq:potential approx} may be rewritten
as the quasilinear elliptic system:
\begin{equation}\label{eq:approx system start of s6}
 \begin{cases}
 \div\big(\varrho^\varepsilon q(\varrho^\varepsilon)\e(\vartheta^\varepsilon)\big) = V_1^\varepsilon, \\
  \div\big(q(\varrho^\varepsilon)\e(\vartheta^\varepsilon-\frac{\pi}{2})\big)= V_2^\varepsilon,
    \end{cases}
\end{equation}
along with
adequate boundary conditions.

The main result contained herein
will subsequently be employed in conjunction with Theorem \ref{thm:compensated compactness framework intro}
to prove Theorem \ref{thm:main existence theorem},
which will be done
in \S \ref{sec:entropy sol}.
Throughout this section, $\Q^\varepsilon$ is used to denote $\Q(\bu^\varepsilon)$,
where $\bu^\varepsilon$ is uniquely characterized by $(\varrho^\varepsilon,\vartheta^\varepsilon)$.

\begin{theorem}\label{big-thm-viscous}
There exists $\{(\varrho^\epsi,\vartheta^\epsi)\}_{\varepsilon>0} \subset C^2(\mathscr{D}) \cap C^1(\overline{\mathscr{D}})$
that is a uniformly bounded sequence of solutions in $L^\infty(\mathscr{D})$
of the approximate problems \eqref{eq:approx system start of s6} and satisfies that
there exist $\varrho_* \in (0,\varrho_{\cri})$ and $\vartheta_* , \vartheta^* \in \mathbb{R}$,
depending only on the boundary data imposed on $\partial\mathscr{D}_2$, such that
\begin{enumerate}
    \item[(i)] for all $\varepsilon>0$ and all $\x \in \overline{\mathscr{D}}$,
    \begin{equation*}
        0< \varrho^\varepsilon(\x) \leq \varrho_*< \varrho_{\cri}, \quad \vartheta_* < \vartheta^\varepsilon(\x) < \vartheta^*;
    \end{equation*}
    \item[(ii)] For any entropy pair $\Q$ generated by kernels $H^\r$ and $H^\s$ via the Loewner--Morawetz
    relations \eqref{eq:lowener mor},
    \begin{equation*}
        \{\div_\mathbf{x}\mathbf{Q}^\varepsilon\}_{\varepsilon>0} \qquad \text{is pre-compact in } H^{-1}(\mathscr{D}).
    \end{equation*}
\end{enumerate}
\end{theorem}

As discussed in \S \ref{sec:intro}--\S \ref{sec:main results},
our strategy is an elaboration of the earlier work due to Chen-Slemrod-Wang \cite{gq-transonic},
in which approximate problems are constructed as above to see whether the approximate solutions
satisfy
Morawetz's compactness
framework in the transonic case $\gamma \in [1,3)$ excluding both stagnation and cavitation.
It turns out that the case $\gamma=3$ differs from the problem studied in \cite{gq-transonic},
as the basic invariant region estimates that lie at the core of \cite{gq-transonic}
lead to different and, in fact, opposite conclusions.
By this, we mean that, contrary to the setting of \cite{gq-transonic} where an upper bound on the speed is obtained,
we here infer a lower bound (\textit{cf.}~Lemma \ref{lem:min pple speed}).
In turn, for suitably chosen supersonic boundary data, the approximate solutions remain supersonic inside the domain;
this is a special feature of system \eqref{eq:system in intro}--\eqref{eq:bernoulli general gamma intro}
for $\gamma\geq3$.

The rest of this section is dedicated to the proof of Theorem \ref{big-thm-viscous}
and associated properties of the elliptic system \eqref{eq:approx system start of s6}.
We do this in multiple steps, starting with \emph{a priori} $L^\infty$ estimates (invariant regions).
Note that it is not clear from equations \eqref{eq:approx system start of s6}
that $0\leq \varrho^\varepsilon \leq 1$ holds \emph{a priori},
whence the Bernoulli relation $q(\varrho)=\sqrt{1-\varrho^2}$ is not well-defined for the approximate solutions.
For this reason, we consider the auxiliary system:
   \begin{equation}\label{eq:variant system s6}
   \begin{cases}
         \div\big(\varrho^\varepsilon \tilde{q}(\varrho^\varepsilon)\e(\vartheta^\varepsilon)\big)
          = \epsi\div\big((1-\displaystyle\frac{c(\varrho^\varepsilon)^2}{\tilde{q}(\varrho^\varepsilon)^2})
          \nabla \varrho^\varepsilon\big), \\
         \div\big(\tilde{q}(\varrho^\varepsilon)\e(\vartheta^\varepsilon-\frac{\pi}{2})\big)
         = \varepsilon \Delta \vartheta^\varepsilon,
    \end{cases}
\end{equation}
where the speed has been replaced by
\begin{equation}\label{eq:variant speed}
    \tilde{q}(\varrho) = (1-\varrho_+^2)_+^{\frac{1}{2}},
\end{equation}
with $(\cdot)_\pm$ as the positive/negative part of the function.
While the derivative of $\tilde{q}$ is a distribution,
$\tilde{q}^2$ is Lipschitz and satisfies the formula:
\begin{equation}\label{eq:q tilde deriv expression}
    \tilde{q}(\varrho) \tilde{q}'(\varrho) = - \varrho \qquad\, \text{for all } \varrho \in (0,1).
\end{equation}
In \S \ref{sec:consistency}, we show that a solution of the above auxiliary system is
also a solution of the original system \eqref{eq:approx system start of s6};
that is, it is consistent with the potential flow system.

The invariant regions computed for
\eqref{eq:variant system s6} are then used to obtain a dissipation estimate,
after which the method laid out in \cite[\S 10]{gq-transonic} yields
the existence of a solution that fulfils the properties required in Theorem \ref{big-thm-viscous}.
This existence proof is carried out through a standard fixed point argument for the equivalent system
\begin{equation}\label{eq:sys for fixed point}
    \begin{cases}
       \laplace\sigma = \varepsilon^{-1}\div\big(\varrho(\sigma)q(\sigma)\e(\vartheta)\big),\\
       \laplace\vartheta = \varepsilon^{-1}\div\big(q(\sigma)\e(\vartheta-\frac{\pi}{2})\big),
    \end{cases}
\end{equation}
where
\begin{equation*}
    \sigma(\varrho)=\int_0^\varrho \frac{1-2s^2}{1-s^2}\d s = 2\varrho - \tanh^{-1}(\varrho)
\end{equation*}
so that
\begin{equation*}
    \sigma'(\varrho) = 1 - \frac{c(\varrho)^2}{q(\varrho)^2}, \quad \textit{i.e.}, \quad
   \Delta \sigma=\div\big(( 1 - \frac{c(\varrho)^2}{q(\varrho)^2}) \nabla \varrho \big).
\end{equation*}
This approach is standard so that the details are omitted (see \cite{gq-transonic}).
Henceforth, we focus exclusively on the required \emph{a priori} $L^\infty$ and dissipation estimates.

The first result concerns the invariant region estimate,
for which we recall the characteristic speed $k$ expressed in \eqref{eq:Wpm expression 2} as a function of $q$.

\begin{prop}[Uniform Upper Bound on the Density]\label{prop:upper bound density}
Let $(\varrho^\epsi,\vartheta^\epsi)\in C^2(\mathscr{D})\cap C^1(\overline{\mathscr{D}})$
be a solution of the auxiliary system \eqref{eq:variant system s6} with the following boundary conditions{\rm :}
on the obstacle boundary
$\partial\mathscr{D}_1$,
\begin{equation}
    \begin{cases}
           \epsi\big(1-\frac{c^2(\varrho^\epsi)}{\tilde{q}^2(\varrho^\epsi)}\big) \nabla\varrho^\epsi \! \cdot \! \mathbf{n}
            = \left|\varrho^\epsi \tilde{q}(\varrho^\epsi) \, \e(\vartheta^\epsi) \! \cdot \!\mathbf{n} \right|, \\
           \nabla\vartheta^\epsi \! \cdot \! \mathbf{n} =0,
    \end{cases}
\end{equation}
together with the Dirichlet-type condition on $\partial\mathscr{D}_2${\rm :}
\begin{equation*}
    \varrho^\epsi|_{\partial\mathscr{D}_2} = \varrho_\infty,\qquad \vartheta^\epsi|_{\partial\mathscr{D}_2} =0,
\end{equation*}
where $\varrho_\infty=\sqrt{1-q^2_\infty}$.
Then there exists $\varrho_* \in (0,\varrho_\cri)$,  depending only $\varrho_\infty$
but independent of $\varepsilon$, such that, for all $\x \in \mathscr{D}$,
\begin{equation*}
   0 < \varrho^\epsi(\x) \leq \varrho_* < \varrho_{\cri}.
\end{equation*}
\end{prop}

We remark that the previous proposition implies the \emph{strict positivity} of the densities $\varrho^\varepsilon$
and speeds $q(\varrho^\varepsilon)$ for any fixed $\varepsilon>0$.
Furthermore, the invariant regions
imply $\tilde{q}(\varrho^\varepsilon)=q(\varrho^\varepsilon)$.
Thus, we deduce that a solution of the auxiliary system \eqref{eq:variant system s6}
is also a solution of \eqref{eq:approx system start of s6} (see Lemma \ref{lem:max pple speed}).
This means that, at this point, the Bernoulli relation:
\begin{equation*}
    \varrho(q) = \sqrt{1 - q^2}
\end{equation*}
can be meaningfully interpreted as the equivalent nonlinear constraint:
\begin{equation*}
    q(\varrho)=\sqrt{1-\varrho^2},
\end{equation*}
with the change of coordinate: $\varrho \mapsto q$
and its inverse mapping being justified on the entire domain as $q'(\varrho)>0$ for $\varrho \in (0,\varrho_*]$.
Furthermore, we remark that the \emph{a priori} estimates of Proposition \ref{prop:upper bound density}
imply $q(\varrho^\varepsilon) \geq q(\varrho_{\cri})$ and
\begin{equation*}
 0<  1 - \frac{\varrho_{*}^2}{1-\varrho_*^2}\leq 1-\frac{c(\varrho^\varepsilon)^2}{q(\varrho^\varepsilon)^2}\leq 1,
\end{equation*}
in view of $\varrho_* < \varrho_{\cri}=\frac{1}{\sqrt{2}}$.
It follows from the above that the approximate problem \eqref{eq:approx system start of s6}
and the auxiliary problem \eqref{eq:variant system s6} are \emph{strongly elliptic} quasilinear systems.
Similar arguments show that the changes of coordinate: $\varrho \mapsto \nu$ and $\varrho \mapsto \sigma$
are also justified and invertible on the entire domain,
in view of relations $\nu'(\varrho) = M^{-2}>0$ and $\sigma'(\varrho) = 1-M^{-2}>0$
when the Mach number $M \in (1,\infty)$ (\textit{cf.}~\S \ref{sec:setup})
which is equivalent to $\varrho\in (0, \varrho_{\cri})$.
In summary, away from the critical points ($M=1$) and the stagnation points ($M=0$),
variables $\varrho$, $q$, $\nu$, and $\sigma$ differ from each other by composition with a smooth function.
Hence, once the \emph{a priori} $L^\infty$ estimates of Proposition \ref{prop:upper bound density} are established,
we refer indistinctly to a viscous solution as being parameterized in Cartesian coordinates $(u,v)$,
or in polar coordinates $(q,\vartheta)$,
or any of the following parameterizations $(\varrho,\vartheta)$, $(\nu,\vartheta)$, or $(\sigma,\vartheta)$.

Similar arguments to those in the proof of Proposition \ref{prop:upper bound density} are used
to obtain the $L^\infty$ estimates of the unknown angle $\vartheta^\epsi$.
Note carefully that, as the problem is naturally interpreted to be periodic in $\vartheta$,
one might mistakenly believe that uniform bounds on the angle are unnecessary.
However, the Loewner--Morawetz entropy framework requires that $|\vartheta\nabla\vartheta|^2$
be integrable, which is not the case when $\vartheta$ is merely represented as a piecewise linear
periodic function with jump discontinuities.
In turn, we have the second invariant region result.

\begin{prop}[Uniform Estimates of the Angle Functions]\label{prop:invariant regions in angle}
Let $(\varrho^\epsi,\vartheta^\epsi) \in C^2(\mathscr{D}) \cap C^1(\overline{\mathscr{D}})$
be a solution of the auxiliary system \eqref{eq:variant system s6} with the prescribed boundary conditions
of Propositions \ref{prop:upper bound density}.
Then there exist $\vartheta_* , \vartheta^* \in \mathbb{R}$ depending only
on $\varrho_\infty$
such that
\begin{equation*}
\vartheta_* \leq \vartheta^\varepsilon(\x) \leq \vartheta^*
\qquad \mbox{for a.e.~$\x \in \mathscr{D}$}.
\end{equation*}
\end{prop}

As already mentioned, we then prove the main dissipation estimate needed to show the required compactness
of the entropy dissipation measures.

\begin{prop}[Main Dissipation Estimate]\label{prop:dissipation estimate}
Let $(\varrho^\epsi,\vartheta^\epsi)\in C^2(\mathscr{D})\cap C^1(\overline{\mathscr{D}})$ be a solution of
the auxiliary system \eqref{eq:variant system s6} with the prescribed boundary conditions
of {\rm Propositions \ref{prop:upper bound density}--\ref{prop:invariant regions in angle}},
assuming the constant Dirichlet conditions on $\partial\mathscr{D}_2$:
\begin{equation*}
    \vartheta|_{\partial \mathscr{D}_2} = \vartheta_\infty = 0,
    \qquad \varrho|_{\partial\mathscr{D}_2} = \varrho_\infty \in (0,\varrho_{\cri}).
\end{equation*}
Then there exists a positive constant $C_\mathscr{D}$,
depending on domain $\mathscr{D}$ but independent of $\varepsilon$, such that, for all $\epsi\in (0,1)$,
\begin{equation}
\epsi \int_\mathscr{D} \Big(|\nabla\vartheta^\epsi|^2
+ \big(1-\frac{c^2(\varrho^\epsi)}{q^2(\varrho^\epsi)}\big)\frac{1}{q(\varrho^\varepsilon)^2}
|\nabla\varrho^\epsi|^2\Big)\d\mathbf{x} \leq C_\mathscr{D}.
\end{equation}
\end{prop}

Using the previous result, we are in a position to prove that any solution of
the viscous problem \eqref{eq:approx system start of s6} yields the $H^{-1}$-compactness of
entropy dissipation measures for the family of
entropy pairs
we exhibited earlier. More precisely, we have

\begin{prop}[Compactness of Entropy Dissipation Measures]\label{prop:H-1 cpct for viscous solns using our kernels}
Let $(\varrho^\epsi,\vartheta^\epsi)\in C^2(\mathscr{D})\cap C^1(\overline{\mathscr{D}})$ be solutions of
the auxiliary system \eqref{eq:variant system s6} with the prescribed boundary conditions
of {\rm Propositions \ref{prop:upper bound density}--\ref{prop:dissipation estimate}}.
Then there exists a compact set $K\subset H^{-1}(\mathscr{D})$ such that,
for any entropy pair $\mathbf{Q}$ generated by kernels $H^\r$ ad $H^\s$,
\begin{equation*}
  \{ \div_\x \Q^\varepsilon \}_{\varepsilon>0} \subset K.
\end{equation*}
\end{prop}

At this stage, one may employ the previous propositions and
the aforementioned fixed-point strategy for system \eqref{eq:sys for fixed point} to
prove the following existence result, which leads to the complete proof of Theorem \ref{big-thm-viscous}.

\begin{prop}\label{prop:existence partial}
Let $\varrho_\infty\in (q_{\rm cr}, q_{\rm cav})$.
Then the mixed Dirichlet--Neumann problem{\rm :}
\begin{equation}\label{eq:sys Tristan}
\begin{cases}
\div\big(\varrho^\varepsilon q(\varrho^\varepsilon)\e(\vartheta^\varepsilon)\big)
= \epsi\div\big((1-\displaystyle\frac{c(\varrho^\varepsilon)^2}{q(\varrho^\varepsilon)^2})\nabla \varrho^\varepsilon\big), \\
\div\big(q(\varrho^\varepsilon)\e(\vartheta^\varepsilon-\frac{\pi}{2})\big)
= \epsi\Delta\vartheta^\varepsilon,\\
\epsi\big(1-\frac{c^2(\varrho^\epsi)}{q^2(\varrho^\epsi)}\big)
\nabla\varrho^\epsi \! \cdot \!\mathbf{n}|_{\partial\mathscr{D}_1}
= \left|\varrho^\epsi  q(\varrho^\epsi) \, \e(\vartheta^\epsi) \! \cdot\!\mathbf{n}\right|,\\
              \nabla\vartheta^\epsi \! \cdot \! \mathbf{n}|_{\partial\mathscr{D}_1} =0, \\
              (\varrho^\varepsilon,\vartheta^\varepsilon)|_{\partial\mathscr{D}_2} = (\varrho_\infty,0),
    \end{cases}
\end{equation}
admits a
solution $(\varrho^\epsi,\vartheta^\epsi)\in C^2(\mathscr{D})\cap C^1(\overline{\mathscr{D}})$. Furthermore, this solution satisfies
the conditions of Theorem \ref{big-thm-viscous}.
\end{prop}

We note in passing that the dissipation estimate in Proposition \ref{prop:dissipation estimate} implies that
the viscosity terms $V_1^\varepsilon$ and $V_2^\varepsilon$, defined in \eqref{eq:choice of viscosities} in
terms of the solution of \eqref{eq:sys Tristan}, converge to $\mathbf{0}$ strongly in $H^{-1}(\mathscr{D})$.

Throughout this section, except stated otherwise, we use the convention: $\nabla = (\partial_x,\partial_y)$;
analogously, all divergences are taken with respect to the physical space variable $\x$.

\subsection{A priori $L^\infty$ estimates}
We consider the invariant regions of the system.
In what follows, we use
$W_\pm^\varepsilon(\x) := W_\pm(\vartheta^\varepsilon(\x),\tilde{q}(\varrho^\varepsilon(\x)))$.
We recall from \cite[Proposition 4.1]{gq-transonic} that, in the case that $\gamma =3$ with $\varrho \in (0,\varrho_\cri)$,
the Riemann invariants satisfy the elliptic equation:
\begin{equation}\label{eq:riemann inv elliptic for max pple}
\frac{\tilde{q}c}{\sqrt{\tilde{q}^2-c^2}} \boldsymbol{\lambda}_\pm \! \cdot \! \nabla W^\varepsilon_\pm \mp \Delta W^\varepsilon_\pm
= -\varepsilon\frac{2 c^3 }{\varrho^2 \tilde{q}^2 \sqrt{\tilde{q}^2 - c^2}}|\nabla \varrho|^2,
\end{equation}
where we have used relation \eqref{eq:q tilde deriv expression}
and $\boldsymbol{\lambda}_\pm = (\lambda^1_\pm,\lambda^2_\pm)$ with
\begin{equation*}
    \lambda^1_\pm = -\sin \vartheta \pm \frac{\sqrt{\tilde{q}^2-c^2}}{c}\cos\vartheta,
    \qquad \lambda^2_\pm = \cos \vartheta \pm \frac{\sqrt{\tilde{q}^2-c^2}}{c}\sin\vartheta;
\end{equation*}
also see~\S \ref{sec:riemann invariants}.

We show that the maximum value of $W_+^\varepsilon$ and the minimum value of $W_-^\varepsilon$ must
occur on boundary $\partial\mathscr{D}_2$.
Indeed, our choice of the boundary conditions and relations \eqref{eq:Riem pdes} lead to
\begin{equation}\label{eq:boundary deriv}
    \nabla W^\varepsilon_\pm \! \cdot \! \mathbf{n}
    = \pm  \frac{\varepsilon^{-1}}{\sqrt{\tilde{q}^2-c^2}}
    |\varrho^\varepsilon  \tilde{q}(\varrho^\varepsilon) \, \e( \vartheta) \! \cdot \! \mathbf{n}|
    \qquad\,\,\,\text{on } \partial \mathscr{D}_1,
\end{equation}
where we have used the relation:
\begin{equation*}
\nabla W_\pm = \nabla \vartheta \mp \frac{\sqrt{q^2-c^2}}{qc}\nabla q = \nabla \vartheta \pm \frac{\sqrt{q^2-c^2}}{q^2 } \nabla \varrho,
\end{equation*}
which is valid for $q \in (q_{\cri},q_{\cav})$.
Using the information on the sign of the normal derivative from \eqref{eq:boundary deriv},
the Hopf Lemma forbids $W_+^\varepsilon$ to attain a maximum (resp.~$W_-^\varepsilon$ to attain a minimum)
on the obstacle boundary $\partial \mathscr{D}_1$.
Nevertheless, equation \eqref{eq:riemann inv elliptic for max pple} is strongly elliptic and contains
no zero-order derivative terms in $W^\varepsilon_\pm$,
whence the Weak Maximum Principle implies that the extremal values of $W^\varepsilon_\pm$ are achieved
on boundary $\partial\mathscr{D} = \partial\mathscr{D}_1 \cup \partial\mathscr{D}_2$.
We thereby deduce that these extremal values are attained only on $\partial\mathscr{D}_2$,
whence
\begin{equation*}
   \sup_{\x \in \mathscr{D}} W^\varepsilon_+(\x) \leq \sup_{\x \in \partial \mathscr{D}_2} W^\varepsilon_+(\x),
   \qquad \inf_{\x \in \mathscr{D}} W^\varepsilon_-(\x) \geq \inf_{\x \in \partial\mathscr{D}_2} W_-^\varepsilon(\x).
\end{equation*}

In what follows, we use the shorthand notation $\tilde{q}^\varepsilon = \tilde{q}(\varrho^\varepsilon)$.
From the above, \eqref{eq:Wpm expression 1}, and \eqref{eq:Wpm expression 2}, we obtain that, for all $\x\in \mathscr{D}$,
\begin{equation}\label{eq:min pple on speed}
    k(\tilde{q}^\varepsilon(\x)) = \frac{1}{2}\big(W_+^\varepsilon(\x)-W_-^\varepsilon(\x)\big)
    \leq \frac{1}{2}\big( \sup_{\x \in \partial \mathscr{D}_2} W^\varepsilon_+(\x)
      - \inf_{\x \in \partial \mathscr{D}_2} W^\varepsilon_-(\x) \big).
\end{equation}
We estimate the quantity on the right-hand side of \eqref{eq:min pple on speed}.
Note that the Dirichlet-type condition $(\varrho_\infty, 0)$ is imposed on
boundary $\partial\mathscr{D}_2$, from which it follows that, for all $\x \in \partial\mathscr{D}_2$ and all $\varepsilon>0$,
\begin{equation*}
    W_-^\varepsilon(\x) = - k(q_\infty),
    \qquad W_+^\varepsilon(\x) =  k(q_\infty),
\end{equation*}
where $q_\infty := q(\varrho_\infty)$.
In turn, equation \eqref{eq:min pple on speed} and the monotonicity of $k$ (\textit{cf.}~\S \ref{sec:riemann invariants})
imply the upper bound
\begin{equation}\label{eq:lower bound on W proper}
    k(\tilde{q}^\varepsilon(\x)) \leq  k(q_\infty)  \qquad \text{for all } \x \in \mathscr{D}.
\end{equation}

\begin{lemma}[Minimum Principle for the Speed]\label{lem:min pple speed}
Let $(\varrho^\epsi,\vartheta^\epsi) \in C^2(\mathscr{D}) \cap C^1(\overline{\mathscr{D}})$ be
a solution of the auxiliary system \eqref{eq:variant system s6}, and let
the boundary data satisfy $q_\infty \in (q_\cri,q_\cav)$.
Then there exists $q_* \in (q_\cri, q_\cav)$ depending only on $\varrho_\infty$
such that
\begin{equation*}
    \tilde{q}^\varepsilon(\x) \geq q_* > q_{\cri} \qquad\,\,\mbox{for all $\x \in \mathscr{D}$}.
\end{equation*}
\end{lemma}

\begin{proof}
First, by definition, $\tilde{q}^\varepsilon$ never exceeds $q_{\cav}$.
Suppose that $\tilde{q}^\varepsilon(\x) = q_{\cav}$ at some point $\x \in \mathscr{D}$.
Then $\inf_{\partial\mathscr{D}_2} q_\infty < q_\cav$ due to the choice of
the boundary data so that we directly see
that $\tilde{q}^\varepsilon(\x) \geq \inf_{\partial\mathscr{D}_2}q_\infty > q_\cri$
at this specific point.

Furthermore, suppose that $\x \in \mathscr{D}$ is such that $\tilde{q}^\varepsilon(\x) < q_{\cav}$.
Define the strictly positive number
$\delta_* := k(q_\infty)$,
and observe from estimate \eqref{eq:lower bound on W proper} and our assumptions that
\begin{equation*}
    k(\tilde{q}^\varepsilon(\x)) \leq \delta_* < (\sqrt{2}-1)\frac{\pi}{2} = k(q_\cri).
\end{equation*}
Recall from \S \ref{sec:riemann invariants} that $k$ is invertible and strictly decreasing in $q$
on the restriction of its domain to $(q_\cri,q_\cav)$, whence
\begin{equation*}
    \tilde{q}^\varepsilon(\x) \geq k^{-1}(\delta_*) > q_\cri
\end{equation*}
at all such points $\x$. In turn, we define $q_* := \min\{k^{-1}(\delta_*) , q_\infty \}$,
and the conclusion of the lemma is manifestly verified.
\end{proof}

Note that the bound: $\sup_{\mathscr{D}}\tilde{q}^\varepsilon \leq q_{\cav}$ follows directly from
$\tilde{q}(\varrho) = (1-\varrho^2_+)_+^{\frac{1}{2}}$.
We are thus in a position to provide the required uniform bound of
sequence $\{\vartheta^\varepsilon\}_{\varepsilon>0}$.

\begin{lemma}[Uniform Estimate of the Angle Functions]\label{lem:unif est angle}
Suppose that $(\varrho^\epsi,\vartheta^\epsi) \in C^2(\mathscr{D}) \cap C^1(\overline{\mathscr{D}})$
is a solution of the auxiliary system \eqref{eq:variant system s6} with the boundary data
$q_\infty \in (q_\cri,q_\cav)$.
Then there exist $\vartheta_*, \vartheta^*\in\mathbb{R}$ depending only on $\varrho_\infty$
such that
\begin{equation*}
    \vartheta_* \leq \vartheta^\varepsilon(\x) \leq \vartheta^*
    \qquad\,\,\mbox{for all $\x \in \overline{\mathscr{D}}$}.
\end{equation*}
\end{lemma}

\begin{proof}
We obtain directly from the expressions in  \eqref{eq:Wpm expression 1} and \eqref{eq:Wpm expression 2}
for the Riemann invariants that, for all $\x \in \mathscr{D}$,
\begin{equation*}
    \begin{aligned}
        \vartheta^\varepsilon(\x) = \frac{1}{2}\big(W_-^\varepsilon(\x) + W_+^\varepsilon(\x)\big)
        \leq&\, \frac{1}{2}\big( W^\varepsilon_-(\x) + \sup_{\x \in \partial \mathscr{D}_2} W^\varepsilon_+(\x) \big) \\
        \leq & \, \frac{1}{2}\big( \vartheta^\epsi(\x) - k(\tilde{q}^\epsi(\x))
        + k(q_\infty)\big),
    \end{aligned}
\end{equation*}
whence, using also that $k$ is nonnegative and decreasing in $q$ (\textit{cf.}~\S \ref{sec:riemann invariants}), we obtain the bound
\begin{align}
        \vartheta^\varepsilon(\x) \leq& \, -k(\tilde{q}^\epsi(\x))
        + k(q_\infty)
        \leq
        k(q_\infty) =:\, \vartheta^*. \label{eq:max pple on angle}
\end{align}
Similarly,
\begin{equation*}
    \begin{aligned}
        \vartheta^\varepsilon(\x) = \frac{1}{2}\big(W_+^\varepsilon(\x) + W_-^\varepsilon(\x)\big)
        \geq&\, \frac{1}{2}\big( W^\varepsilon_+(\x) + \inf_{\x \in \partial \mathscr{D}_2} W^\varepsilon_-(\x) \big) \\
        \geq &\, \frac{1}{2}\big( \vartheta^\epsi(\x) + k(\tilde{q}^\epsi(\x))
        - k(q_\infty)\big),
    \end{aligned}
\end{equation*}
which implies
\begin{align}
        \vartheta^\varepsilon(\x) \geq&\, k(q^\epsi(\x))
        -k(q_\infty)
        \geq
        - k(q_\infty) =: \,\vartheta_*, \label{eq:min pple on angle}
\end{align}
and the result follows.
\end{proof}

It follows that, to prove Propositions \ref{prop:upper bound density}--\ref{prop:invariant regions in angle},
it
suffices to verify that the solutions of the auxiliary system satisfy
the stronger estimate: $\varrho^\varepsilon(\x)>0$ for all $\x\in\overline{\mathscr{D}}$.
This is the focus of the next section.

\subsection{Consistency of the auxiliary system and strict positivity of $\varrho^\varepsilon$}\label{sec:consistency}

\begin{lemma}[Consistency of the Auxiliary Approximate Problem]\label{lem:max pple speed}
Suppose that $(\varrho^\epsi,\vartheta^\epsi) \in C^2(\mathscr{D}) \cap C^1(\overline{\mathscr{D}})$
is a solution of the auxiliary system \eqref{eq:variant system s6}.
Then $(\varrho^\epsi,\vartheta^\epsi)$ is a solution of system \eqref{eq:sys Tristan}
such that, for all $\varepsilon>0$ and all $\x \in \overline{\mathscr{D}}$,
\begin{equation}\label{eq:unif approx est}
    q_{\cri} < q_* \leq q^\varepsilon(\x) \leq q_{\cav}, \quad \vartheta_* \leq \vartheta^\varepsilon(\x) \leq \vartheta^*.
\end{equation}
\end{lemma}
\begin{proof}
In this proof, we omit the $\varepsilon$-superscript to simplify notation.

Suppose that $(\varrho^\epsi,\vartheta^\epsi) \in C^2(\mathscr{D}) \cap C^1(\overline{\mathscr{D}})$
solves \eqref{eq:variant system s6}.
Then, testing the first equation in \eqref{eq:variant system s6} with the negative part
$\varrho_-^\varepsilon = -\min\{\varrho^\varepsilon,0\}$ and using the Divergence Theorem,
and omitting the $\varepsilon$-superscript for ease of notation,
we obtain
\begin{equation*}
    \begin{aligned}
        \epsi \int_\mathscr{D} (1-\frac{c^2}{\tilde{q}^2}) |\nabla \varrho_-|^2 \d \x
         = & \int_\mathscr{D} \varrho_- \tilde{q}(\varrho) \, \nabla \varrho_- \! \cdot \! \e(\vartheta)  \d \x
         + \int_{\partial\mathscr{D}} \varrho_-^2 \tilde{q}(\varrho) \e(\vartheta) \cdot \mathbf{n} \d \mathcal{H}(\x) \\
        &-\, \epsi \int_{\partial \mathscr{D}_2} \big(1-\frac{c^2}{\tilde{q}^2}\big) (\varrho_\infty)_-\,
         \nabla \varrho_\infty \! \cdot \! \mathbf{n}  \d \mathcal{H}(\x) \\
        &-\int_{\partial \mathscr{D}_1} |\varrho_-^2 \tilde{q}(\varrho) \, \e( \vartheta) \! \cdot \! \mathbf{n} |  \d \mathcal{H}(\x) \\
        =& \,  0;
    \end{aligned}
\end{equation*}
the third term on the right-hand side vanishes because $\varrho_\infty \geq 0$,
and all the other terms also vanish due to
$\supp \tilde{q}(\varrho) \cap \supp \varrho_- = \{ \x \in \overline{\mathscr{D}}\,:\, \varrho(\x) = 0 \}$.
In view of the lower bound $\tilde{q} \geq q_* > q_{\cri}$ from Lemma \ref{lem:min pple speed},
it follows that there exists $C>0$ such that $1-\frac{c^2}{\tilde{q}^2} \geq C$, whence
\begin{equation*}
    \int_\mathscr{D} |\nabla \varrho_-|^2 \d \x = 0.
\end{equation*}
Since $\varrho_\infty \geq 0$,
it follows from the above that $\varrho_- \equiv 0$ in $\mathscr{D}$, whence $\varrho \geq 0$ in $\mathscr{D}$.

Meanwhile, the lower bound on $\tilde{q}$ of Lemma \ref{lem:min pple speed} implies that
$\varrho \leq \varrho(q_*) < \varrho_{\cri} < 1$.
In turn, it follows that $\tilde{q}(\varrho^\varepsilon) = q(\varrho^\varepsilon) = \sqrt{1-(\varrho^\varepsilon)^2}$,
whence $(\varrho^\varepsilon,\vartheta^\varepsilon)$ solves
problem \eqref{eq:sys Tristan} and satisfies the uniform estimates \eqref{eq:unif approx est}.
\end{proof}

In fact, a stronger result is available, now that the weak lower bound $\varrho^\varepsilon \geq 0$ has been established.

\begin{corollary}\label{cor:strictly positive rho viscous}
For all $\x \in \overline{\mathscr{D}}$, $\varrho^\epsi(\x)>0$ so that $q^\epsi(\x)<q_{\cav}.$
\end{corollary}

\begin{proof}
Suppose for contradiction that $\varrho^\varepsilon(\x_0)=0$ for some $\x_0 \in \overline{\mathscr{D}}$.
First, our choice of the boundary data on $\partial\mathscr{D}_2$ implies $\x_0 \notin \partial\mathscr{D}_2$.
We now show that the other two possibilities, $\x_0 \in \mathscr{D}$ (an interior point) or $\x_0 \in \partial\mathscr{D}_1$,
also yield contradictions.

\smallskip
\noindent   1. \textit{$\x_0 $ is not an interior point of $\mathscr{D}$}:
Suppose first that $\x_0 \in \mathscr{D}$ is an interior point.
Recall that $\varrho$ solves the original viscous problem \eqref{eq:approx system start of s6},
which is strongly elliptic and reads in the non-divergence form as
    \begin{equation*}
       \big(1- \frac{\varrho}{q(\varrho)}\big) \nabla \varrho \cdot \e(\vartheta) + \varrho q(\varrho) \nabla \vartheta \cdot \e(\vartheta+\frac{\pi}{2})
       = - \frac{2\varrho}{(1-\varrho^2)^2} \varepsilon|\nabla \varrho|^2 + \big(1-\frac{\varrho^2}{1-\varrho^2}\big) \varepsilon\Delta \varrho.
\end{equation*}
Since $\varrho(\x_0) =0= \inf_{\mathscr{D}} \varrho$, the version of the Strong Maximum Principle
that requires no sign condition on the lowest-order term implies that $\varrho \equiv 0$ is constant,
which contradicts the boundary condition on $\partial\mathscr{D}_2$.

\smallskip
\noindent  2. \textit{$\x_0$ does not belong to boundary $\partial\mathscr{D}_1$}:
Suppose that $\x_0 \in \partial\mathscr{D}_1$.
Then $\varrho(\x_0) = 0 = \inf_\mathscr{D} \varrho$ is achieved on boundary $\partial\mathscr{D}_1$,
and thus the version of the Hopf Lemma that requires no sign condition
on the lowest-order term implies either $\nabla \varrho(\x_0) \cdot \mathbf{n}>0$ or $\varrho \equiv 0$ is constant.
This latter conclusion cannot be true because of the boundary condition on $\partial\mathscr{D}_2$,
whence $\nabla \varrho(\x_0) \cdot \mathbf{n}>0$.
However, by evaluating the boundary condition on $\partial\mathscr{D}_1$ at point $\x_0$, we have
\begin{equation*}
    \varepsilon \nabla \varrho(\x_0) \cdot \mathbf{n} = 0,
\end{equation*}
which is again a contradiction.
\end{proof}

Combining the results in Lemmas \ref{lem:min pple speed}--\ref{lem:max pple speed} and
Corollary \ref{cor:strictly positive rho viscous}, and setting $\varrho_* := \varrho(q_*)$,
we complete the proof of Propositions \ref{prop:upper bound density}--\ref{prop:invariant regions in angle}.

\subsection{Dissipation estimate and $H^{-1}$-compactness}\label{section:what estimates}

We begin this section by proving Proposition \ref{prop:dissipation estimate}, which in effect verifies
that the dissipation estimate \cite[Proposition 8.1]{gq-transonic} is still valid for the solutions generated
by the viscous problems \eqref{eq:sys Tristan}. Although we write the proof explicitly only for $\gamma=3$,
the proof generalizes to all $\gamma > 2$; see \cite[Lemma 4.17]{thesis} for further details.

\begin{proof}[Proof of Proposition \ref{prop:dissipation estimate}]

The proof is divided into two steps.

\smallskip
\noindent 1. \textit{Special entropy generator and pointwise estimate}:
For $q_\infty\in (q_\cri, q_{\cav})$, denote $\nu_\infty=\nu(\varrho(q_\infty))$.
We set $\varrho_\infty = \varrho(q_\infty)$ and recall $\varrho_\cri = \varrho(q_\cri)$.
In view of the monotonicity provided by \eqref{eq:nu prime eqn}, we see that
$\nu_\infty > 0$ and $0 <  \varrho_\infty< \varrho_{\cri}$.
We now recall the special entropy generator \eqref{eq:special entropy generator}:
\begin{equation*}
    H^*(\nu,\vartheta) = \frac{\vartheta^2}{2}
   -\frac{\nu}{\varrho_\infty}
   + \int_{\nu_\infty}^\nu \int^\tau_{\nu_\infty} \frac{M(\nu')^2 - 1}{\varrho(\nu')^2} \d \nu' \d \tau,
\end{equation*}
and observe that the double integral is well-defined, since \begin{equation*}
    \int_{\nu_\infty}^\tau \frac{M(\nu')^2 - 1}{\varrho(\nu')^2} \d \nu'
    = \int_{\nu_\infty}^\tau k'(\nu')^2 \d \nu' \leq C\int_{\nu_\infty}^\tau (\nu')^{-\frac{4}{3}} \d \nu'
    \leq C\tau^{-\frac{1}{3}}
\end{equation*}
for some positive constant $C$ depending on $\bar{\nu}$,
and the final right-hand side is integrable. Furthermore, a direct computation via the chain rule and \eqref{eq:nu prime eqn} yields
the equality away from the vacuum:
\begin{equation}\label{eq:Hnu special gen}
    H^*_\nu = -\frac{1}{\varrho} + N(\varrho),
\end{equation}
where we have defined
\begin{equation}\label{8.17a}
    N(\varrho) := - \int_{\varrho_\infty}^\varrho \frac{\d \varrho'}{q(\varrho')^2}.
\end{equation}
Provided $\varrho \in [0,\varrho_{\cri}]$,
where we recall for $\gamma=3$ that $\varrho_{\cri}=\frac{1}{\sqrt{2}}$, we estimate this
term as
\begin{equation*}
    \begin{aligned}
        |N(\varrho)| \leq&   \left| \int_{\varrho}^{\varrho_\infty} \frac{\d \tau}{1- \tau^{2} } \right| \mathds{1}_{\{\varrho < \varrho_\infty\}}
        +   \left|\int^{\varrho}_{\varrho_\infty}
         \frac{\d \tau}{1- \tau^{2} } \right| \mathds{1}_{\{\varrho \geq \varrho_\infty\}}
        \\
        \leq & \, \int_{0}^{\varrho_\infty} \frac{\d \tau}{1- \tau^{2} }
        + \int_{\varrho_\infty}^{\varrho_{\cri}} \frac{\d \tau}{1- \tau^{2} } \\
        \leq & \, \frac{\varrho_{\cri}}{1-\varrho_{\cri}^2}
    \end{aligned}
\end{equation*}
and hence
\begin{equation}\label{eq:N(rho) est}
\sup_{\varrho \in [0,\varrho_{\cri}]}|N(\varrho)| \leq  C,
\end{equation}
for some positive constant $C$ depending only on $\varrho_{\cri}$.
Observe from \eqref{8.17a} that
\begin{equation}\label{eq:N rho infty is null}
    N(\varrho_\infty) =0.
\end{equation}

\smallskip
\noindent 2. \textit{Special entropy pair and dissipation estimate}:
A direct computation via the Loewner--Morawetz relations and \eqref{eq:Hnu special gen}
shows that the special generator $H^*$ generates the entropy pair $\Q_* = (Q_{1*},Q_{2*})$:
\begin{align}
Q_{1*}(\nu,\vartheta) &= -q(\vartheta\sin \vartheta + \cos \vartheta)
   + N(\varrho) \varrho q \cos \vartheta, \nonumber \\
Q_{2*}(\nu,\vartheta) &= q(\vartheta\cos \vartheta - \sin \vartheta)
+ N(\varrho) \varrho q \sin \vartheta. \label{eq:special entropy explicit}
\end{align}
We then consider sequence $\{\Q^\varepsilon_*\}_{\varepsilon>0}$
for $\Q_*^\varepsilon = \Q_*(\bu^\varepsilon)$ and $\bu^\varepsilon$ uniquely
characterized by $(\nu^\varepsilon,\vartheta^\varepsilon)$, and observe that
the uniform boundedness of $N(\varrho^\varepsilon)$
and $\vartheta^\varepsilon$ (\textit{cf.}~Proposition \ref{prop:invariant regions in angle})
implies that $\{\Q^\varepsilon_*\}_{\varepsilon>0}$ is uniformly bounded
in $L^\infty(\mathscr{D})$.
Furthermore, using relation \eqref{eq:ent dissipation from polar sys with H}
computed in \S \ref{sec:entropic struc intro}
(\textit{cf.}~\cite[\S 7]{gq-transonic}), we have
\begin{align}
\div_\x \Q^\varepsilon_*
=& \,  -\varepsilon \vartheta^\varepsilon \Delta \vartheta^\varepsilon
+ \varepsilon N(\varrho^\varepsilon)
\div\big((1-\frac{c(\varrho^\varepsilon)^2}{q(\varrho^\varepsilon)^2})\nabla \varrho^\varepsilon\big)\nonumber \\
= & \, \varepsilon \div \big( -\vartheta^\varepsilon \nabla \vartheta^\varepsilon
   + (1-\frac{c(\varrho^\varepsilon)^2}{q(\varrho^\varepsilon)^2})
   N(\varrho^\varepsilon) \nabla\varrho^\varepsilon \big) \nonumber \\
&\, +\varepsilon \big( |\nabla \vartheta^\varepsilon|^2
  + (1-\frac{c(\varrho^\varepsilon)^2}{q(\varrho^\varepsilon)^2})\frac{1}{q(\varrho^\varepsilon)^2}
  |\nabla \varrho^\varepsilon|^2 \big). \label{eq:dissipation special entropy}
\end{align}

Integrating the entropy dissipation measures over $\mathscr{D}$
and applying the divergence theorem in the plane, and
recalling that $\mathbf{n}$ is the inward pointing unit normal, we obtain
\begin{equation*}
\begin{aligned}
\int_{\partial \mathscr{D}} \Q_*^\varepsilon \!\cdot \!\mathbf{n} \d \mathcal{H}(\x)
=& \,  \varepsilon\int_{\partial \mathscr{D}}
 \Big(- \vartheta^\varepsilon \nabla \vartheta^\varepsilon
  + \big(1-\frac{c(\varrho^\varepsilon)^2}{q(\varrho^\varepsilon)^2}\big)
    N(\varrho^\varepsilon) \nabla \varrho^\varepsilon \Big) \! \cdot \!\mathbf{n} \d \mathcal{H}(\x) \\
&- \varepsilon \int_{\mathscr{D}} \Big( |\nabla \vartheta^\varepsilon|^2
 + \big(1-\frac{c(\varrho^\varepsilon)^2}{q(\varrho^\varepsilon)^2}\big)
 \frac{1}{q(\varrho^\varepsilon)^2}|\nabla \varrho^\varepsilon|^2 \Big) \d \x \\
=: & \,  \, N^\varepsilon_1 + N^\varepsilon_2.
\end{aligned}
\end{equation*}
The uniform boundedness in $L^\infty(\mathscr{D})$ of sequence $\{\Q_*^\varepsilon\}_{\varepsilon>0}$
implies that the left-hand side bounded independently of $\varepsilon$, even though the estimate
depends on the size of $\mathcal{H}(\partial\mathscr{D})$.
Using the boundary conditions, we obtain
\begin{equation*}
\begin{aligned}
    N^\varepsilon_1 = & \, \int_{\partial \mathscr{D}_1} |\varrho^\varepsilon \, \bu^\varepsilon \!\cdot\! \mathbf{n}
    | N(\varrho^\varepsilon) \d \mathcal{H}(\x) \\
    &+ \varepsilon \int_{\partial\mathscr{D}_2}
    \Big( \underbrace{-\vartheta_\infty \nabla \vartheta^\varepsilon
    + \big(1-\frac{c(\varrho_\infty)^2}{q(\varrho_\infty)^2}\big)
      N(\varrho_\infty) \nabla \varrho^\varepsilon}_{=0}\Big) \! \cdot \! \mathbf{n} \d \mathcal{H}(\x),
\end{aligned}
\end{equation*}
where we have used $\vartheta_\infty = N(\varrho_\infty)=0$ from \eqref{eq:N rho infty is null}.
Hence, $N^\varepsilon_1$ is uniformly bounded in $\varepsilon$
due to the uniform bound on $|\bu^\varepsilon|$ and the estimate \eqref{eq:N(rho) est} in Step 1.
We thereby conclude that the second integral $N^\varepsilon_2$ is bounded, independently of $\varepsilon$.
This completes the proof.
\end{proof}

The next lemma is essential to the proof of Proposition \ref{prop:H-1 cpct for viscous solns using our kernels}.
In what follows,
$(\varrho^\varepsilon,\vartheta^\varepsilon) \in C^2(\mathscr{D})\cap C^1(\overline{\mathscr{D}})$
is a solution of the auxiliary system \eqref{eq:variant system s6}
satisfying the boundary conditions of Proposition \ref{prop:upper bound density}.

\begin{lemma}[Compactness Estimates]\label{lem:lm cpct est}
Let $H$ be a smooth entropy generator satisfying that
there exists $C>0$ independent of $\varepsilon$ such that, for $j=0,1,2$,
\begin{equation*}
   \begin{aligned}
       &  \left| \partial_\vartheta^j \big( \varrho(\nu^\varepsilon) H_{\nu }(\nu^\varepsilon,\vartheta^\varepsilon) \big) \big|
       + \big| \partial^j_\vartheta \big(H_{\vartheta}(\nu^\varepsilon,\vartheta^\varepsilon)\big) \right| \leq C, \\
       & \left| \partial^j_\vartheta \big( \varrho(\nu^\varepsilon) H_{\nu}(\nu^\varepsilon,\vartheta^\varepsilon)
        + H_{\vartheta\vartheta}(\nu^\varepsilon,\vartheta^\varepsilon) \big) \right| \leq C\varrho(\nu^\varepsilon).
       \end{aligned}
\end{equation*}
Let $\Q$ the corresponding entropy pair generated by $H$ via the Loewner--Morawetz relations.
Then the sequence of entropy dissipation measures $\{\div_\x \Q^\varepsilon\}_{\varepsilon>0}$
is pre-compact in $H^{-1}(\mathscr{D})$.
\end{lemma}

\begin{proof}
We omit the $\varepsilon$-superscript in the following computations to simplify notation.
Recalling relation \eqref{eq:ent dissipation from polar sys with H}
from \S \ref{sec:entropic struc intro},
we obtain
\begin{align}
\div_\x \Q^\varepsilon
=&\, \varepsilon\div \Big(\nabla \vartheta( \varrho H_{\nu \vartheta} - H_\vartheta)
 + (1 - \frac{c(\varrho)^2}{q(\varrho)^2}) \nabla \varrho ( H_\nu + \frac{1}{\varrho} H_{\vartheta\vartheta})
 \Big) \nonumber \\
&-\varepsilon \Big( \nabla \vartheta \! \cdot \! \nabla( \varrho H_{\nu \vartheta} - H_\vartheta)
 + (1 - \frac{c(\varrho)^2}{q(\varrho)^2}) \nabla \varrho \! \cdot \! \nabla
 ( H_\nu + \frac{1}{\varrho} H_{\vartheta\vartheta})\Big) \nonumber \\
=: & \, D_1^\varepsilon + D_2^\varepsilon. \label{eq:entropy dissipation measure general expression}
\end{align}

\smallskip
\noindent 1. \textit{Control of $D^\varepsilon_2$}: We focus first on $D^\varepsilon_2$ as it is the more difficult term to bound. Observe that
\begin{equation*}
\begin{aligned}
-\varepsilon^{-1}D_2^\varepsilon
=& \,  \Big(2( 1-\frac{c^2}{q^2}) H_{\nu \vartheta}
+ \frac{c^2}{q^2}  \varrho  H_{\nu \nu \vartheta}
+ ( 1-\frac{c^2}{q^2}) \frac{1}{\varrho} H_{\vartheta\vartheta\vartheta} \Big)
\nabla \vartheta \! \cdot \! \nabla \varrho \\
& + \big(\varrho H_{\nu \vartheta \vartheta} -  H_{\vartheta\vartheta} \big)
|\nabla \vartheta|^2
+ \big(\varrho^2 H_{\nu \nu}  + \varrho H_{\nu \vartheta\vartheta}-\frac{q^2}{c^2} H_{\vartheta\vartheta}\big)
  \big( 1-\frac{c^2}{q^2}\big) \frac{1}{q^2} |\nabla \varrho|^2,
\end{aligned}
\end{equation*}
which, using the entropy equation to rewrite $\partial^j_\vartheta H_{\nu\nu}
= \frac{M^2-1}{\varrho^2}\partial^j_\vartheta H_{\vartheta\vartheta}$ for $j=0,1$, simplifies to
\begin{align}
-\varepsilon^{-1} D_2^\varepsilon
= & \, \big( \varrho H_{\nu \vartheta\vartheta} - H_{\vartheta\vartheta}\big)
\big( |\nabla \vartheta|^2 + ( 1-\frac{c^2}{q^2})\frac{1}{q^2} |\nabla \varrho|^2 \big)\nonumber \\
&+  \big(H_{\vartheta\vartheta\vartheta}+\varrho H_{\nu \vartheta}\big)\frac{2}{\varrho}
\big( 1-\frac{c^2}{q^2}\big) \nabla \vartheta \! \cdot \! \nabla \varrho.
\label{eq:more general entropies}
\end{align}
By the dissipation estimate of Proposition \ref{prop:dissipation estimate},
we know that the first term on the right-hand side is controlled, provided that
there exists $C>0$ independent of $\varepsilon$ such that
\begin{equation*}
    \big| \varrho(\nu^\varepsilon) H_{\nu \vartheta\vartheta}(\nu^\varepsilon,\vartheta^\varepsilon)
    - H_{\vartheta\vartheta}(\nu^\varepsilon,\vartheta^\varepsilon) \big| \leq C,
\end{equation*}
which is guaranteed by the assumptions.
The second term on the right-hand side is controlled similarly. Indeed,
\begin{equation*}
\begin{aligned}
&\Big| \big( H_{\vartheta\vartheta\vartheta}+\varrho H_{\nu \vartheta} \big)
\frac{2}{\varrho} \big( 1-\frac{c^2}{q^2} \big) \nabla \vartheta \! \cdot \! \nabla \varrho \Big| \\
&=\frac{2q}{\varrho}  \sqrt{ 1-\frac{c^2}{q^2} }\,
\big| H_{\vartheta\vartheta\vartheta}+\varrho H_{\nu \vartheta} \big|
\Big| \nabla \vartheta \! \cdot \! \nabla \varrho \frac{1}{q}\sqrt{  1-\frac{c^2}{q^2} }\, \Big| \\
&\leq C \varrho^{-1} \big|H_{\vartheta\vartheta\vartheta}+\varrho H_{\nu \vartheta} \big|
\Big( |\nabla \vartheta|^2 +  \big( 1-\frac{c^2}{q^2}\big ) \frac{1}{q^2}|\nabla \varrho|^2 \Big)
\end{aligned}
\end{equation*}
for some constant $C>0$ depending on $\varrho_*$ but independent of $\varepsilon$,
where we have used the Cauchy--Young inequality and the boundedness of $1-\frac{c^2}{q^2}$.
Then our assumptions imply that
\begin{equation*}
    \left| H_{\vartheta\vartheta\vartheta}(\nu^\varepsilon,\vartheta^\varepsilon)
    +\varrho(\nu^\varepsilon) H_{\nu \vartheta}(\nu^\varepsilon,\vartheta^\varepsilon) \right| \leq C\varrho(\nu^\varepsilon),
\end{equation*}
whence the dissipation estimate of Proposition \ref{prop:dissipation estimate} imply that there exists $C>0$,
independent of $\varepsilon$, such that
\begin{equation}\label{eq:D2eps est}
   \sup_\varepsilon \Vert D^\varepsilon_2 \Vert_{L^1(\mathscr{D})} \leq C,
\end{equation}
and hence sequence $\{D^\varepsilon_2\}_{\varepsilon>0}$ is
pre-compact in $W^{-1,r}(\mathscr{D})$ for all $r \in (1,2)$ by the Rellich Theorem,
as $\partial\mathscr{D}$ is Lipschitz.

\smallskip
\noindent 2. \textit{Control of $D^\varepsilon_1$}: The term inside the divergence of $D^\varepsilon_1$ is bounded as follows:
\begin{align}
&\Big| \nabla\vartheta  \big( \varrho H_{\nu \vartheta} - H_\vartheta \big)
  + \big(1-\frac{c^2}{q^2}\big)  \nabla \varrho \big( H_\nu + \frac{1}{\varrho} H_{\vartheta\vartheta} \big) \Big|  \nonumber \\
&\leq | \varrho H_{\nu \vartheta} - H_\vartheta | |\nabla \vartheta|
  + \big(1-\frac{c^2}{q^2}\big) |\varrho H_\nu + H_{\vartheta\vartheta}| \varrho^{-1}|\nabla \varrho|  \nonumber \\
& \leq |\nabla \vartheta| + \big(1-\frac{c^2}{q^2}\big)  |\nabla \varrho|. \label{eq:control of D1 part i}
\end{align}
Furthermore,
using the boundedness of $q$ and $1-\frac{c^2}{q^2}$,
we have
\begin{equation*}
    1-\frac{c^2}{q^2}\leq Cq^{-1}(1-\frac{c^2}{q^2}) \leq Cq^{-1}\sqrt{1-\frac{c^2}{q^2}},
\end{equation*}
whence it follows from
Proposition \ref{prop:dissipation estimate} and the Minkowski inequality that
\begin{align}
&\varepsilon\Big\Vert \nabla\vartheta  \big( \varrho H_{\nu \vartheta} - H_\vartheta \big)
+ \big(1-\frac{c^2}{q^2}\big)  \nabla \varrho \big( H_\nu + \frac{1}{\varrho} H_{\vartheta\vartheta} \big)  \Big\Vert_{L^2(\mathscr{D})} \nonumber \\
&\leq C \sqrt{\varepsilon} \bigg( \varepsilon \int_\mathscr{D}
\Big( |\nabla \vartheta|^2 + \big(1-\frac{c^2}{q^2}\big)\frac{1}{q^2}|\nabla \varrho|^2 \Big)
\d \x \bigg)^{\frac{1}{2}} \nonumber \\
&\leq C\sqrt{\varepsilon}. \label{eq:control of D1 part ii}
\end{align}
We deduce that there exists $C>0$ independent of $\varepsilon$ such that,
for all $\varepsilon>0$,
\begin{equation}\label{eq:D1eps est}
    \Vert D^\varepsilon_1\Vert_{H^{-1}(\mathscr{D})} \leq C\sqrt{\varepsilon}.
\end{equation}

\smallskip
\noindent 3. \textit{Interpolation compactness}:
Observe from the Loewner--Morawetz relations that
sequence $\{\Q^\varepsilon\}_{\varepsilon>0}$ is uniformly bounded in $L^\infty(\mathscr{D})$.
It follows that there exists $C>0$ independent of $\varepsilon$ such that
\begin{equation}\label{eq:Qeps unif bound cpctness}
    \sup_\varepsilon \Vert D^1_\varepsilon + D^2_\varepsilon \Vert_{W^{-1,\infty}(\mathscr{D})} \leq C.
\end{equation}
Estimates \eqref{eq:D2eps est}--\eqref{eq:Qeps unif bound cpctness} and the interpolation compactness
lemma of Ding-Chen-Luo \cite[\S 4]{dingchenluo1} (\textit{cf.}~Murat's Lemma \cite{muratcone})
imply that sequence $\{ \div_\x \Q^\varepsilon\}_{\varepsilon>0}$ is contained in a compact subset of $H^{-1}(\mathscr{D})$.
\end{proof}

The second estimate of Lemma \ref{lem:lm cpct est} highlights a key cancellation property,
which can be observed formally from the initial conditions of the kernels
recorded in Theorem \ref{thm:kernels}.

Combining the results of Lemma \ref{lem:lm cpct est} with
Propositions \ref{prop:est for h-1 cpct reg} and \ref{prop:cpctness est sing},
we have proved Proposition \ref{prop:H-1 cpct for viscous solns using our kernels}
and hence also Theorem \ref{big-thm-viscous}.

\section{Existence of Global Entropy Solutions}\label{sec:entropy sol}

In this section, we first combine the compensated compactness
framework (Theorem \ref{thm:compensated compactness framework intro})
with the approximate solutions $\mathbf{u}^\varepsilon$ constructed and estimated in
\S\ref{sec:viscous} to show the existence of entropy solutions
on the bounded domains $\mathscr{D}$ as
introduced
in \S \ref{sec:viscous}. Then we employ
the compactness theorem (Theorem \ref{big-thm-3})
to establish Theorem \ref{thm:main existence theorem}
by considering an exhaustion of the whole domain $\Omega$ in the half-plane $\mathbb{H}$.

\begin{proof}[Proof of Theorem \ref{thm:main existence theorem}]
For each $k \in \mathbb{N}$,
let $\mathscr{D}^k$ be a domain of identical shape
to $\mathscr{D}$ (\textit{cf.}~\S \ref{sec:viscous}) with the same boundary data
and same obstacle (both in shape and size)
of increasing size with $k$ such that $\{\mathscr{D}^k\}_{k\in \mathbb{N}}$ is an increasing compact exhaustion
of the half-plane $\mathbb{H}$
with the shaded obstacle denoted by $\obstacle$.
Then $\Omega=\bigcup_k \mathscr{D}^k$, and $\mathscr{D}^k \subset \mathscr{D}^{k+1}$ for all $k$.
We divide the proof into two steps.

\medskip
1. For any fixed domain $\mathscr{D}^k$, we first observe that
the approximate solutions $\{\bu^\varepsilon_k\}_{\varepsilon>0}$ generated
by Theorem \ref{big-thm-viscous} in $\mathscr{D}^k$
satisfy the assumptions of the compensated compactness
framework established in Theorem \ref{thm:compensated compactness framework intro}
in $\mathscr{D}^k$.
It therefore follows that, up to a subsequence that is not relabeled,
there exists $\bu_k \in L^\infty(\mathscr{D}^k)$ such that
\begin{equation*}
    \bu^\varepsilon_k \to \bu_k \qquad \text{a.e.~and strongly in } L^p \text{ for all } p \in [1,\infty).
\end{equation*}
Now we show that the limit function $\bu_k \in L^\infty(\mathscr{D})$ is an entropy solution defined
in $\mathscr{D}^k$.
For simplicity of notation, we use
$\varrho^\varepsilon_k = \varrho(\bu^\varepsilon_k)$ and $\theta^\varepsilon_k = \theta(\bu^\varepsilon_k)$
for the entropy pairs in this step below.

\smallskip
\textit{We first show that $\bu_k$ is a distributional solution in $\mathscr{D}^k$}.
Let $\psi \in C^\infty_c(\mathscr{D}^k)$.
Inserting this function in the weak formulation of the viscous problem \eqref{eq:sys Tristan}
and using the boundary conditions, we have
\begin{equation*}
          \begin{aligned}
               &\int_{\mathscr{D}^k} (\bu^\varepsilon_k)^\perp \! \cdot \! \nabla \psi \d \x
               = -\varepsilon \int_{\mathscr{D}^k}
                \nabla \psi \! \cdot \! \nabla \vartheta^\varepsilon_k \d \x, \\
       & \int_{\mathscr{D}^k} \varrho^\varepsilon \, \bu^\varepsilon_k \! \cdot \! \nabla \psi \d \x
       = \varepsilon \int_{\mathscr{D}^k}
        \big(1-\frac{c(\varrho^\varepsilon_k)^2}{q(\varrho^\varepsilon_k)^2}\big)
       \nabla \psi \! \cdot \! \nabla \varrho^\varepsilon_k \d \x,
    \end{aligned}
\end{equation*}
where we have used $\curl \bu^\varepsilon_k = \div (\bu^\varepsilon_k)^\perp$.
Using the dissipation estimate of Proposition \ref{prop:H-1 cpct for viscous solns using our kernels},
the strictly positive lower bound of $q(\varrho^\varepsilon_k) \geq q(\varrho_*)$
in Proposition \ref{prop:upper bound density}, and the dissipation estimate
in Proposition \ref{prop:dissipation estimate}, we obtain
\begin{align}
               &\Big|\int_{\mathscr{D}^k} (\bu^\varepsilon_k)^\perp \! \cdot \! \nabla  \psi \d \x\Big|
               + \Big| \int_{\mathscr{D}^k}
                 \varrho^\varepsilon_k\, \bu^\varepsilon_k \! \cdot \! \nabla \psi \d \x\Big| \nonumber \\
               &\leq C_\psi \sqrt{\varepsilon} \bigg(\varepsilon \int_{\mathscr{D}^k}
                \Big( |\nabla \vartheta^\varepsilon_k|^2
                +  \big(1-\frac{c(\varrho^\varepsilon_k)^2}{q(\varrho^\varepsilon_k)^2}\big)
                \frac{1}{q(\varrho^\varepsilon_k)^2}|\nabla \varrho^\varepsilon_k|^2 \Big)
                 \d \x \bigg)^{\frac{1}{2}} \nonumber \\
&\leq C_\psi \sqrt{\varepsilon}, \label{eq:final convergence weak form}
\end{align}
and pass to the limit as $\varepsilon\to0$ so that the right-hand side vanishes.
We pass to the limit in the left-hand terms
by applying the Dominated Convergence Theorem, and using
that $\bu^\varepsilon_k \to \bu_k$ almost everywhere and the uniform $L^\infty$ bound.
Hence, the limit function $\bu_k$ satisfies the first condition of Definition \ref{def:ent sol}.

\smallskip
\noindent
\textit{We now show that $\bu_k$ satisfies the entropy inequality}. For the special entropy pair $\Q_*$,
recalling the dissipation equality \eqref{eq:dissipation special entropy}
and testing against any nonnegative function $\psi \in C^\infty_c(\mathscr{D})$, we have
\begin{align}
\int_{\mathscr{D}^k} \Q_*(\bu^\varepsilon_k) \! \cdot \! \nabla \psi \d \x
= & \, \varepsilon \int_{\mathscr{D}^k}
\nabla \psi \cdot \Big( -\vartheta^\varepsilon_k \nabla \vartheta^\varepsilon_k
+ \big(1-\frac{c(\varrho^\varepsilon_k)^2}{q(\varrho^\varepsilon_k)^2}\big)
   N(\varrho^\varepsilon_k) \nabla\varrho^\varepsilon_k \Big) \d \x \nonumber \\
&-\varepsilon \underbrace{\int_{\mathscr{D}^k} \psi
\Big( |\nabla \vartheta^\varepsilon_k|^2
  + \big(1-\frac{c(\varrho^\varepsilon_k)^2}{q(\varrho^\varepsilon_k)^2}\big)
\frac{1}{q(\varrho^\varepsilon_k)^2}|\nabla \varrho^\varepsilon_k|^2 \Big) \d \x}_{\geq 0} \nonumber \\
\leq & \, C_\psi \sqrt{\varepsilon} \bigg(\varepsilon \int_{\mathscr{D}^k}
\Big( |\nabla \vartheta^\varepsilon_k|^2
+  \big(1-\frac{c(\varrho^\varepsilon_k)^2}{q(\varrho^\varepsilon_k)^2}\big)
\frac{1}{q(\varrho^\varepsilon_k)^2}
 |\nabla \varrho^\varepsilon_k|^2 \Big) \d \x \bigg)^{\frac{1}{2}} \nonumber \\
\leq & \, C_\psi \sqrt{\varepsilon}, \label{eq:Qstar computation}
\end{align}
where we have argued as we did in \eqref{eq:final convergence weak form}.
The uniform $L^\infty$ bound of $\bu^\varepsilon_k$ transfers to $\Q_*(\bu^\varepsilon_k)$
via formulas \eqref{eq:special entropy explicit},
whence
the Dominated Convergence Theorem implies that the left-hand side converges
to $\int_{\mathscr{D}^k} \Q_* \!\cdot \nabla \psi \d \x$, which implies
that
$\div_\x \Q_* \geq 0$ in the sense of distributions.

Similarly, let $\Q_H$ be an entropy pair generated by
kernels $H^\r$ and $H^\s$ (\textit{cf.}~Remark \ref{rmk:generation}).
Assume also that $H$ satisfies the convexity conditions in \eqref{eq:generator conditions ent ineq}.
The estimates in Propositions \ref{prop:est for h-1 cpct reg} and \ref{prop:cpctness est sing}
imply that there exists a positive constant $C_H$ independent of $(\nu,\vartheta)$ such that,
for any $j=0,1,2$,
\begin{equation*}
|\partial_\vartheta^j ( \varrho H_\nu + H_{\vartheta \vartheta}) | \leq C_H \varrho,
\qquad |\partial_\vartheta^j (\varrho H_\nu)| + |\partial_\vartheta^j H_\vartheta| \leq C_H,
\end{equation*}
so that $H$ satisfies the assumptions in Lemma \ref{lem:lm cpct est}.
Meanwhile, the convexity conditions in \eqref{eq:generator conditions ent ineq}:
 \begin{equation*}
     \varrho H_{\nu \vartheta \vartheta} - H_{\vartheta\vartheta} \leq 0,
     \qquad |H_{\vartheta\vartheta\vartheta}+\varrho H_{\nu \vartheta}|
     \leq - \varrho (\varrho H_{\nu \vartheta\vartheta} - H_{\vartheta\vartheta}),
 \end{equation*}
yield that  the sign condition $D^\varepsilon_2 \geq 0$ in \eqref{eq:more general entropies}
by using computations \eqref{eq:entropy dissipation measure general expression}--\eqref{eq:more general entropies}
and Young's inequality. Thus, we have
  \begin{align*}
     &\int_{\mathscr{D}^k} \psi \div_\x \Q_H(\bu^\varepsilon_k) \d \x\\
     &\geq
      - \varepsilon\int_{\mathscr{D}^k}
      \nabla \psi \cdot \Big( \nabla \vartheta^\varepsilon_k
       ( \varrho^\varepsilon_k H_{\nu \vartheta} - H_\vartheta )
      + \big(1 - \frac{c(\varrho^\varepsilon_k)^2}{q(\varrho^\varepsilon_k)^2} \big)
      \nabla \varrho^\varepsilon_k ( H_\nu + \frac{1}{\varrho} H_{\vartheta\vartheta} ) \Big) \d \x
  \end{align*}
  for any nonnegative function $\psi \in C^\infty_c(\mathscr{D})$.
  Integrating by parts on the left-hand side, following
  the computations in \eqref{eq:control of D1 part i}--\eqref{eq:control of D1 part ii},
  and arguing as per \eqref{eq:Qstar computation}, we obtain
    \begin{equation*}
      \int_{\mathscr{D}^k} \Q_H(\bu^\varepsilon_k) \! \cdot \! \nabla \psi \d \x \leq C_{\psi,H} \sqrt{\varepsilon}.
  \end{equation*}
Note from the Loewner--Morawetz relations \eqref{eq:lowener mor} that the integrand on the left-hand side
is uniformly bounded in $L^\infty$.
Applying the Dominated Convergence Theorem, we pass to the limit as $\varepsilon\to0$
and deduce the entropy inequality: $\div_\x \Q_H(\bu_k) \geq 0$ in the sense of distributions.
Thus, the limit function  $\bu_k$ is an entropy solution as defined in Definition \ref{def:ent sol}.

\smallskip
\noindent
\textit{We now verify that $\bu_k$ satisfies the weak normal trace condition on
the obstacle boundary $\partial\mathscr{D}_1^k$}.
To this end, we insert a nonnegative test function $\varphi \in C^\infty(\overline{\mathscr{D}})$
with $\supp\varphi \cap \partial\mathscr{D}_2 = \emptyset$ into the weak formulation of the approximate problem
to obtain
\begin{equation}\label{eq:weak trace verify}
\begin{aligned}
\int_{\partial\mathscr{D}_1^k} \varphi \varrho^\varepsilon_k \bu^\varepsilon_k \!\cdot \!\mathbf{n}
 \d \mathcal{H}(\x)
&= \int_{\mathscr{D}^k} \varrho^\varepsilon_k \bu^\varepsilon_k\! \cdot \!\nabla \varphi \d \x
   + \int_{\partial\mathscr{D}_1^k} \varphi |\varrho^\varepsilon_k \bu^\varepsilon_k \!\cdot \!\mathbf{n}| \d \mathcal{H}(\x) \\
&\quad\, - \varepsilon \int_{\mathscr{D}^k}
 \big(1-\frac{c(\varrho^\varepsilon_k)^2}{q(\varrho^\varepsilon_k)^2}\big)
 \nabla \varphi \! \cdot \! \nabla \varrho^\varepsilon_k \d \x.
    \end{aligned}
\end{equation}
Note that the final term on the right-hand side of the previous equality may be controlled as in \eqref{eq:final convergence weak form}, whence
\begin{equation*}
    \varepsilon \Big|\int_{\mathscr{D}^k} \big(1-\frac{c(\varrho^\varepsilon_k)^2}{q(\varrho^\varepsilon_k)^2}\big)
    \nabla \varphi \! \cdot \! \nabla \varrho^\varepsilon_k \d \x\bigg| \leq C_\varphi \sqrt{\varepsilon}
\end{equation*}
vanishes in the limit as $\varepsilon \to 0$.
Thus, by using the convergence: $\bu^\varepsilon_k \to \bu_k$ {\it a.e.}~and \eqref{eq:weak trace verify},
the Dominated Convergence Theorem implies
\begin{equation*}
   \begin{aligned}
       \int_{\mathscr{D}^k} \varrho_k \bu_k \! \cdot\! \nabla \varphi \d \x
       &= \lim_{\varepsilon\to0}\int_{\mathscr{D}_k}
       \varrho^\varepsilon_k \bu^\varepsilon_k\! \cdot \!\nabla \varphi \d \x \\
       &\leq \liminf_{\varepsilon \to 0} \int_{\partial\mathscr{D}_1^k}
        \varphi \big( \underbrace{\varrho^\varepsilon_k \bu^\varepsilon_k\!\cdot\!\mathbf{n}
        - |\varrho^\varepsilon_k \bu^\varepsilon_k \!\cdot\! \mathbf{n}| }_{\leq 0}\big) \d \mathcal{H}(\x) \\
       &\leq 0.
   \end{aligned}
\end{equation*}
The weak normal trace condition follows immediately from the above
by using the non-negativity of $\varrho$ and the Dubois--Reymond Lemma.

\medskip
2. In this setup, we prove Theorem \ref{thm:main existence theorem} regarding the existence of
a globally-defined entropy solution on the entire plane
with the incoming supersonic flow at $x=-\infty$,
as a corollary of our weak continuity result.

First, in Step 1,
we have obtained the entropy solutions $\bu_k \in L^\infty(\mathscr{D}^k)$
satisfying Definition \ref{def:ent sol} on each $\mathscr{D}^k$:
For each $\psi \in C^\infty_c(\mathscr{D}^k)$,
\begin{equation}\label{eq:weak form for global}
    \int_{\mathbb{H}} \bu_k^\perp \!\cdot \!\nabla \psi \d \x = 0,
    \qquad \int_{\mathbb{H}} \varrho(\bu_k)\bu_k \! \cdot\! \nabla \psi \d \x = 0,
\end{equation}
and, for every nonnegative function $\varphi \in C^\infty_c(\mathscr{D}^k)$, the entropy inequality holds:
\begin{equation}\label{eq:ent ineq for global}
    \int_{\mathbb{H}} \Q(\bu_k) \! \cdot \! \nabla \varphi \d \x \leq 0,
\end{equation}
where $\Q \in \{\Q_*,\Q_H\}$ with $\Q_H$ being an entropy pair generated via $H$ generated by
kernels $H^\r$ and $H^\s$ and satisfying the convexity conditions in \eqref{eq:generator conditions ent ineq}.
Furthermore, since the invariant regions in Theorem \ref{big-thm-viscous} depend only
on the boundary data (which is the same for all $k$), then, for all $k \in \mathbb{N}$:
\begin{equation*}
    0 \leq \varrho(\bu_k(\x)) \leq \varrho_*, \quad q_{\cri} < q_* \leq |\bu_k(\x)| \leq q_{\cav}
    \qquad\,\, \text{for a.e.~}\x \in \mathscr{D}^k,
\end{equation*}
where we emphasize that $\varrho_*$ and $q_*$ are independent of $k$.

\smallskip
By extending functions $\bu_k$ by $q_*$ outside of $\mathscr{D}^k$,
then sequence $\{\bu_k\}_k$ is uniformly bounded such that the previous bounds hold almost everywhere
in $\mathbb{H}$.
This implies that there exists a subsequence (still denoted by) $\{\bu_k\}_k$
such that $\bu_k \overset{*}{\rightharpoonup} \bu$ in $L^\infty(\Omega)$.

\smallskip
Fix any $l \in \mathbb{N}$. The previous weak-* convergence implies that
$\bu_k \overset{*}{\rightharpoonup} \bu$ in $L^\infty(\mathscr{D}^l)$.
Theorem \ref{big-thm-3} implies that there exists a further subsequence (still denoted by)
$\{\bu_k\}_{k\in \mathbb{N}}$
such that $\bu_k \to \bu$ almost everywhere and strongly
in $L^p(\mathscr{D}^l)$ for all $p \in [1,\infty)$.
This mode of convergence suffices
to pass to the limit as $k\to\infty$ in \eqref{eq:weak form for global}--\eqref{eq:ent ineq for global},
where we emphasize that the test functions $\psi$ and $\varphi$ belong to $C^\infty_c(\mathscr{D}^l)$ for $l$ fixed.

\smallskip
We have therefore shown that, for any $l \in \mathbb{N}$,
$\psi \in C^\infty_c(\mathscr{D}_l)$, and nonnegative function
$\varphi \in C^\infty_c(\mathscr{D}^l)$,
relations \eqref{eq:weak form for global}--\eqref{eq:ent ineq for global}
are verified by $\bu$.
Since $\{\mathscr{D}^l\}_{l\in\mathbb{N}}$ is an increasing compact exhaustion of the half-plane $\mathbb{H}$,
we deduce that \eqref{eq:weak form for global}--\eqref{eq:ent ineq for global}
are verified by $\bu$ for any test function
$\psi \in C^\infty_c(\Omega)$ and any nonnegative test function $\varphi \in C^\infty_c(\Omega)$.
In addition, the weak normal trace condition follows similarly.

This completes the proof.
\end{proof}

\appendix
\section{Fundamental Solutions of the Tricomi--Keldysh Equation}\label{appendix:Glambda}

In this appendix, we present some properties of distribution $G_\lambda$ that is the fundamental solution
of the Tricomi--Keldysh equation (\textit{cf.}~\cite{gelfandshilov}). We also extend the classical definition
and properties of $G_\lambda$ to the case when $\lambda$ is a negative integer;
this is needed for the construction of the kernels in \S \ref{sec:reg kernel}--\S \ref{sec:sing kernel}.
Throughout this appendix, we use $\mathscr{F}$ to
denote the Fourier transform with respect to variable $s$.

\subsection{Classical definition of $G_\lambda$ and extension as a distribution}

Recall from the classical definition of $G_\lambda$ given in \eqref{eq:Glambda first def} that
\begin{equation*}
  G_0(\nu,s) = \mathds{1}_{[-k(\nu),k(\nu)]}(s), \quad G_1(\nu,s) = [k(\nu)^2 - s^2]_+, \quad  G_2(\nu,s) = [k(\nu)^2-s^2]_+^2.
\end{equation*}
An elementary computation yields
\begin{equation}\label{eq:FT G0 and G1}
    \begin{aligned}
      \hat{G}_0(\nu,\xi) = & \, \frac{2 \sin (\xi k(\nu))}{\xi}
        = \,  \sqrt{\pi} 2^{\frac{1}{2}} k(\nu)^\frac{1}{2} |\xi|^{-\frac{1}{2}}J_{\frac{1}{2}}(|\xi| k(\nu)),   \\
        \hat{G}_1(\nu,\xi) = & \, \frac{4 k(\nu)}{\xi^2} \big( \frac{\sin(\xi k(\nu))}{\xi k(\nu)} - \cos (\xi k(\nu)) \big)
          = \,  \sqrt{\pi} 2^{\frac{3}{2}} k(\nu)^{\frac{3}{2}}|\xi|^{-\frac{3}{2}} J_{\frac{3}{2}}(|\xi|k(\nu)), \\
       \hat{G}_2(\nu,\xi) = & \, \frac{16 k(\nu)^2}{\xi^3}\big( \frac{3 \sin(\xi k(\nu))}{\xi^2 k(\nu)^2}
         - \frac{3 \cos(\xi k(\nu))}{\xi k(\nu)} - \sin(\xi k(\nu)) \big)  \\
    =& \,  \,  2!\,2^{\frac{5}{2}}\sqrt{\pi}k(\nu)^{\frac{5}{2}} |\xi|^{-\frac{5}{2}}J_{\frac{5}{2}}(|\xi|k(\nu)),
    \end{aligned}
\end{equation}
where $\hat{G}_n(\nu,\xi)$ denotes the Fourier transform of $G_n$ with respect to
its second variable for $n=0,1,2$.
More generally (\textit{cf.}~\cite{gelfandshilov}), for $\lambda \in (-1,\infty)$,
\begin{equation}\label{eq:FT Glambda lambda bigger than -1}
    \hat{G}_\lambda(\nu,\xi) = c_{\lambda } k(\nu)^{\lambda+ \frac{1}{2}} |\xi|^{-\lambda -\frac{1}{2}} J_{\lambda +\frac{1}{2}}(|\xi|k(\nu))
\end{equation}
with $c_{\lambda} := \sqrt{\pi} 2^{\lambda+\frac{1}{2}} \Gamma(\lambda+1)$.

Note that all the functions on the right-hand sides
of \eqref{eq:FT G0 and G1}--\eqref{eq:FT Glambda lambda bigger than -1} are smooth
in variable $\xi$ and grow at most polynomially; that is,
they belong to space $\mathcal{S}'(\mathbb{R})$ of tempered distributions.
It is therefore natural to consider the tempered distribution
defined by
$\mathscr{F}^{-1}\{\sqrt{\pi}2^{-\frac{1}{2}}k(\nu)^{-\frac{1}{2}}|\xi|^{\frac{1}{2}}J_{-\frac{1}{2}}(|\xi|k(\nu))\}$
as a candidate for $G_{-1}$.
In fact, due to the recurrence relations for the Bessel functions of the first kind
and their asymptotic behavior, the Paley--Wiener--Schwartz Theorem implies that
this distribution is compactly supported.
Furthermore, recall the recurrence relation:
\begin{equation}
    \partial_z \hat{G}_\lambda(\nu,z) = -\frac{2\lambda+1}{z}\hat{G}_\lambda(\nu,z)
      + \frac{2\lambda}{z}k(\nu)^2 \hat{G}_{\lambda-1}(\nu,z) \qquad \mbox{for any $z \in \mathbb{C}$},
\end{equation}
which holds whenever $\lambda > 0$. We now define $G_{-2}$ from $G_{-1}$ by using the above:

\begin{definition}\label{def:defining G negative integer}
Fix $\nu>0$. We define
\begin{equation}\label{eq:G-1 def}
\begin{aligned}
G_{-1}(\nu,\cdot) :=& \,  \mathscr{F}^{-1}\{ \sqrt{\pi}2^{-\frac{1}{2}}k(\nu)^{-\frac{1}{2}}
  |\xi|^{\frac{1}{2}}J_{-\frac{1}{2}}(|\xi|k(\nu))\}, \\
G_{-2}(\nu,\cdot) :=& \,  -\frac{1}{2}k(\nu)^{-2}\mathscr{F}^{-1} \{\xi \partial_\xi \hat{G}_{-1}(\nu,\xi)
   - \hat{G}_{-1}(\nu,\xi) \}, \\
G_{-3}(\nu,\cdot) :=& \,  -\frac{1}{4}k(\nu)^{-2}\mathscr{F}^{-1} \{\xi \partial_\xi \hat{G}_{-2}(\nu,\xi)
-3 \hat{G}_{-2}(\nu,\xi) \}
\end{aligned}
\end{equation}
in the sense of $\mathcal{S}'(\mathbb{R})$.
\end{definition}

\begin{lemma}\label{lem:FT for G negative}
Fix $\nu>0$. The following identities hold:
\begin{equation}
\begin{aligned}
\hat{G}_{-1}(\nu,\xi) = & \, c_{-1} k(\nu)^{-\frac{1}{2}} |\xi|^{\frac{1}{2}}J_{-\frac{1}{2}}(|\xi| k(\nu))
= \frac{\cos(\xi k(\nu))}{k(\nu)}, \\
\hat{G}_{-2}(\nu,\xi) = & \, c_{-2} k(\nu)^{-\frac{3}{2}} |\xi|^{\frac{3}{2}}J_{-\frac{3}{2}}(|\xi| k(\nu))
=  \frac{1}{2}k(\nu)^{-3} \big( \xi k(\nu) \sin(\xi k(\nu)) + \cos(\xi k(\nu)) \big), \\
\hat{G}_{-3}(\nu,\xi) = & \, c_{-3} k(\nu)^{-\frac{5}{2}} |\xi|^{\frac{5}{2}}J_{-\frac{5}{2}}(|\xi| k(\nu)) \\
=& \,  \frac{1}{8}k(\nu)^{-5}\Big( 3 \cos (\xi k(\nu)) + 3 \xi k(\nu) \sin (\xi k(\nu))
- (\xi k(\nu))^2 \cos (\xi k(\nu)) \Big),
\end{aligned}
\end{equation}
where $c_{-n} = \frac{(-1)^{n-1}}{(n-1)!}\sqrt{\pi}2^{\frac{1}{2}-n}$ for $n=1,2,3$,
and the recurrence relation:
\begin{equation}\label{eq:jumping from -n-1 up to -n}
    \xi\partial_\xi \hat{G}_{-n}(\nu,\xi) = (2 n-1)\hat{G}_{-n}(\nu,\xi) - 2n k(\nu)^2 \hat{G}_{-n-1}(\nu,\xi),
\end{equation}
is understood in the sense of $\mathcal{S}'(\mathbb{R})$ for $n=1,2$.
Furthermore,  distributions $G_{-n}(\nu,\cdot)$, $n=1,2,3$, are in
space $\mathcal{E}'(\mathbb{R})$ of compactly supported distributions.
\end{lemma}

\begin{proof}
The first Fourier transform relation for $G_{-1}$ holds in view of the exact representation for
the Bessel function $J_{-\frac{1}{2}}$.
The final two relations hold by direct computation via
Definition \ref{def:defining G negative integer}
and the trigonometric representation for $\hat{G}_{-1}$,
followed by using the exact representations for the Bessel functions $J_{-\frac{3}{2}}$
and $J_{-\frac{5}{2}}$.
The recurrence relation is direct from Definition \ref{def:defining G negative integer}
and the fact that $G_{-n}$ is a compactly supported distribution is deduced
from the Fourier transform relations and the Paley--Wiener--Schwartz Theorem.
\end{proof}

\subsection{Properties of distribution $f_\lambda$}
\begin{definition}\label{def:little f def}
For $a \in \mathbb{R}\setminus \{0\}$ and $\zeta \in \mathcal{D}'(\mathbb{R})$,
define its rescaling $\zeta(\cdot a)$ by
\begin{equation*}
\langle \zeta(\cdot a),\varphi \rangle := a^{-1}\langle \zeta ,
\varphi(\cdot a^{-1}) \rangle \qquad\,\, \text{for all } \varphi \in \mathcal{D}(\mathbb{R}).
\end{equation*}
For $\nu>0$ and $n=-3,\dots,0,1,2$,
we define $f_{n} \in \mathcal{E}'(\mathbb{R})$ to be the $\nu$-independent rescalings:
\begin{equation*}
    f_{n} := k(\nu)^{-2n} G_{n}(\nu, \cdot k(\nu)),
\end{equation*}
which holds in the sense of distributions.
\end{definition}

Then we have

\begin{lemma}\label{lem:G hat in terms of f hat and f hat precise}
For $\nu>0$ and $n=-3,-2,-1,0,1,2$,
\begin{equation*}
    \hat{G}_{n}(\nu,\xi) = k(\nu)^{1+2n} \hat{f}_{n}(\xi k(\nu)) \qquad \text{in } \mathcal{S}'(\mathbb{R}),
\end{equation*}
and, for $n=1,2,3$,
\begin{equation*}
    \hat{f}_{-n}(\xi) = c_{-n}|\xi|^{n-\frac{1}{2}}J_{ -(n-\frac{1}{2})}(|\xi|)
    \qquad\,\mbox{for all $\xi \in \mathbb{C}$}.
\end{equation*}
In turn, for all $\xi \in \mathbb{R}$,
\begin{equation*}
 \begin{aligned}
 &\hat{f}_0(\xi) = \frac{2\sin \xi}{\xi}, \quad \hat{f}_{-1}(\xi) = \cos \xi, \\
 &\hat{f}_{-2}(\xi) = \frac{1}{2}\big( \xi \sin \xi + \cos \xi \big),
        \quad \hat{f}_{-3}(\xi) = \frac{1}{8}\big( 3 \cos \xi + 3 \xi \sin \xi - \xi^2 \cos \xi \big),\\
&\hat{f}_1(\xi) = \frac{4}{\xi^2}\big(\frac{\sin \xi}{\xi} - \cos \xi\big),
\quad
\hat{f}_2(\xi) = \frac{16 }{\xi^3}\big( \frac{3 \sin(\xi )}{\xi^2 } - \frac{3 \cos(\xi )}{\xi } - \sin(\xi) \big),
\end{aligned}
\end{equation*}
and $\hat{f}_{-2},\hat{f}_{-3} \in L^\infty_\loc(\mathbb{R})$,
while $\hat{f}_\lambda \in L^\infty(\mathbb{R})$ for all $\lambda \geq -1$.
\end{lemma}
These assertions follow directly from Definition \ref{def:little f def} and Lemma \ref{lem:FT for G negative}.

\smallskip
We recall the following standard result.
\begin{lemma}\label{lemma:flambda relations}
Let $l>-1$. Then
\begin{equation}
    \hat{f}_{l}''(z) + \hat{f}_{l}(z) = \hat{f}_{l+1}(z) = -\frac{2(l+1)}{z}\hat{f}_{l}'(z)
    \qquad\,\,\mbox{for all $z \in \mathbb{C}\setminus\{0\}$}.
\end{equation}
Furthermore, if $l>0$, then
\begin{equation}
    \hat{f}'_{l}(z) = - \frac{2l+1}{z}\hat{f}_{l}(z) + \frac{2l}{z}\hat{f}_{l-1}(z)
    \qquad \mbox{for all $z \in \mathbb{C}$}.
\end{equation}
\end{lemma}

We have the following recurrence formulas for $\hat{f}_{-n}(\xi)$,
which mirror those of the previous lemma.

\begin{lemma}\label{lem:f hat relations}
For all $\xi \in \mathbb{C}\setminus \{0\}$, the following relations hold{\rm :}
\begin{equation*}
    \begin{aligned}
        &\hat{f}_{-1}''(\xi) + \hat{f}_{-1}(\xi) = 0, \\
        &\hat{f}_{-j}''(\xi) + \hat{f}_{-j}(\xi) = \hat{f}_{-j+1}(\xi) = \frac{2(j-1)}{\xi}\hat{f}_{-j}'(\xi)
        \qquad \text{for $j=0,2,3$},
    \end{aligned}
\end{equation*}
and
\begin{equation*}
    \begin{aligned}
    &\hat{f}_{-1}'(\xi) = -\frac{\xi}{2}\hat{f}_0(\xi),\\
    &\hat{f}_1'(\xi) = -\frac{3}{\xi}\hat{f}_1(\xi) + \frac{2}{\xi}\hat{f}_0(\xi), \\
            &\hat{f}_0'(\xi) = -\frac{1}{\xi}\hat{f}_0(\xi) + \frac{2}{\xi}\hat{f}_{-1}(\xi), \\
            &\hat{f}_{-j}'(\xi) = \frac{2j-1}{\xi}\hat{f}_{-j}(\xi) - \frac{2j}{\xi}\hat{f}_{-j-1}(\xi)
            \quad \text{for $j=1,2$}.
    \end{aligned}
\end{equation*}
\end{lemma}
These formulas follow from direct calculations.

In turn, we deduce the following recurrence relations for distributions $G_{-n}$.

\begin{lemma}\label{lem:G0 to G negatives}
For $n=1,2$, the following relations hold in the sense of distributions{\rm :}
\begin{equation}
    \begin{aligned}
       &\partial_s  G_{-n}(\nu,s ) = 2n s  G_{-n-1}(\nu,s ), \\
       &\partial_\nu G_{-n}(\nu,s ) = -2n k'(\nu)k(\nu) G_{-n-1}(\nu,s ), \\
       &(k(\nu)^2-s ^2)G_{-n-1}(\nu,s ) =  G_{-n}(\nu,s ),
    \end{aligned}
\end{equation}
while
$\partial_s G_n(\nu,s) = -2s G_{n-1}(\nu,s)$ for $n=0,1$, $\partial_\nu G_0(\nu,s )
 = 2k'(\nu)k(\nu)G_{-1}(\nu,s )$ and $(k(\nu)^2-s ^2)G_{-1}(\nu,s )=0$ in $\mathcal{S}'(\mathbb{R})$.
Finally, the following relations hold in the sense of distributions{\rm :}
\begin{equation}\label{eq:two derivs formulas}
    \begin{aligned}
        &\partial^2_s G_1(\nu,s) =  -6 G_0(\nu,s) + 4k(\nu)^2 G_{-1}(\nu,s), \\
        &\partial_s ^2 G_{0}(\nu,s ) = 2 G_{-1}(\nu,s ) - 4 k(\nu)^2 G_{-2}(\nu,s ), \\
        &\partial_s ^2 G_{-1}(\nu,s ) = - 6 G_{-2}(\nu,s ) + 8 k(\nu)^2 G_{-3}(\nu,s ),
    \end{aligned}
\end{equation}
so that $\supp G_{-n}(\nu,\cdot) \subset [-k(\nu),k(\nu)]$ for $n=1,2,3$, and $\nu>0$.
\end{lemma}
\begin{proof}
We use the fact that the Fourier transform $\mathscr{F}:\mathcal{S}'(\mathbb{R}) \to \mathcal{S}'(\mathbb{R})$
is an automorphism to compute the derivatives via the formulas:
\begin{equation*}
    \begin{aligned}
        \partial_s  G_{-n}(\nu,s ) &= \mathscr{F}^{-1}\{ \i \xi \hat{G}_{-n}(\nu,\xi)\} \}, \\
        \partial_\nu G_{-n}(\nu,s ) &= \mathscr{F}^{-1}\{ \partial_\nu \hat{G}_{-n}(\nu,\xi)\} \}.
    \end{aligned}
\end{equation*}
The first and second equalities follow then from Lemma \ref{lem:FT for G negative}
and the fact that $\mathscr{F}\{x \varphi(x)\} = \i \hat{\varphi}'(\xi)$
for all $\varphi \in \mathcal{S}(\mathbb{R})$.
Recall also that $\mathscr{F}\{x^2 \varphi(x)\} = -\hat{\varphi}''(\xi)$
for all $\varphi \in \mathcal{S}(\mathbb{R})$, whence
\begin{equation*}
    \begin{aligned}
        \mathscr{F}\{(k(\nu)^2 - s ^2) G_{-2}(\nu,s ) \} \! = \! k(\nu)^2 \hat{G}_{-2}(\nu,s ) + \partial_\xi^2 \hat{G}_{-2}(\nu,\xi)\! =
        \! \frac{\cos(\xi k(\nu))}{k(\nu)}
        =\hat{G}_{-1}(\nu,\xi).
    \end{aligned}
\end{equation*}
We deduce that $(k(\nu)^2-s ^2)G_{-2}(\nu,s ) = G_{-1}(\nu,s )$ in the sense of distributions,
since the Fourier transform is an automorphism of $\mathcal{S}'(\mathbb{R})$.
An identical strategy shows the results for $G_{-3}$ and $G_{-1}$,
which verifies the third and fourth equalities.
For the final part, we employ the recurrence relations proved earlier to obtain
\begin{equation*}
    \begin{aligned}
        \partial_s ^2 G_{-1}(\nu,s) &= 2 G_{-2}(\nu,s ) + 2 s  \partial_s  G_{-2}(\nu,s ) \\
        &= 2 G_{-2}(\nu,s ) - 8 (k(\nu)^2-s ^2) G_{-3}(\nu,s ) + 8 k(\nu)^2 G_{-3}(\nu,s ).
    \end{aligned}
\end{equation*}
Then the result follows by using the same strategy to obtain analogous relations
for $\partial^2_s G_0$ and $\partial^2_s G_1$, and $\partial_sG_1(\nu,s) = -2 s G_0(\nu,s)$.
The relation regarding the support of the distributions is clear from the formulas
in \eqref{eq:two derivs formulas} and
the fact that $\supp G_1(\nu,\cdot)=\supp G_0(\nu,\cdot) = [-k(\nu),k(\nu)]$.
\end{proof}

\vspace{0.5cm}

\noindent
\textbf{Acknowledgements.}
The research of Gui-Qiang G. Chen was supported in part by the UK Engineering and Physical
Sciences Research Council (EPSRC) Awards EP/L015811/1, EP/V008854, and EP/V051121/1.
The research of Tristan P. Giron and Simon M. Schulz was supported in part
by the EPSRC Grant EP/ L015811/1.
Simon M. Schulz also acknowledges support from Centro di Ricerca Matematica Ennio De Giorgi.

\vspace{0.3cm}

\noindent \textbf{Statements and Declarations.} The authors declare no conflicts of interest.

\end{document}